\documentclass[reqno]{amsart}
\usepackage[utf8]{inputenc}
\usepackage[table]{xcolor}
\usepackage{mathtools,tikz-cd,amsmath,amssymb,graphicx}
\usepackage{enumerate}
\usepackage[parfill]{parskip}
\usepackage{pgf}
\usepackage{pdfpages}
\usepackage{url}

\usepackage{caption}
\usepackage{subcaption}

\usepackage{tikz,tikz-3dplot}

\usepackage{amsthm,amssymb,latexsym,epic,bbm,comment,mathrsfs}
\usepackage[all,2cell]{xy}
\xyoption{2cell}
\allowdisplaybreaks

\newtheorem{theorem}{Theorem}[section]
\newtheorem{conjecture}[theorem]{Conjecture}

\newtheorem{question}[theorem]{Question}

\newtheorem{lemma}[theorem]{Lemma}
\newtheorem{corollary}[theorem]{Corollary}

\newtheorem{proposition}[theorem]{Proposition}
\newtheorem{defn}[theorem]{Definition}

\newtheorem{thmx}{Theorem}
\newtheorem{questx}[thmx]{Question}
\newtheorem{conjx}[thmx]{Conjecture}
\newtheorem{defx}[thmx]{Definition}
\theoremstyle{definition}%below are displayed non-italic
\newtheorem{example}[theorem]{Example}

\newtheorem{remark}[theorem]{Remark}

\newtheorem{obs}[theorem]{Observation}

\newcommand{\JM}{\mathbf{JM}}

\newcommand{\Ext}{\operatorname{Ext}}
\newcommand{\soc}{\operatorname{soc}}
\newcommand{\surj}{\rightarrow\mathrel{\mkern-14mu}\rightarrow}

\newcommand{\inj}{\hookrightarrow}

\newcommand{\hd}{\operatorname{hd}}

\newcommand{\inv}{^{-1}}

\newcommand{\kl}{\underline H}

\newcommand{\jc}{\mathcal J}
\newcommand{\lc}{\mathcal L}
\newcommand{\rc}{\mathcal R}
\newcommand{\hc}{\mathcal H}
\newcommand{\cO}{\mathcal O}

\newcommand{\co}{\Delta_e/\Delta}
\newcommand{\WD}{W(D_{n+2})}
\newcommand{\WB}{W(B_{n+1})}

\newcommand{\so}{\operatorname{soc}\Delta_e/\Delta}

\newcommand{\maxd}{\operatorname{max.deg. }}
\newcommand{\mind}{\operatorname{min.deg. }}

\newcommand{\hk}[1]{{\color{teal}HK: #1}}

\newcommand{\desc}[2]{\ensuremath{ {}_{#1}W_{#2} }}

\newcommand{\BG}[2]{\ensuremath{ {}_{#1}\mathbf{BG}_{#2} }}

\newcommand{\JI}[2]{\ensuremath{ {}_{#1}\mathbf{JI}_{#2} }}

\newcommand{\Ou}{\ensuremath{ 0^+ }}
\newcommand{\Od}{\ensuremath{ 0^- }}

\newcommand{\skal}{\operatorname{skal}}
\newcommand{\skbl}{\operatorname{skbl}}

\usetikzlibrary{shapes.misc}
\tikzset{cross/.style={cross out, draw=black, minimum size=2*(#1-\pgflinewidth), inner sep=0pt, outer sep=0pt},
cross/.default={2pt}}

\begin{document}

\title[Join operation for the Bruhat order and Verma modules]
{Join operation for the Bruhat order and Verma modules}

\author[H.~Ko, V.~Mazorchuk and R.~Mr{\dj}en]
{Hankyung Ko, Volodymyr Mazorchuk and Rafael Mr{\dj}en}

\begin{abstract}
We observe that the join operation for the Bruhat order on a Weyl group agrees with the intersections of Verma modules in type $A$. The statement is not true in other types, and we propose a conjectural statement of a weaker correspondence. Namely, we introduce distinguished subsets of the Weyl group on which the join operation conjecturally agrees with the intersections of Verma modules.
We also relate our conjecture with a statement about the socles of the 
cokernels of inclusions between Verma modules. The latter determines the first Ext spaces between a simple module and a Verma module.
  We give a conjectural complete description of such socles, which we verify in a number of cases. 
Along the way, we determine the poset structure of the join-irreducible elements in Weyl groups and obtain closed formulae for certain families of Kazhdan-Lusztig polynomials.
\end{abstract}

\maketitle

\section{Introduction}

Let $\mathfrak{g}$ be a finite dimensional complex semisimple Lie algebra
with a fixed triangular decomposition 
\[\mathfrak{g} = \mathfrak{n}^+\oplus\mathfrak{h}\oplus\mathfrak{n}^-.\]
Let $W$ be the associated Weyl group, which we view as 
a Coxeter system $(W,S)$. Consider the associated BGG category $\cO$
and its principal block $\cO_0$ (the indecomposable 
summand of $\cO$ containing the trivial  $\mathfrak{g}$-module).  
The Verma modules in $\cO_0$ are indexed by the Weyl group elements. For $w\in W$, we denote by $\Delta_w$ the Verma module of heighst weight $w.0 = w\rho - \rho$, where $\rho$ is the half of the sum of all positive roots. 
The set $W$ has a poset structure with respect to the Bruhat order $\leq$.
By a result of Bernstein, Gelfand and Gelfand, see \cite[Chapter~7]{Di},
for $w,x\in W$, the following assertions are equivalent:
\begin{itemize}
\item there is a (necessarily unique up to scalar and injective)
non-zero homomorphism from $\Delta_w$ to $\Delta_x$.
\item $w\geq x$.
\end{itemize}
This allows us to unambiguously write $\Delta_w\subseteq \Delta_x$
provided that $w\geq x$. In particular, we can view each $\Delta_w$
as a canonical submodule of the dominant Verma module $\Delta_e$. 
In other words, we have an isomorphism of posets as follows:
\begin{equation}\label{posetiso}
    (W,\leq) \cong (\{\Delta_w\subseteq \Delta_e\ |\ w\in W\},\supseteq).
\end{equation}

Being a poset, $W$ has the join operation 
\[ \vee:  2^W \to W\coprod \{\text{`does not exist'}\},\]
where $2^W$ denotes the power set of $W$.
The main purpose of the current paper is to 
understand the module-theoretic interpretation of 
this structure on the right hand side of \eqref{posetiso}. 
The obvious candidate is the intersection:
\begin{questx}\label{intque}
Given $U\subseteq W$, do we have $w=\bigvee U$ if and only if $\displaystyle\bigcap_{x\in U}\Delta_x = \Delta_w$?
\end{questx}
Using the results of \cite{kmm2}, we observe that
the answer to Question~\ref{intque} is always positive in type $A$ 
(see Corollary~\ref{typeAconj}). A similar observation is made in \cite{Ko}.

\begin{thmx}\label{intthm:A}
Suppose $\mathfrak{g} = \mathfrak{sl}_n$. Then for any $U\subseteq W$, we have $\bigvee U =w\in W$ if and only if $\displaystyle\bigcap_{x\in U}\Delta_x = \Delta_w$.
\end{thmx}

In general, the answer to   Question~\ref{intque}
can be negative (see Example~\ref{e644bigex} in type $E_6$).
In this paper, we explore a number of other special cases and examples
which suggest that full understanding of Question~\ref{intque} 
is very likely to be highly non-trivial. 
Our main idea for relating the join operation on $W$ and the intersection of Verma submodules is to restrict $\vee$ from $2^W$ to certain subsets consisting of 
join-irreducible elements, which we introduce now.

An element $w\in W$ is called \emph{join-irreducible} if there is no `join expression'
\[w = \bigvee U\]
with $w\not\in U$.
We denote by $\JI{}{}$ the set of join-irreducible elements in $W$.
Then every element $w\in W$ has a distinguished join expression, as the join of join-irreducibles, given by the subset $\mathbf{JM}(w)\subset \JI{}{}$ defined as follows.
First, recall that $s\in S$ is a left (resp., right) descent of $x\in W$ if $sx < x$ (resp., $xs < x$). 
Consider the set
\[\JI{}{}(w) = \{x \in \mathbf{JI}\ |\   x \leq w \text{ and each left (resp., right) decent of $x$ is a left (resp., right) descent of $w$}\}.\]
Now we let
\[\JM(w) = \max\JI{}{}(w) = \text{the set of Bruhat maximal elements in \JI{}{}(w)} .\]
Then $\JM(w)$ is, in fact, a join expression of $w$, that is, we have (see Lemma~\ref{lemma:joinJM})
\[w = \bigvee \JM(w).\]
What we conjecture to hold, in all types, is that each Verma module has an 
expression given by $\JM(w)$. 

\begin{conjx}\label{conjintro}
For each $w\in W$, we have
\begin{equation*}
\Delta_w = \bigcap_{x\in \mathbf{JM}(w)}\Delta_x.    
\end{equation*}
\end{conjx}

When $W$ is dissective (see Subsection~~\ref{ss:posets}), Conjecture~\ref{conjintro} answers Question~\ref{intque} completely and positively (see Corollary~\ref{coraddedforintro}), and, in fact, this is how we prove Theorem~\ref{intthm:A} in type $A$.

In our proof of Conjecture~\ref{conjintro} in type $A$, the following property plays a crucial role and provides a connection to the results in \cite{kmm2}.
Recall that, by definition and \eqref{posetiso}, for each $x\in \mathbf{JM}(w)$, we have 
a canonical map
\[\psi_x:\Delta_e/\Delta_w\surj \Delta_e/\Delta_x.\]

\begin{defx}\label{sosum def}\label{defintro}
An element $w\in W$ is said to have the \emph{socle-sum property} if 
\begin{enumerate}
    \item\label{1} each $\psi_x$ restricts to 
\[\phi_x:\operatorname{soc}\Delta_e/\Delta_w\to \operatorname{soc}\Delta_e/\Delta_x;\]
\item\label{3}  we have $\displaystyle\bigcap_{x\in \mathbf{JM}(w)} \ker \phi_x=0$;
\item\label{2} the maps $\phi_x$, where $x\in \mathbf{JM}(w)$, are surjective.
\end{enumerate}
\end{defx}

The properties \eqref{1} and \eqref{2} make precise the statement that `$\so_w$ contains the sum of $\so_x$ taken over all $x\in \JM(w)$ inside $\Delta_e$'. 
The properties \eqref{1} and \eqref{3} together say that `$\so_w$ is contained in the sum of $\so_x$'.
Often, and always in type $A$, this sum is direct. In such a case, the socle-sum property is equivalent to a simpler condition
\[\so_w \cong \bigoplus_{x\in\JM(w)} \so_x  .\]
The main result of \cite{kmm2} can be interpreted as the socle-sum property in type $A$.
%graded!!
The socle-sum property does not hold in general as we show in Remark~\ref{soclesumfk} in type $B$
and in Example~\ref{F4example} in type $F$. 
However, one direction of the socle-sum property is true in general.

\begin{thmx}\label{intthm:sosum}
For each $w\in W$, the properties \eqref{1} and \eqref{3} hold. In other words, 
the socle of $\co_w$ is contained in the sum of $\so_x$ taken over all $x\in \JM(w)$. 
\end{thmx}

Motivated by the socle-sum property, we determine the socles of $\co_x$, 
for $x\in \JI{}{}$, in almost all cases (we know the answer but do not have
complete proofs in the remaining cases). We also establish the socle-sum property 
for many special cases and examples. 
In order to obtain a description of the socles, we explicitly give a set of generating relations for the poset of join-irreducible elements with fixed descent sets, for each $W$.
There we encounter several interesting examples which illustrate the complexity of the combinatorics and representation theory outside type $A$. In exceptional types, 
this relies heavily on computer assisted computations.

Determining $\so_w$, for $w\in W$, is an important problem independently 
of Question~\ref{intque} or Conjecture~\ref{conjintro}. 
As pointed out in \cite{kmm2}, it has interesting applications to 
understanding various homological invariants of category $\cO$.
In particular, knowing $\so_w$ would completely determine the dimension of
the extension spaces
\[ \Ext^1(L_x,\Delta_w) \cong \Ext^1(\nabla_w,L_x).\]

We use the notation
\[_s \JM_t (w)= \{z\in \JM(w)\ |\ \text{ the left (resp., right) descent of $z$ is $s$ (resp., $t$)}\}, \]
for $w\in W$ and  $s,t\in S$. 
Note that we have 
$\displaystyle\JM(w) = \coprod_{s,t\in S} \ _s \JM_t (w)$ by Lemma~\ref{JIisbig}.
Let $w_0\in W$ be the longest element in $(W,S)$ and let $\jc$ be the 
two-sided Kazhdan-Lusztig cell containing the elements $w_0 S$ and consider the decomposition 
\[\jc = \coprod_{s,t\in S}\ {}^s\hc^t = \coprod_{s,t\in S}\ \{u\in \jc\ |\ \text{ the left (resp., right) descent of $z$ is $S\setminus s$ (resp., $S\setminus t$)}\}\] 
into the Kazhdan-Lusztig $\mathtt{H}$-cells (see Subsection~ \ref{ss:J Jc}).
Then Theorem~\ref{intthm:sosum}, together with Propositions~\ref{prop3}
and \ref{prop6} and 
\cite[Theorem~32]{Ma}, 
gives the following statement.

\begin{thmx}\label{intthm:ext}
Let $w\in W$. 
\begin{enumerate}[(a)]
    \item If $x\not\in \jc\coprod\{w_0\}$, then $\Ext^1(L_x,\Delta_w) = 0 $.
    \item If $x=w_0$, then $\dim\Ext^1(L_{w_0},\Delta_w)$ equals the rank of the minimal parabolic subgroup containing $w_0w$.
    \item\label{c} If $x\in \! {}^s\hc^t$, then
    \begin{equation}\label{eqc}
      \dim\Ext^1(L_x,\Delta_w) \leq [ \so_b : L_x]  \leq |_s\JM_t |
    \end{equation}
        where $b = \bigvee \!_s\JM_t$. In particular, $\Ext^1(L_x,\Delta_w)=0$ if $_s\JM_t =\emptyset$.
\end{enumerate}
\end{thmx}

For information on Theorem~\ref{intthm:ext}\eqref{c} depending on the type of $(W,S)$,
see Subsection~ \ref{s:JIcase}. 
The number $[ \so_b : L_x]$ is bounded
\begin{itemize}
\item by $1$ in type $AB$ (see \cite[Corollary 2]{kmm2} for type $A$ where both inequalities in \eqref{c} are, in fact, the equalities),
\item by $2$ in type $DF$,
\item $3$, $4$, and $6$ in types $E_6$, $E_7$ and $E_8$, respectively.
\end{itemize}
Each of these bounds is achieved by $\dim\Ext^1(L_x,\Delta_w)$, for some $x,w\in W$. 
The number  $|_s\JM_t |$ can be larger, e.g., up to $21$ in type $E_8$, 
although we include it to give  a purely combinatorial bound.

The paper is organized as follows: the next section contains all
necessary preliminaries. Section~\ref{s:join_Bruhat} studies
the join operation for Bruhat orders and then 
Section~\ref{s-new4} discusses the module-theoretic interpretation of the
join operation for Verma modules. 
Section~\ref{s:KL} contains a case-by-case combinatorial analysis of 
the posets of join irreducible elements for all types and, in particular,
determines the socles of $\co_x$. Here we also provide figures
illustrating the poset structure of join irreducible elements and
closed formulae for some families of Kazhdan-Lusztig polynomials
in type $BD$. We finish the paper with  Appendix which provides tables
of some families of Kazhdan-Lusztig polynomials in types $E_7$ and $E_8$.

\section{Preliminaries}

\label{sec:preliminaries}

\subsection{Hecke algebra and Kazhdan-Lusztig basis}

Let $(W,S)$ be a finite Coxeter system. We denote by $\leq$
the Bruhat order on $W$, and 
by $\ell:W\to\mathbb{Z}$ the associated length function. As usual, the identity element in $W$ is denoted by $e$ and the longest one by $w_0$.

The Hecke algebra $H(W,S)$ associated to $(W,S)$ is the $\mathbb Z[v,v\inv]$-algebra 
generated by $H_s$, for $s\in S$, which satisfy the (Coxeter) braid relations and
the quadratic relation
\[(H_s+v)(H_s-v\inv)=0,\]
for all $s\in S$.
Given a reduced expression $w=st\cdots u$ of $w\in W$, we let $H_w=H_sH_t\cdots H_u$. 
The element $H_w$ is, in fact, independent of the choice of the reduced expression, 
and $\{H_w\}_{w\in W}$ is a ($\mathbb Z[v,v\inv]$-)basis of $H(W,S)$ called the {\em standard basis}.
Now consider the ($\mathbb Z$-algebra-)involution 
\[\overline{\phantom{A}} :H(W,S)\to H(W,S)\] 
uniquely determined by $\overline{v}= v\inv$ and $\overline{H_s} = H_s\inv$.
Then there is a unique element $\kl_w$ in $H(W,S)$ such that $\overline{\kl_w} = \kl_w$ and 
\begin{equation}
\label{p}
\kl_w = H_w + \sum_{y\in W} p_{y,w}H_y,
\end{equation}
for some $p_{y,w}\in v\mathbb Z[v]$.
The elements $\kl_w$, where $w\in W$, form a basis of $H(W,S)$ called the {\em Kazhdan-Lusztig (KL) basis}. Given $x,y\in W$, the coefficient of $v$ in $p_{x,y}+p_{y,x}$ is denoted by 
$\mu(x,y)=\mu(y,x)$, defining the {\em (Kazhdan-Lusztig) $\mu$-function}. If $s\in S$, then $\kl_s=H_s+v$,
and 
\begin{equation}
\label{sy}
\kl_s\kl_y = \begin{cases} \displaystyle (v+v\inv)\kl_{y} &\colon  sy<y, \\[.7em] \displaystyle
\kl_{sy}+\sum_{sx<x, x<y}\mu(x,y)\kl_x &\colon sy>y.
\end{cases}
\end{equation}
Another basic fact is that
%\begin{equation*}
%\label{longest}
$ p_{x,w_0}=v^{\ell(w_0)-\ell(x)}$.
%\end{equation*}
%
%We can write \begin{equation}\label{p'} p_{x,y} = v^{\ell(y)-\ell(x)}+ p_{x,y}'   \end{equation} where the degree of $p'_{x,y}$ is less than $\ell(y)-\ell(x)-1$ (strictly, by parity).
For more details about KL basis, we refer to \cite{KL,So3}. Note that we are using the normalization of the Hecke algebra as in \cite{So3}.

\subsection{The small and penultimate KL cells}\label{ss:J Jc}
We call an element $s\in S$ a \emph{left descent}  of $y$ if $sy<y$ and a \emph{left ascent} of $y$ is $sy>y$. \emph{Right descent} and \emph{right ascent} are defined similarly. 
We denote by $LD(y)$ and $RD(y)$, respectively, the sets of left and right descents of $y$, and by $LA(y)$ and $RA(y)$, respectively, the sets of left and right ascents of $y$.
An element $w \in W$ is called \emph{bigrassmanian} if it has a unique left and a unique right descent. 

We consider the Kazhdan-Lusztig left, right, and two-sided orders on $W$, denoted, respectively, by $\leq_{\mathtt{L}}$, $\leq_{\mathtt{R}}$, $\leq_{\mathtt{J}}$, and the associated equivalence relations $\sim_{\mathtt{L}}$, $\sim_{\mathtt{R}}$, and $\sim_{\mathtt{J}}$. The corresponding equivalence classes in $W$ are called left, right, and two-sided cells.
We have an explicit description of the two smallest two-sided cells with respect to the two-sided order. The first one is $\mathcal J_0=\{e\}$. The next one $\jc_1$, the so-called \emph{small two-sided cell}, it consists of all elements in $W$ having a unique reduced expression. 
The left and the right cells in $\jc_1$ are both indexed by $S$:  for each $s\in S$, the elements in $\jc_1$ with the right (resp., left) descent set $\{s\}$ form a left (resp., right) cell in $\jc_1$. 

The $\mu$-functions is easy on $\jc_1$, see %\cite[Proposition 3.8 (d)]{Lu} or %add to bibliography
\cite{KMMZ,KM}:

\begin{proposition}\label{smallJmu}
Let $x,y\in \mathcal J_1$. Then $\mu(x,y)\neq 0$ if and only if $x,y$ are adjacent in the Bruhat graph, in which case $\mu(x,y)=1$.
\end{proposition}

Multiplication by $w_0$ induces an involution on the set of two-sided cells, under which the small cell $\jc_1$ corresponds to the \emph{penultimate two-sided cell} denoted by $\jc$.
That is, $\jc=w_0\jc_1=\jc_1w_0$.

The left and the right cells in $\jc$ have unique ascents, so we index them by $S$. 
For example, $\lc^s$, for $s\in S$, is the left cell that has the right ascent set $\{s\}$ (and hence the right descent set $S\setminus\{s\}$). Similalry, ${}^t\rc$, for $t\in S$, is the 
right cell that has the left ascent set $\{t\}$.
We index the intersections of left and right cells (called $\mathtt{H}$-cells) by $S\times S$. An $\mathtt{H}$-cell in $\jc$ is thus written as ${}^t\hc^s=\lc^s\cap \ ^t\rc$, and $y\in {}^t\hc^s$ satisfies $RA(y)= \{s\}$ and $LA(y)=\{t\}$. In (Weyl) types $A$, $D$ and $E$, all $\mathtt{H}$-cells in $\jc$ are singletons. The latter property fails in types $B$, $F$ and $G$.

\begin{proposition}\label{muJ}
Let $x,y\in \jc$. Then $\mu(x,y)\neq 0$ if and only if $x,y$ are adjacent in the Bruhat graph, in which case $\mu(x,y)=1$.
\end{proposition}

\begin{proof}
By the Kazhdan-Lusztig inversion formula, see \cite[Section~3]{KL}, (or, equivalently, by Koszul duality,
see \cite{BGS}), we have $\mu(x,y)=\mu(w_0x\inv,w_0y\inv)$. Now, the claim follows from Proposition \ref{smallJmu}.
\end{proof}

\subsection{Kazhdan-Lusztig combinatorics, descent sets, and Verma modules}\label{ss:klprelim}

The statements in this subsection are proved in type A in \cite[Subsection 2.2]{kmm2}. The general case has exactly the same proof.

\begin{proposition}\label{prop3}
Let $x,y,z\in W$ be such that $x\geq y$ and $L_z$ is in the socle of
$\Delta_y/\Delta_x$. Then $z\in \mathcal{J}$. 
\end{proposition}

By Kazhdan-Lusztig conjecture (proved  in \cite{BB,BK,EW}), the 
multiplicities of the
graded composition factors in $\Delta_y$ isomorphic to each $L_z$, for $z\in\jc$, are determined by the Kazhdan-Lusztig polynomials as
\begin{equation}
\label{eq:p_Delta}
    \sum_{d} [\Delta_y:L_z\langle -d\rangle] v^d =p_{y,z},
\end{equation}
where $p_{y,z}$ is defined in (\ref{p}).
Since each Verma module $\Delta_y$ is canonically a submodule of $\Delta_e$, we restrict our attention to the case $y=e$. 

We denote by $\mathbf{a} \colon W \to \mathbb{Z}_{\geq 0}$ \emph{Lusztig's $\mathbf{a}$-function}, see \cite{Lu2}. The value $\mathbf{a}(w)$ does not change when $w$ varies over a two-sided cell $\mathcal{J}'$. We write $\mathbf{a}(\mathcal{J}')=\mathbf{a}(w)$, for $w\in \mathcal{J}'$. By definition, the value $\mathbf{a}(w)$, for $w \in \mathcal{J}'$,  describes the minimal possible  degree shift $d$, for which there exists $u\in \mathcal{J}'$ such that $\Delta_e$
has $L_u\langle -d\rangle$ as a simple subquotient.
The minimal shift is achieved exactly when $u$ is a \emph{Duflo element}
(called \emph{distinguished involution} in \cite[Section~1]{Lu2}). It follows that 
\begin{equation}\label{pdegrees}
    p_{e,w}\in \mathbb Z \{ v^a\ |\ \mathbf{a}(w)\leq a\leq \ell(w)\},
\end{equation} 
where the first equality holds if and only if $w$ is a Duflo element. Moreover, all exponents appearing in $p_{x,y}$ are of the same parity.
Every Duflo element is an involution, and hence belongs to a diagonal $\mathtt{H}$-cell.

\begin{proposition}\label{prop4}
Let $x\in W$ and $s\in S$.
\begin{enumerate}[$($i$)$]
\item \label{prop4.1} If $sx<x$, then the socle of $\Delta_e/\Delta_x$ contains
some $L_y$ such that $sy>y$.
\item \label{prop4.2} If $xs<x$, then the socle of $\Delta_e/\Delta_x$ contains
some $L_y$ such that $ys>y$.
\end{enumerate}
\end{proposition}

\begin{corollary}\label{cor5}
Let $w\in W$  be such that  the socle of $\Delta_e/\Delta_w$ is simple.
Then $w$ is bigrassmannian.
\end{corollary}

\begin{proposition}\label{prop6}
Assume that $x\in W$ and $y\in {}^s\mathcal{H}^t$, for $s,t\in S$, be such that 
$L_y$ is in the socle of $\Delta_e/\Delta_x$.  Then $s\in LD(x)$ and $t\in RD(x)$.
\end{proposition}

The following statement is implicit in \cite{kmm2}, but we provide a proof.

\begin{lemma}\label{socledie}
Let $x,y\in W$, with $x\leq y$, and suppose that the socle of $\co_x$ contains a graded component $X$ (of $\Delta_e$) isomorphic to $L_w\langle d\rangle$, for $w\in {}^s\mathcal{H}^t$.
If there exist $z \in W$ with $x\leq z\leq y$, such that $s\not\in LD(z)$ or  $t\not\in RD(z)$ (equivalently, if  $x \leq sy $ or $ x \leq yt$), then $\soc\co_y$ does not contain $X$. % isotypic component $L_w\langle d\rangle$.
\end{lemma}

\begin{proof}
Consider the quotients maps $\co_y\surj \co_z\surj \co_x$.
Since $s\not\in LD(z)$, by Proposition \ref{prop6}, the socle of $\co_z$ does not contain any $L_w\langle d\rangle$.
So $X$ is not in the socle of $\co_z$ and thus is not in the socle of $\co_y$ either.
\end{proof}

\subsection{Join-irreducible elements and dissectors}
\label{ss:posets}

Let $(P,\leq)$ be a partially ordered set, and let $\vee$ denote the join operation (i.e., supremum) in $P$. Of course, not every subset of $P$ needs to have a join. An element $x \in P$ is called \emph{join-irreducible} if $x \neq \bigvee Y$ for any $Y \subseteq P \setminus \{x\}$. Equivalently (see \cite[Lemme 2.3]{LS}), there exists $y \in P$ such that $x$ is minimal in $P \setminus \{z \in P \colon z \leq y\}$. Note that the smallest element $e$ in $P$ (if exists) is not join-irreducible, since then, by convention, we have $\bigvee \varnothing = e$.

Denote by $\mathbf{JI}$ the set of all join-irreducible elements in $P$. This set forms a \emph{base} of $P$ in the sense of \cite[Page~284]{GeK}. In more detail, denote by $\mathcal{P}(\mathbf{JI})$ the poset of all subsets of $\mathbf{JI}$ ordered by inclusion. The map $P \to \mathcal{P}(\mathbf{JI})$, $w \mapsto \{x \in \mathbf{JI} \colon x \leq w \}$ is an isomorphism of posets onto the image, and, moreover, any subset of $P$ with this property must contain $\mathbf{JI}$.

For $w \in P$, denote
\[ \mathbf{JM'}(w) := \text{ the set of maximal elements in } \{x \in \mathbf{JI} \colon x \leq w \} .\]

\begin{lemma}[{\cite[Page~284]{GeK}} or {\cite[Proposition 9]{Re}}]\label{rex}
For any $w\in P$, we have $w=\bigvee\mathbf{JM}'(w)$.
\end{lemma}

An element $x \in P$ is called a \emph{dissector} if $P$ can be written as a disjoint union $P=\{z \colon z \geq x\} \coprod \{z \colon z \leq y\}$, for some $y \in P$ (such $y$ is called a \emph{codissector} associated to $x$).

\begin{lemma}[{\cite[Proposition 12]{Re}}]
Every dissector is join-irreducible.
\end{lemma}

The poset $P$ is called \emph{dissective} (one also says that it admits \emph{``clivage''}, see \cite{GeK}) if every join-irreducible element is also a dissector.

\begin{proposition}
\label{prop:dissJM}
Suppose $P$ is dissective, and fix $w \in P$. Then $\mathbf{JM'}(w)$ is the unique minimal set of join-irreducible elements with the property that $w=\bigvee\mathbf{JM}'(w)$.
\end{proposition}
\begin{proof}
Assume $w=\bigvee Y$ for $Y \subseteq \mathbf{JI}$. We will prove that $\mathbf{JM}'(w) \subseteq Y$. Suppose $x \in \mathbf{JM}'(w) \setminus Y$, and let $y$ be the codissector associated to the dissector $x$. Since, for $z \in Y$, we cannot have $z \geq x$, we must have $z \leq y$. Also, since $w \geq x$, we cannot have $w \leq y$. This contradicts $w=\bigvee Y$.
\end{proof}

\section{Join operation on Bruhat orders}

\label{s:join_Bruhat}

In this section we collect and prove some properties of the join operation on a Coxeter group $(W,S)$,
%Even though we are primarily interested in Weyl groups, in this section we consider slightly more generally, a finite Coxeter group $(W,S)$, where $W$ is
partially ordered with respect to the strong Bruhat order $\leq$. 

\begin{lemma}[{\cite[Theorems 2.5, 3.4, 4.6, and Table I]{GeK}}]\label{JIisbig}
Every join-irreducible element in $W$ is bigrassmanian. For finite Coxeter groups, the converse is true if and only if $(W,S)$ is of type $A$, $B_3$ or of rank $2$.
\end{lemma}

\begin{lemma}[{\cite{LS}} and {\cite[Table I]{GeK}}]\label{lem:diss}
A finite Coxeter group $W$ is dissective if and only if it is of type $A$, $B$, $H$ or of rank $2$.
\end{lemma}

Given $w\in W$, we can find $x,\dots, y\in \mathbf{JI}$ such that $w=x\vee\cdots\vee y$. We call this a \emph{join expression} of $w$.
One way to find a join expression is to consider $\mathbf{JM'}(w)$. We can find a potentially smaller join expression by restricting the descent sets. More precisely, for $T,U \subseteq S$, denote
\[ \desc{T}{U} := \{x \in W \colon LD(x)\subseteq T,\ RD(x)\subseteq U\}.\]
This is a poset
with the order induced from the Bruhat order on $W$. Note that $\desc{T}{U}$ is %(by, for example, \cite[Subsection~ 2]{Ko2}) 
the set of Bruhat minimal representatives of the two sided cosets $W_{\tilde{T}} \backslash W / W_{\tilde{U}}$, where $W_{\tilde{T}}, W_{\tilde{U}}$ are the parabolic subgroups of $W$ generated by $\tilde{T}=S \setminus T$ and $\tilde{U}= S \setminus U$, respectively. We would  like to 
point out that our notation ${}_TW_U$ is 
different from, in particular, the notation   ${}^{\tilde{T}}W^{\tilde{U}}$ used in \cite{Re}. Also note that, for $s,t \in S$, the set $\desc{\{s\}}{\{t\}} \setminus \{e\}$ is exactly the set of all bigrassmannians with left descent $s$ and right descent $t$. We will simplify the notation in such a case and write
\[\BG{s}{t} = \desc{\{s\}}{\{t\}} \setminus \{e\}. \]
For fixed simple reflections $s,t\in S$, we let
\[ \JI{s}{t} := \JI{}{}\cap \BG{s}{t}, \]
the (po)set of all join-irreducible elements with fixed descents.

For $w \in W$, define
\[ \mathbf{JM}(w) := \text{ the set of Bruhat maximal elements in } \left\{x \in \mathbf{JI} \cap \left( \desc{LD(w)}{RD(w)} \right) \colon x \leq w \right\} .\]

\begin{lemma}
\label{lm:cosets}
Let $X \subseteq \desc{T}{U}$ and $w \in W$ such that $w$ is a minimal upper bound of $X$ in $W$. Then $w \in \desc{T}{U}$. 
Consequently, %the join of $X$ in the poset $\desc{T}{U}$ exists if and only if it exists in the poset $W$, and in this case they coincide.
the following diagram commutes:
%
\begin{comment}
%This version makes too much spacing afterwards.
\[\begin{tikzcd}[remember picture]
   \prod_i W \arrow{r}{\vee} & W\coprod \{\text{`does not exist'}\}\\
    \prod_i \desc{T}{U} \arrow{r}{\vee} & \desc{T}{U}\coprod \{\text{`does not exist'}\}\\
\end{tikzcd}
\begin{tikzpicture}[overlay,remember picture]
\path (\tikzcdmatrixname-1-1) to node[midway,sloped]{$\supset$}
(\tikzcdmatrixname-2-1);
\path (\tikzcdmatrixname-1-2) to node[midway,sloped]{$\supset$}
(\tikzcdmatrixname-2-2);
\end{tikzpicture}\]
\end{comment}
%
\[ \xymatrix{
 \displaystyle\prod_i W \ar[r]^-\vee & W\coprod \{\text{`does not exist'}\} \\
 \displaystyle\prod_i \desc{T}{U} \ar[r]^-\vee
 \ar@{}[u]|{\rotatebox[origin=c]{90}{$\subset$}} & \desc{T}{U}\coprod \{\text{`does not exist'}  \ar@{}[u]|{\rotatebox[origin=c]{90}{$\subset$}}   \}
} \]

\end{lemma}
\begin{proof}
Suppose $s \in LD(w)\setminus T$. Then $sw<w$ and $sx>x$, for all $x \in X$. From $w>x$ and the lifting property of Coxeter groups,
it follows that $sw>x$, for all $x \in X$, which is in contradiction with the minimality of $w$. This shows $LD(w)\subseteq T$. An analogous argument applies for the right descents of $w$.
\end{proof}

By Lemma \ref{lm:cosets}, the join operation for the poset $(W,\leq)$ can be computed in the subposets $(\BG{s}{t},\leq)$. It is worth noting that, according to \cite[Proposition 32]{Re} (see also \cite[Corollary 2.8]{GeK}), an element $x \in \desc{s}{t}$ is join-irreducible (resp. a dissector) in $W$ if and only if it is join-irreducible (resp. a dissector) in the poset $\desc{s}{t}$.

\begin{lemma}
\label{lemma:joinJM}
For any $w\in W$, we have $w=\bigvee\mathbf{JM}(w)$.
\end{lemma}
\begin{proof}
According to \cite[Proposition 27 and Proposition 31]{Re}, the join-irreducible elements in the poset $\desc{T}{U}$ (with $T = LD(w)$ and $U = RD(w)$) are exactly $\mathbf{JI} \cap \desc{T}{U}$. From Lemma \ref{rex} we have that $w=\bigvee\mathbf{JM}(w)$, where the join is taken inside the poset $\desc{T}{U}$. From Lemma \ref{lm:cosets} we see that the same equality also holds in the poset $W$.
\end{proof}

\begin{lemma}
\label{lemma:JM_subset}
Given $w\in W$, if $\mathbf{JM'}(w) \subseteq \mathbf{JM}(w)$ then $\mathbf{JM'}(w) = \mathbf{JM}(w)$.
\end{lemma}
\begin{proof}
Take $x \in \mathbf{JM}(w)$ and assume that $x \notin \mathbf{JM'}(w)$. Then there is $y \in \mathbf{JI}$ such that $x<y \leq w$. Without loss of generality we can take $y$ to be maximal with this property. This gives $y \in \mathbf{JM'}(w)$. But then $y \in \mathbf{JM}(w)$ %gives $y \in {}^{LD(w)} W^{RD(w)}$, 
which contradicts $x \in \mathbf{JM}(w)$. This proves $\mathbf{JM'}(w) = \mathbf{JM}(w)$.
\end{proof}
\begin{corollary}
\label{cor:JM=JM}
If $W$ is dissective%of type $A$, $B$, $H$ or of rank $2$
, then $\mathbf{JM'}(w)=\mathbf{JM}(w)$, for all $w \in W$.
\end{corollary}
\begin{proof}
This follows from %Lemma \ref{lem:diss}, 
Proposition \ref{prop:dissJM}, Lemma \ref{lemma:joinJM} and Lemma \ref{lemma:JM_subset}.

Alternatively, this also follows from the \emph{essential set} description of $\mathbf{JM'}(w)$ from \cite[Subsection~4.3]{kmm2} (originally introduced in \cite{Fu}) for type $A$, and \cite[Theorem 2.2]{anderson} for type $B$.
%
\begin{comment}
%OLD PROOF
For type $G_2$, one can check the claim directly (all elements except $e$ and $w_0$ are join-irreducible).
Assume $W$ is of type $A$, and take $x \in \mathbf{JM'}(w)$. From \cite[Corollary 18]{kmm2} (see also \cite{Ko}), the fact that join-irreducible elements in  type $A$ are exactly bigrassmannian permutations, and the definition of essential set (\cite[4.3]{kmm2}, originally introduced in \cite{Fu}) it follows easily that $x \in\desc{LD(w)}{RD(w)}$, i.e., $x \in \mathbf{JM}(w)$. This proves $\mathbf{JM'}(w) \subseteq \mathbf{JM}(w)$ in type $A$. Now Lemma \ref{lemma:JM_subset} gives the claim.
Assume $W$ is of type $B$, and take $x \in \mathbf{JM'}(w)$. Then one can analogously use \cite[Theorem 2.2]{anderson} and track the definitions, to find that $x \in \desc{LD(w)}{RD(w)}$. This proves $\mathbf{JM'}(w) \subseteq \mathbf{JM}(w)$ in type $B$, and again Lemma \ref{lemma:JM_subset} finishes the proof.
\end{comment}
\end{proof}

\begin{example}
Corollary \ref{cor:JM=JM} is not true in non-dissective Weyl groups. %types other than the ones stated. 
For example, in type $D_4$ with the following labeling of simple roots: \begin{tikzpicture}[scale=0.4,baseline=-3]
\protect\draw (0 cm,0) -- (-2 cm,0);
\protect\draw (-2 cm,0) -- (-4 cm,0.7 cm);
\protect\draw (-2 cm,0) -- (-4 cm,-0.7 cm);
\protect\draw[fill=white] (0 cm, 0 cm) circle (.15cm) node[above=1pt]{\scriptsize $2$};
\protect\draw[fill=white] (-2 cm, 0 cm) circle (.15cm) node[above=1pt]{\scriptsize $1$};
\protect\draw[fill=white] (-4 cm, 0.7 cm) circle (.15cm) node[above=1pt]{\scriptsize $\Ou$};
\protect\draw[fill=white] (-4 cm, -0.7 cm) circle (.15cm) node[above=1pt]{\scriptsize $\Od$};
\end{tikzpicture},
take $w=\Od{}1\Ou{}21\Od{}1\Ou{}$. Then we have $\mathbf{JM}(w) = \{\Od{}1\Ou{}21\Od{}, \, 1\Od{}21\Ou{}\}$, but $\mathbf{JM'}(w) = \mathbf{JM}(w) \cup \{1\Od{}\Ou{}1 \}$.

In type $F_4$, with the following labeling of simple roots:
\begin{tikzpicture}[scale=0.4,baseline=-3]
\protect\draw (0 cm,0) -- (-2 cm,0);
\protect\draw (-2 cm,0.1cm) -- (-4 cm,0.1 cm);
\protect\draw (-2 cm,-0.1cm) -- (-4 cm,-0.1 cm);
\protect\draw (-4 cm,0) -- (-6 cm,0);
\protect\draw[fill=white] (0 cm, 0 cm) circle (.15cm) node[above=1pt]{\scriptsize $4$};
\protect\draw[fill=white] (-2 cm, 0 cm) circle (.15cm) node[above=1pt]{\scriptsize $3$};
\protect\draw[fill=white] (-4 cm, 0 cm) circle (.15cm) node[above=1pt]{\scriptsize $2$};
\protect\draw[fill=white] (-6 cm, 0 cm) circle (.15cm) node[above=1pt]{\scriptsize $1$};
\end{tikzpicture},
take $w=3423123432$. In this case, we have $\mathbf{JM}(w) = \{34231234, \, 32341232\}$, but $\mathbf{JM'}(w) = \mathbf{JM}(w) \cup \{3234323 \}$.
\end{example}

\begin{remark}\label{D4exrem}
Note that $w=\bigvee\mathbf{JM}(w)$ is the unique minimal join expression of $w$ in the dissective types (see Proposition \ref{prop:dissJM}, Lemma \ref{lem:diss} and Corollary \ref{cor:JM=JM}). This is not true in general. The smallest type where one can find a counterexample is $D_4$.
In this type, we have a unique counterexample, namely $w=1\Ou{}\Od{}21$. It has $\mathbf{JM}(w)=\mathbf{JM}'(w) = \{1\Ou{}21, 1\Od{}21, 1\Ou{}\Od{}1 \}$. None of the latter three elements 
is a dissector. However, $w$ can be written as the join of any two of them. In particular, $w$ is a bigrassmannian which is not join-irreducible. We give the Bruhat graph of $\BG{1}{1}$, which contains all the above mentioned elements:
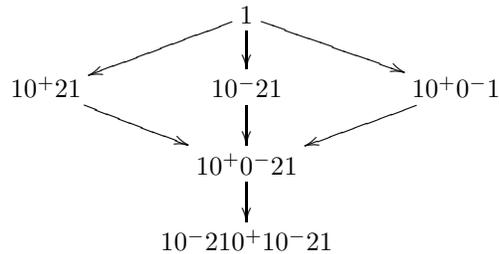
\[\xymatrix@R=1.5em{
 & 1 \ar[dl] \ar[d] \ar[dr]   & \\
1\Ou{}21 \ar[dr]   &     1\Od{}21 \ar[d]  &    1\Ou{}\Od{}1 \ar[dl] \\
& 1\Ou{}\Od{}21 \ar[d]  & \\
 & 1\Od{}21\Ou{}1\Od{}21 & } \]
In Subsection~\ref{ss:JID} below, we show
that these three elements correspond to three distinct simple submodules of the isotypic subquotient $2L_x\langle  -9\rangle$ in $\Delta_e$, and the sum of any two is necessarily the entire $2L_x\langle  -9\rangle$. 
This is the place (unique in type $D_4$) where the polynomial  $p_{e,x}$, with $x \in \jc$, has a coefficient bigger than one. Namely, $x=\Ou{}1\Od{}21\Ou{}21\Od{}12 \in {}^1 \mathcal{H}^1$ and $p_{e,x}=v^{11}+2v^9+v^7$.
\end{remark}

\begin{example}
The smallest $W$ where $\mathbf{JM}(w) \not\subseteq \mathbf{JM'}(w)$, for some $w$, is of type $D_6$. We denote the simple roots by \begin{tikzpicture}[scale=0.4,baseline=-3]
\protect\draw (4 cm,0) -- (2 cm,0);
\protect\draw (2 cm,0) -- (0 cm,0);
\protect\draw (0 cm,0) -- (-2 cm,0);
\protect\draw (-2 cm,0) -- (-4 cm,0.7 cm);
\protect\draw (-2 cm,0) -- (-4 cm,-0.7 cm);
\protect\draw[fill=white] (4 cm, 0 cm) circle (.15cm) node[above=1pt]{\scriptsize $4$};
\protect\draw[fill=white] (2 cm, 0 cm) circle (.15cm) node[above=1pt]{\scriptsize $3$};
\protect\draw[fill=white] (0 cm, 0 cm) circle (.15cm) node[above=1pt]{\scriptsize $2$};
\protect\draw[fill=white] (-2 cm, 0 cm) circle (.15cm) node[above=1pt]{\scriptsize $1$};
\protect\draw[fill=white] (-4 cm, 0.7 cm) circle (.15cm) node[above=1pt]{\scriptsize $\Ou{}$};
\protect\draw[fill=white] (-4 cm, -0.7 cm) circle (.15cm) node[above=1pt]{\scriptsize $\Od{}$};
\end{tikzpicture} and take $w=1\Ou{}21\Od{}321\Ou{}4321\Od{}321\Ou{}12$. Then we have
\begin{align*}
\mathbf{JM}(w) &=   \{ 4321\Od{}\Ou{}12, \, 1\Ou{}21\Od{}321\Ou{}4321\Od{}, \, 1\Od{}21\Ou{}321\Od{}4321\Ou{}, \, 1\Ou{}21\Od{}1\Ou{}2132 \} , \\
\mathbf{JM'}(w) &=  \{ 4321\Od{}\Ou{}12, \, 1\Ou{}21\Od{}321\Ou{}4321\Od{}, \, 1\Od{}21\Ou{}321\Od{}4321\Ou{}, \, 1\Od{}21\Ou{}321\Od{}21\Ou{}321 \} ,
\end{align*}
while $\mathbf{JM}(w) \cap \mathbf{JM'}(w)$ contains only $3$ elements.
\end{example}

Put 
\begin{equation}\label{JM''}
    \mathbf{JM''}(w) := \{ x \in  \mathbf{JM'}(w)\ |\ x \in \desc{LD(w)}{RD(w)} \}
\end{equation}
and  note that $\mathbf{JM''}(w) = \mathbf{JM}(w) \cap \mathbf{JM'}(w)$.
We do not know if $w=\bigvee\JM''(w)$ holds in general. In type $D_6$ (and all types of rank smaller than $6$) we do have $w = \bigvee \mathbf{JM''}(w)$.

\section{Join operation and Verma modules}\label{s-new4}

\subsection{Intersection conjecture}

Our main conjecture is that the join expression $w=\bigvee \mathbf{JM}(w)$ gives 
rise to a join expression for the Verma module $\Delta_w$.

\begin{conjecture}\label{con2JM}
For $w\in W$, we have 
\begin{equation}\label{eq:joinexpintersection}
\Delta_w = \bigcap_{x\in \mathbf{JM}(w)}\Delta_x ,    
\end{equation}
where the intersection is taken inside (the ungraded module) $\Delta_e$.
\end{conjecture}

Each Verma module $\Delta_w$ is uniquely determined by the socle of 
$\Delta_e/\Delta_w$. By Proposition~\ref{prop3}, each simple subquotient of
the latter socle is isomorphic to a shift of $L_u$, for some $u\in \jc$. Thus, if each $L_u\langle d\rangle$ (for $u\in\jc$ and $d\in\mathbb Z$) is multiplicity-free in $\Delta_e$, then \eqref{eq:joinexpintersection} is equivalent to
\begin{equation}\label{[conj]}
[\Delta_w\langle -\ell(w)\rangle:L_u\langle d \rangle] = 
\min_{x\in\JM(w)}[\Delta_x\langle -\ell(x)\rangle:L_u\langle d \rangle],
\end{equation}
for all $u\in\jc$ and $d\in \mathbb Z$. This reduces the module-theoretic 
Conjecture~\ref{con2JM} to Kazhdan-Lusztig combinatorics. 
The multiplicity-free assumption holds in types $A$ and $B$.
In type $A$, Conjecture~\ref{con2JM} 
follows from the results of \cite{kmm2}, see Corollary~\ref{con2JMA}.
In type $B$, we checked \eqref{[conj]} for small ranks on a computer.

In the next subsections, we discuss another approach to Conjecture~\ref{con2JM}. 
 
\subsection{Socle-sum property and the intersection conjecture}

The following proposition explains a relation between the intersection 
of Verma modules and and the socles of the corresponding quotients.

\begin{proposition}\label{sumtocap}
Let $w\in W$ and $U\subset W$.
Assume that
the following holds:
\begin{itemize}
\item $x\leq w$, for any $x\in U$;
\item the natural surjection
$\co_w\to\co_x$ restricts to a surjection $\phi_x:\operatorname{soc}\Delta_e/\Delta_w\to \operatorname{soc}\Delta_e/\Delta_x$, for any $x\in U$;
\item $\displaystyle\bigcap_{x\in U}\mathrm{ker}(\phi_x)=0$.
\end{itemize}
Then $\displaystyle\bigcap_{x\in U} \Delta_x = \Delta_w$.
\end{proposition}

\begin{proof}
That $\Delta_w\subseteq \displaystyle\bigcap_{x\in U} \Delta_x$ follows directly from the first assumption.
For the opposite inclusion, consider the quotient map $\pi: \Delta_e\surj\Delta_e/\Delta_w$. 
We claim that the kernel of $\pi$ contains $\displaystyle\bigcap_{x\in U}\Delta_x=:M$. 
Suppose not. Then $\pi(M)$ intersects the socle of $\co_w$, and, by our assumptions, there exist $x\in U$ such that $\phi_x(\pi(M)\cap \soc\co_w)\neq 0$. But this contradicts $M\subseteq \Delta_x$.
Therefore, the map $\pi$ factors through $\Delta_e/M \surj \co_w$, which implies that $M\subseteq \Delta_w$.
\end{proof}

Note that the converse of Proposition \ref{sumtocap} is not true.
For example, if $x<y=w$, then $\Delta_x\cap \Delta_y = \Delta_w$.
However, the map $\so_y\to \so_x$ is often not surjective.

Proposition~\ref{sumtocap} and Conjecture~\ref{con2JM} motivate Definition~\ref{defintro}.

\begin{corollary}\label{capverma}
If all $w\in W$ has the socle-sum property, then Conjecture~\ref{con2JM} is true for $(W,S)$.
\end{corollary}
\begin{proof}
This follows from  Proposition \ref{sumtocap}.
\end{proof}

The result in \cite{kmm2} readily proves the socle-sum property and Conjecture~\ref{con2JM} in type $A$. 

\begin{theorem}\label{con1A}
Let $(W,S)$ be of type $A$. Then every $w\in W$ has the socle-sum property.
\end{theorem}

\begin{proof}
Since the bigrassmanians are exactly the join-irreducibles in type A, the statement follows from the main theorem in \cite{kmm2}.
\end{proof}

\begin{corollary}\label{con2JMA}
Conjecture~\ref{con2JM} is true in type $A$.
\end{corollary}
\begin{proof}
This follows from Corollary~\ref{capverma} and Theorem~\ref{con1A}.
\end{proof}

%Compare the following examples concerning Conjecture~\ref{con1} and Conjecture~\ref{con2JM} respectively.

Note that it is crucial to take the intersection over $\mathbf{JM}(w)$ in 
Definition~\ref{defintro}, as we see in the following easy example. 

\begin{example}\label{con1notJM}
We have $w=w\vee x$, for any $x<w$. In this case, $\soc\co_x$ does not contribute to $\so_w$. For example, if $w=st$ for $s,t\in S$ then $\so_w$ is of the form $L_{w}$ (up to shift) with $w\in {}^s\hc{}^t$ (see Proposition~\ref{prop3} and Proposition~\ref{prop6}). Writing $st=st\vee s$ we see that $\so_s$, which is of the form $L_{x}$ for $x\in{}^s\hc^s$ (see Proposition~\ref{prop3} and Proposition~\ref{prop6}), is not contained in $\so_w$.
\end{example}

In contrast, Conjecture~\ref{con2JM} seems to be less sensitive
to the choice of a join expression.
Here is a comparison with the situation in Example~\ref{con1notJM}.

\begin{example}\label{con2notJM}
We have $w=w\vee x$, for any $x<w$. In this case, we have $\Delta_w=\Delta_w\cap \Delta_x$ since $\Delta_x\supseteq \Delta_w$.
\end{example}

The following proposition also suggests that \eqref{eq:joinexpintersection} may not require the join expression to be given by $\mathbf{JM}$. 

\begin{proposition}\label{propnw-2}
If $\displaystyle\bigcap_{x\in U}\Delta_x = \Delta_w$, then $w=\bigvee U$.
\end{proposition}

\begin{proof}
Suppose $\displaystyle\bigcap_{x\in U}\Delta_x = \Delta_w$. Then we have $x\leq w$, for each $x\in U$. If $z\in W$ is such that $x\leq z$, for all $x\in U$, then $\Delta_x\supseteq \Delta_z$, for all $x\in U$, so $\displaystyle\Delta_w =\bigcap_{x\in U}\Delta_x \supseteq \Delta_z$.
The latter implies $w\leq z$ which proves the claim.
\end{proof}

This motivates the main question of the paper (Question~\ref{intque}).

\begin{question}\label{con2}
Given $U\subseteq W$, do we have $w=\bigvee U$ if and only if $\displaystyle\bigcap_{x\in U}\Delta_x = \Delta_w$?
\end{question}

Note that the join $\bigvee U$ does not exist in general, in which case $\displaystyle\bigcap_{x\in U}\Delta_x$ is not isomorphic to a Verma module by Proposition~\ref{propnw-2}. Therefore, for such $U\subseteq W$, the answer to Question~\ref{con2} is positive. 

Under a combinatorial condition on the Bruhat order, the socle-sum property guarantees that
the answer to Question~\ref{con2} is positive.

\begin{proposition}\label{disscon1}
Let $(W,S)$ be dissective and suppose Conjecture~\ref{con2JM} holds for $W$.
Then the answer to Question~\ref{con2} is positive, for any $U\subseteq W$.
\end{proposition}

\begin{proof}
The ``if'' claim follows from Proposition~\ref{propnw-2}. 
To prove the ``only if'' part, let $U\subseteq W$ and $\bigvee U = w$. 
Consider the equality
\begin{equation}\label{vvv}
\bigvee_{y\in U}\bigvee_ {x\in \mathbf{JM}(y)} x = \bigvee_{z\in \mathbf{JM}(w)} z.    
\end{equation}
Since $W$ is dissective, by Proposition~\ref{prop:dissJM} and Corollary~\ref{cor:JM=JM}, the expression on the right hand side of \eqref{vvv} is a subexpression of the left hand side. 
Therefore, for any $y\in U$ and each $x\in\mathbf{JM}(y)$, either $x\in\mathbf{JM}(w)$ or $x\leq z$, for some $z\in \mathbf{JM}(w)$.
%Consequently, for each $x$ appearing in the left hand side of \eqref{vvv}, we have $\Delta_x\supseteq \bigcap_{z\in\mathbf{JM}(w)} \Delta_z$.
Therefore, we have
\begin{equation}\label{capcapcap}
\bigcap_{y\in U}\bigcap_{x\in\mathbf{JM}(y)}\Delta_x = \bigcap_{z\in \mathbf{JM}(w)}\Delta_z.    
\end{equation}
Now the validity of Conjecture~\ref{con2JM} reduces \eqref{capcapcap} to $\displaystyle\bigcap_{y\in U}\Delta_y = \Delta_w$.
\end{proof}

\begin{corollary}\label{coraddedforintro}
Let $(W,S)$ be dissective and suppose all $w\in W$ have the socle-sum property.
Then the answer to Question~\ref{con2} is always positive.
\end{corollary}
\begin{proof}
This follows from Corollary~\ref{capverma} and Proposition~\ref{disscon1}.
\end{proof}

\begin{corollary}\label{typeAconj}
Let $(W,S)$ be of type $A$. Then, for any $U\subseteq W$, we have 
$\bigvee U =w\in W$ if and only if $\displaystyle\bigcap_{x\in U}\Delta_x = \Delta_w$.
\end{corollary}

\begin{proof}
This follows from Theorem~\ref{con1A}, Lemma~\ref{lem:diss}, and
Proposition~\ref{disscon1}.
\end{proof}

As we will see later in Example~\ref{e644bigex}, in the situation when
$W$ is not dissective, the answer to Question~\ref{con2} can be negative,
for some $U\subseteq W$. 
Therefore, we ask a more precise question:

\begin{question}\label{con2precise}
Given $w\in W$, which $U\subseteq W$ satisfying $w=\bigvee U$ have 
the property $\displaystyle\bigcap_{x\in U}\Delta_x = \Delta_w$?
\end{question}

\subsection{Another interpretation of Question \ref{con2}}

Let $Sub$ be the poset of all submodules of $\Delta_e$ where the poset structure is given by inclusion. Then we have an embedding of posets
\begin{equation}\label{mapdelta}
    \Delta: (W,\leq) \to (Sub,\subseteq)^{\operatorname{op}}
\end{equation}
given by $w\mapsto \Delta_w\subseteq\Delta_e$.
In $(Sub,\supseteq)=(Sub,\subseteq)^{\operatorname{op}}$, arbitrary joins exist, since joins are given by intersections. 
If a join of some elements in the image $\Delta(W,\leq)$ lies in $\Delta(W,\leq)$ again, then the join exists already in $(W,\leq)$.
While the above statement is true for any poset inclusion, the converse, that is, whether joins are mapped to joins under the inclusion, is not true in general.

Let $CSub$ be the subposet of $Sub$ consisting of the submodules of $\Delta_e$ with simple head. 
The join operation on $(Sub,\supseteq)$ restricts to that in $(CSub,\supseteq)$, that is, the meet in $Sub$ restricts to $CSub$, in the following sense.

\begin{proposition}\label{capCSub}
For $P\subseteq CSub$, the meet $\bigwedge P$ exists if and only if $\displaystyle\bigcap_{M\in P} M$ has simple head, in which case we have $\displaystyle\bigwedge P = \bigcap_{M\in P} M$.
\end{proposition}

\begin{proof}
The ``if'' statement is clear. 
Conversely, if the head of $\displaystyle\bigcap_{M\in P} M$ is not simple, then there are two incomparable  submodules $\displaystyle X,Y\subset \bigcap_{M\in P} M$ with simple heads, none of which is properly contained in any other submodule  of $\displaystyle \bigcap_{M\in P} M$ with simple head. 
If the meet exists, then it satisfies $\displaystyle\bigcap_{M\in P} M \supseteq \bigwedge P \supseteq X,Y$, contradicting the maximality of $X,Y$.
\end{proof}

\begin{corollary}\label{con2andCSUB}
The answer to Question \ref{con2} is positive, for every $U\subseteq W$, 
if and only if the join operation on $CSub^{\operatorname{op}}$ 
restricts to the join operation on $W$.
\end{corollary}

\begin{proof}
Consider the following diagram where
the downward map $Sub^{\operatorname{op}}\to CSub^{\operatorname{op}}\coprod \{\text{`does not exist'}\}$ is given by $M\mapsto M$, if $M\in CSub$, and $M\mapsto \text{'does not exist'}$, otherwise. The map below the latter is given similarly.
\[ \xymatrix{
 \prod Sub^{\operatorname{op}} \ar[r]^-\vee & Sub^{\operatorname{op}}  \\
 \prod CSub^{\operatorname{op}} \ar[r]^-\vee
 \ar@{}[u]|{\rotatebox[origin=c]{90}{$\subset$}} & CSub^{\operatorname{op}}\coprod \{\text{`does not exist'}\} \ar@{}[u]|{\rotatebox[origin=c]{-90}{$\to$}}  \\
  \prod W \ar[r]^-\vee 
  \ar@{}[u]|{\rotatebox[origin=c]{90}{$\subset$}} & W \coprod \{\text{`does not exist'}
 \ar@{}[u]|{\rotatebox[origin=c]{-90}{$\to$}}   \}
} \]
Conjecture~\ref{con2} says that the big square commutes. Proposition~\ref{capCSub} says that the upper square commutes. It follows that Conjecture~\ref{con2} is equivalent to the lower square being commutative, which is our claim.
\end{proof}

Intuitively, the poset $CSub$ is isomorphic to the poset of simple subquotients of 
$\Delta_e$ where the partial order is given by (some kind of) generation. 
Let us make this statement more general and precise. For $M\in\mathcal{O}$, 
consider the \emph{set of subquotients} of $M$ given by
\[SQ(M) = \{(X,Y)\ |\ Y\subseteq \operatorname{rad}X\subseteq M \},\]
where one thinks of $(X,Y)$ as the subquotient $X/Y$ of $M$.
Then $SQ(M)$ is a poset with the partial order 
\[(X,Y)\leq (Z,V) \Longleftrightarrow X\subseteq Z \text{ and } Y\subseteq V,\]
which should be thought of as the induced map $X/Y\to Z/V$.
(Here we are not considering all subquotients, for example, we only consider `quotients of submodules' and not `submodules of quotients'. The zero subquotients are also not included in $SQ(M)$. But the above definition suffices for our purpose and allows us to avoid lengthy discussions.)
The \emph{poset of simple subquotients} of $M$ is the subposet $SSQ(M)$ in $SQ(M)$ consisting of the pairs $(X,Y)$ such that $\hd X$ is simple and $Y=\operatorname{rad}X$ (so that $X/Y$ is simple).
Then $CSub(M)$ (the submodules of $M$ with simple head) is isomorphic to $SSQ(M)$, the isomorphism being given by $X\mapsto (X,\operatorname{rad}X)$.

Note that in $SSQ(M)$ we have $(X,Y)\leq (Z,V)$ if and only if $X\subseteq Z$, since the second condition $Y\subseteq V$ is automatic.
We thus see that the partial order on $SSQ(M)$ is alternatively given by $(X,Y)\leq (Z,V)$ if and only if there is a composition series of the form 
$0\subsetneq \cdots\subsetneq X\subsetneq\cdots\subsetneq Z$, 
which is the most natural partial order that can be given to the simple subquotients of $M$. 

Taking $M=\Delta_e$ and relating to the poset $(W,\leq)$,
we have the diagram of posets 
\begin{equation}\label{eqeqneq1}
 (W,\leq) \xhookrightarrow{\Delta} (CSub,\supseteq)\xrightarrow[\simeq]{\operatorname{hd}} (SSQ,\geq)\xrightarrow{\deg} (\mathbb Z,\leq)
,\end{equation}
where the first map is induced from \eqref{mapdelta} (where $(CSub,\supseteq)$ denotes the opposite poset of $CSub$ and similarly for $SSQ$) and the last map takes the minimal degree of a 
(not necessarily graded) simple subquotient in the graded module $\Delta_e$. The 
composition in \eqref{eqeqneq1} gives the length function on $(W,S)$.

The discussion above proves the following reformulation of Corollary~\ref{con2andCSUB}.

\begin{corollary}\label{cormdiagram}
The answer to Question~\ref{con2} is positive, for every $U\subseteq W$, if and only if the join operation on $SSQ$ restricts to the join operation on $W$, i.e., 
if and only if the following diagram commutes.
\[ \xymatrix{
 \prod SSQ^{\operatorname{op}} \ar[r]^-\vee
% \ar@{}[u]|{\rotatebox[origin=c]{90}{$\subset$}} 
& SSQ^{\operatorname{op}}\coprod \{\text{`does not exist'}\}   \\
  \prod W \ar[r]^-\vee 
  \ar@{}[u]|{\rotatebox[origin=c]{90}{$\subset$}} & W \coprod \{\text{`does not exist'}
 \ar@{}[u]|{\rotatebox[origin=c]{-90}{$\to$}}   \}
} \]
\end{corollary}

We note that Corollary~\ref{cormdiagram} is non-trivial. Indeed,
it follows from Corollary~\ref{cormdiagram} and Example~\ref{e644bigex} below that the join operation on $SSQ$ or $CSub$ does not restrict to the join expression on $W$ in type $E$.

\subsection{Half socle-sum property via combinatorics}

Here we give a criterion for the socle-sum property to hold in one direction. Type by type analysis in the next section will prove that the criterion is always satisfied, hence establishing Theorem~\ref{intthm:sosum}.

\begin{lemma}\label{newbgchain}
Let $s,t\in S$ and $m:=\sum_{u\in ^s\hc ^t}p_{e,u}(1)$. Suppose there exists a chain 
\[x_1<\cdots <x_m\]
in $\BG{s}{t}$. Then every composition factor of $\Delta_e$ isomorphic to 
$L_u$ with $u\in ^s\hc ^t$ is contained in the sum of $\soc\co_y$, taken over all $y\in \JI{s}{t}$. 
\end{lemma}

By `$X$ is contained in the sum of socles of $\Delta_e/ \Delta_{y_i}$' we mean the following: 
Given a maximal isotypic subquotient $X\cong cL_x\langle d \rangle$ in $\Delta_e$, let $M$ be the submodule of $\Delta_e$ generated by $X$, so that $\operatorname{hd}M = X$. Then there are some $y_i\in \JI{s}{t}$ so that $\Delta_{y_i}\subseteq \operatorname{rad} M$ and the kernels of the induced maps
$\psi_i:X \inj  \soc(\Delta_e/\operatorname{rad} M)\to \soc(\Delta_e/ \Delta_{y_i})$
have trivial intersection in $X$.

It is more convenient to formulate the dual statement. Let $K_x$ be the kernel of the canonical surjection $\nabla_e\to \nabla_x$, so that a statement about $\soc\co_x$ is dual to a statement about $\hd K_x$.
Let $Y$ be a graded isotypic component in $\nabla_e$ corresponding to an element in $\hc(a,b)$ and $N\subset \nabla_e$ be the submodule generated by $Y$. Then the statement 
in the previous paragraph is equivalent to 
\[\bigoplus K_y\xrightarrow{\oplus \eta_y} N\]
being surjective. Here $y$ runs over all elements in $\JI{a}{b}$ for which the map $K_y\inj \nabla_e$ restricts to  $\eta_i:K_y \inj N$.

\begin{proof}
By Proposition~\ref{prop3} and Proposition~\ref{prop6}, the numbers
\[\sum_{u\in ^s\hc ^t}[\co_{x_i}: L_u ]\]
strictly increase along $x_1<\cdots <x_m$, and moreover, each component of the form $L_u$ with $u\in ^s\hc ^t$ in  $\Delta_{x_{i-1}}/\Delta_{x_{i}}$ belongs to $\so_{x_i}$ under  $\Delta_{x_{i-1}}/\Delta_{x_{i}}\inj \co_{x_i}$. 
Thus the length of the sum of $\so_{x_1},\cdots, \so_{x_i}$ strictly increases with $i$. In particular, the sum of $\so_{x_1},\cdots, \so_{x_m}$ has length at least $m$. But, by Proposition~\ref{prop6} and \eqref{eq:p_Delta}, the length is at most $m$.
The claim follows.
\end{proof}

\begin{proposition}\label{ccoor:sum}
Suppose that, for each $s,t\in S$, every composition factor of $\Delta_e$
isomorphic to $L_u$, with $u\in {}^s\hc ^t$, 
is contained in the sum of $\soc\co_y$ taken over $y\in \JI{s}{t}$. 
Then, for each $w\in W$, we have
\[\bigcap_{y\in \mathbf{JM}(w)} \ker \{\operatorname{soc}\Delta_e/\Delta_w\xrightarrow{\phi_y} \operatorname{soc}\Delta_e/\Delta_y\}=0,\] where the maps
$\phi_y$ are induced by $\Delta_y\subseteq \Delta_w$. In other words, the socle of $\co_w$ is contained in the sum of $\so_{x}$ taken over $x\in \JM(w)$.
\end{proposition}

\begin{proof}
Suppose not. Then we have a simple socle component $X\subseteq \soc\co_w$ such that $\phi_y(X)=0$, 
for all $y\in \mathbf{JM}(w)$. 
By Proposition \ref{prop3}, it is isomorphic to (some shift of) $L_x\in \jc$.
Let $a$ and $b$ be the left and right ascents of $x$, respectively, i.e., $x\in {}^a\hc^b$.
Then the assumption provides $y_i\in \JI{a}{b}$ such that the sum of $\soc\co_{y_i}$ contains $X$. By replacing $y_i\leq w$ by  $y'_i\in \mathbf{JM}(w)$ with $y_i\leq y'_i$, we get a contradiction.
\end{proof}

\begin{proposition}\label{cor:sum}
Suppose, for each $s,t\in S$, there exists a chain in $\BG{s}{t}$ of 
length $\displaystyle\sum_{u\in ^s\hc ^t}p_{e,u}(1)$. 
Then, for each $w\in W$, we have
\[\bigcap_{y\in \mathbf{JM}(w)} \ker \{\operatorname{soc}\Delta_e/\Delta_w\xrightarrow{\phi_y} \operatorname{soc}\Delta_e/\Delta_y\}=0,\] where the maps
$\phi_y$ are induced by $\Delta_y\subseteq \Delta_w$. In other words, the socle of $\co_w$ is contained in the sum of $\so_{x}$ takenover $x\in \JM(w)$.
\end{proposition}

\begin{proof}
This follows from Lemma \ref{newbgchain} and Proposition~\ref{ccoor:sum}.
\end{proof}

The proof of Corollary~\ref{cor:sum} is also easier to make precise in the dual setting. 
The dual statement of Lemma~\ref{newbgchain} gives that $\oplus_{y_i} K_y\surj M$, for some $y_i\in\JI{a}{b}$, where $M$ is the submodule of $K_w$ generated by a head component isomorphic to $L_x$, for $x\in\hc(a,b)$.
Then, taking $y_i'\in\JM(w)$ with $y'_i>y_i$, we factor the above surjective map as
$\oplus_{y_i} K_{y_i}\to \oplus_{y'_i} K_{y_i'}\xrightarrow{\eta} M$. This implies that $\eta$ is surjective, as desired.

\subsection{Socle-sum property and $\JM''$}\label{ss:JM''}

Recall the set $\JM '' (w)$, for $w\in W$, defined in \eqref{JM''}.
One could ask whether $\JM''(w)$ might be a better 
join expression than $\JM(w)$ for the socle-sum property. 
However, Example~\ref{F4example} gives an example of $w\in W$ with $\JM(w)  \supsetneq \JM''(w)$ such that 
\[\so_w = \sum_{z\in \JM''(w)} \so_z \subsetneq \sum_{z\in \JM(w)} \so_z .\]
%
%????Also, the sum in the socle-sum property, when true, can be direct when we restrict to $\JM''(w)$, i.e., we have
%\[\so_w = \bigoplus_{z\in \JM''(w)}\so_z = \sum_{z\in \JM(w)}\so_z\neq \bigoplus_{z\in \JM(w)}\so_z. \]
%See Proposition~\ref{socleabD} \eqref{socOBX} and \hk{add this in S23}
% \phi JM = JM''
Also, Remark~\ref{soclesumfk} shows that the socle-sum property does not hold, in general, 
if we replace $\JM$ by $\JM''$ either. In fact, Remark~\ref{soclesumfk} shows that there is no join expression such that the socle-sum property always holds, since it provides a non-join-irreducible element $f\in W$ such that
\[\so_f = L\langle d \rangle \oplus L'\langle d' \rangle\]
while there is no $z\in W$ with $\so_z \cong L\langle d \rangle$.
We also do not know whether $\JM''$ provides a join expression, i.e., whether $\bigvee \JM''(w) = w$ is true in general.

%\section{Kazhdan-Lusztig polynomials}\label{s:KL}
\section{Join-irreducibles and Verma modules}\label{s:KL}\label{s:JIcase}

The goal of this section is to explicitly describe  $\so_y$, for $y\in \JI{}{}\subset W$, by closely studying combinatorial aspects of $(W,S)$. 
Our first ingredient is the Kazhdan-Lusztig combinatorics. We explicitly compute $p_{e,y}$, for $y\in \jc$, which determines the multiplicities of the composition subquotients of $\Delta_e$ of the form $L_y\langle d\rangle$.
The second ingredient is the Bruhat order on $W$. We determine the subposet $\JI{s}{t}\subset W$, for each $s,t\in S$, and match them with the Kazhdan-Lusztig polynomials.

We collect some general lemmas in Subsection~ \ref{s:Strategies} and proceed to Weyl type analysis in the following subsections.

\subsection{Strategies}\label{s:Strategies}

\subsubsection{Useful equations in KL polynomials}

\begin{lemma}\label{lem:KL_sz_y}
Let $y \in \mathcal{J}$, $z \in W$ and $s \in S$ be such that $z<sz$ and $z \neq y,sy$. Then we have
\[ p_{sz, y} = \begin{cases}
v^{-1} \cdot p_{z,y}, & sy<y;  \\[.75em]
\displaystyle p_{z,sy} + \sum_{\substack {t \in S \\ ty\sim_{\mathtt{L}} y \\ ty<y}} p_{z,ty} - v \cdot p_{z,y}, & sy>y. \end{cases}\]
\end{lemma}
\begin{proof}
In Equation \eqref{sy} we take the coefficients next to $H_z$ on both sides. Note that all $x$ appearing in the sum there satisfy $x \sim_\mathtt{L} y$, by the definition of the order $\leq_\mathtt{L}$ and the penultimate cell $\mathcal{J}$. Proposition \ref{muJ} finishes the proof.
\end{proof}

We record separately the case $z=e$. %, which is also proved in \cite[Lemma 9 and Lemma 8]{kmm2}.

\begin{lemma}\label{prop2p}
Suppose $y \in \mathcal{J}$ and $s\in S$ are such that $sy>y$. Then we have
\begin{equation}\label{2p}
    v \cdot p_{e,y} + p_{s,y} = p_{e,sy} + \sum_{\substack {u \in S \\ uy\sim_{\mathtt{L}} y \\ uy<y}} p_{e,uy}.
\end{equation}
Moreover, if $y\in {}^i\hc^j$ with $i\neq j$, then both sides of \eqref{2p} are equal to $(v+v\inv) \cdot p_{e,y}$.
\end{lemma}

\subsubsection{Socle-killing combinatorics}

%We use the notation from Subsection \ref{ss:posets} and Section \ref{s:join_Bruhat}. 

By Lemma \ref{socledie}, the following combinatorial property plays an important role in determining socles.

\begin{defn}\label{soclechain} A relation $x<y$, with $y\in\JI{s}{t}$, is called \emph{socle-killing} if either $x \leq s y$ or $x \leq y t$ holds. A chain $x_0<x_1<\cdots < x_m$ in $\JI{s}{t}$ is called \emph{socle-killing} if the pair $x_{i} < x_{i+1}$ is socle-killing, for each $i=0,1,\ldots, m-1$.
%there is $z\in \BG{u}{v}$ with $(u,v) \neq (s,t)$ and such that $x_{i-1}\leq z\leq x_i$. 
\end{defn}

%\hk{Some of the next lemmas were for some of the type E cases which we end up not including. Remove what's not used before submitting.}

To bound the degrees in which the socle of $\Delta_e/\Delta_x$ can appear, we start with the following lemma.

\begin{lemma} Let $x,y \in W$.
\label{lem:bruh_soc}
\begin{enumerate}
    \item\label{bru mm} If $x \leq y$, then $\maxd \Delta_e/\Delta_{x} \leq \maxd \Delta_e/\Delta_y$. If moreover $\maxd \Delta_e/\Delta_{x} = \maxd \Delta_e/\Delta_y$, then the maximal degree component of $\Delta_e/\Delta_x$ is contained in the maximal degree component of $\Delta_e/\Delta_y$.
    
    \item\label{sk m<m} If $y$ is join-irreducible, and the relation $x<y$ is socle-killing, then we have \[ \maxd \Delta_e/\Delta_{x} < \maxd \Delta_e/\Delta_{y}. \]
    
    \item If all the composition factors of $\soc \Delta_e/\Delta_x$ appear in $\Delta_e/\Delta_y$, then $x \leq y$.
\end{enumerate}
\end{lemma}
\begin{proof}
The first two claims follow directly from the definitions.
To prove the last claim, assume that $x \not\leq y$. Then the composition factor $L_y$ that comes from the top of $\Delta_y$, appears in $\Delta_e/\Delta_x$. So, the image of
$\Delta_y$ in $\Delta_e/\Delta_x$ has a non-trivial intersection with the socle
of $\Delta_e/\Delta_x$. This is a contradiction to our assumption.
\end{proof}

Let us introduce some auxiliary notation. 
For a fixed penultimate $\mathtt{H}$-cell ${}^s\hc^t$, write 
\begin{equation}\label{eq:newpex}
    p_{st}:=\sum_{w\in {}^s\hc^t} p_{e,w}=c_0v^{d_0}+ c_1v^{d_1}+\cdots+ c_rv^{d_r}.
\end{equation} 
Here $c_i=c_i(s,t)\neq 0$ and, furthermore, $d_i=d_i(s,t)\in \mathbb N$ form 
a strictly increasing sequence.
In particular, the number $r+1=r(s,t)+1$ is the number of homogeneous terms in $p_{st}$.
For $y\in \JI{s}{t}$, define $\skal(y)$ (here ``skal'' stands for {\em socle killing above length}) to be the maximal length of all socle-killing chains in $\JI{s}{t}$ ending at $y$, i.e., the maximal $m$ such that there exists a socle-killing chain $y_0<\ldots <y_m = y$ with $y_i\in \JI{s}{t}$. 
Similarly, define $\skbl(y)$ ({\em socle killing below length}) to be the maximal length of the socle-killing chains in $\JI{s}{t}$ starting from $y$.

\begin{lemma}\label{skandmaxd}
Let $y\in \JI{s}{t}$. Then we have
\[ d_{\skal(y)} \leq \maxd\so_y \leq d_{(r-\skbl(y))}.\]
\end{lemma}
\begin{proof}
This follows from Lemma~\ref{lem:bruh_soc}.
\end{proof}

The minimal degree is not bounded by socle-killing chains on the nose but can be determined inductively by the following lemma.

\begin{lemma}
\label{lem:soc_low_bound}
Let $y\in \JI{s}{t}$ and suppose that the sum of $\so_x$, taken over all $x\in \JI{s}{t}$ such that $x < y$ is socle-killing, contains all subquotients in $\Delta_e$ isomorphic to $L_w\langle -d_i\rangle$ with $w\in {}^s\hc^t$ and $i\leq m$. Then we have
\[d_m< \mind\so_y.\]
\end{lemma}

\begin{proof}
This follows from Proposition~\ref{prop6} and Lemma~\ref{socledie}.
\end{proof}

In particular, if $c_i=1$, for all $i< \skal(y)$, then we have 
\[ d_{\skal(y)} \leq \mind\so_y \leq \maxd\so_y\leq d_{(r-\skbl(y))},\]
where the first inequality comes from Lemma \ref{lem:soc_low_bound} and the last one from Lemma \ref{skandmaxd}. If, moreover, $y$ belongs to a socle-killing chain of length $r$, i.e., if $\skal(y) + \skbl(y) =r$,  then $\so_y$ is homogeneous.

% If $c_i>1$, for some $i\leq \skal(y)$, then we need the following lemma.
% 
% \begin{lemma}\label{chaincondgen}
% Let $w\in W$ and suppose we have $w=y_1\vee \cdots \vee y_k$ with $y_i\in \JI{s}{t}$ such that
% \begin{enumerate}
%     \item $y_1<y_1\vee y_2<\cdots< y_1\vee\cdots\vee y_{k-1}< w $;
%     \item $\so_{y_i}$ are homogeneous of the same degree, say $m$, for all $1\leq i\leq k$;
% \end{enumerate}
% Then we have
% \[\sum_{x\in {}^s\hc^t}[\co_w,L_x\langle-m\rangle]\geq k.\] 
% \end{lemma}
% \begin{proof}
% We prove the claim by induction on $k$. 
% Let $w$ and $y_1,\ldots,y_k$ as above. Put $z:=y_1\vee\cdots\vee y_{k-1}<w$, and observe that $z$ and $y_1,\ldots,y_{k-1}$ satisfy the conditions of the lemma, and thus the induction hypothesis
% implies that
% $\sum_{x\in {}^s\hc^t}[\co_z:L_x\langle-m\rangle]\geq k-1$.
% Since $z< z\vee y_k$ we have in particular $y_k\not\leq z$, while $y_k\leq w$. Therefore, $\so_{y_k}$ is not a subquotient of $\co_z$, while it is a subquotient of $\co_w$. It follows that
% \[\sum_{x\in {}^s\hc^t}[\co_w:L_x\langle-m\rangle] > \sum_{x\in {}^s\hc^t}[\co_z:L_x\langle-m\rangle],\] and the induction step is established.
% \end{proof}

\subsection{Type $A$}\label{ssKLA}

In type $A$, the join-irreducibles in $(W,S)$ are exactly 
the bigrassmannians in $(W,S)$, see Lemma \ref{JIisbig}. 
The relevant Kazhdan-Lusztig polynomials are computed 
in \cite[Proposition 12]{kmm2} and a complete description 
of the socles of the cokernels of inclusions of Verma modules
is given in \cite{kmm2}.

\subsection{Type $B$}

Let $(W,S)$ be of type $B_{n+1}$. 
We use the following labeling of $S$:
\[\begin{tikzpicture}[scale=0.4,baseline=-3]
\protect\draw (4 cm,0) -- (2 cm,0);
\protect\draw (2 cm,0) -- (0 cm,0);
\protect\draw (0 cm,0) -- (-2 cm,0);
\protect\draw (-2 cm,0.1cm) -- (-4 cm,0.1 cm);
\protect\draw (-2 cm,-0.1cm) -- (-4 cm,-0.1 cm);
\protect\draw[fill=white] (4 cm, 0 cm) circle (.15cm) node[above=1pt]{\scriptsize $n$};
\protect\draw[fill=white] (2 cm, 0 cm) circle (0cm) node[above=1pt]{\scriptsize $\cdots$};
\protect\draw[fill=white] (0 cm, 0 cm) circle (.15cm) node[above=1pt]{\scriptsize $2$};
\protect\draw[fill=white] (-2 cm, 0 cm) circle (.15cm) node[above=1pt]{\scriptsize $1$};
\protect\draw[fill=white] (-4 cm, 0 cm) circle (.15cm) node[above=1pt]{\scriptsize $0$};
\end{tikzpicture}
\]

We have $\ell(w_0)=(n+1)^2$. The element $w_0$ is central, and the multiplication by $w_0$ gives rise to a natural bijection between the $\mathtt{H}$-cells in the small two-sided cell and the $\mathtt{H}$-cells in the penultimate two-sided cell $\jc$, which preserves diagonal $\mathtt{H}$-cells. Denote by $s_{ij} \in W$ the product $i \cdots j$ of simple reflections along the unique shortest path starting in $i$ and ending in $j$ in the Dynkin diagram. Put also
\begin{align}
\label{align:wu}
\begin{split}
     t_{ij} &:= s_{i0} \cdot 0 \cdot s_{0j} = i(i-1)\cdots 101 \cdots (j-1)j , \qquad \text{ if }  i \neq 0 \text{ or } j \neq 0, \\
     t_{00} &:=010, \\
    w_{ij} &:= s_{ij} \cdot w_0, \\ 
    u_{ij} &:= t_{ij} \cdot w_0 .
\end{split}
\end{align}

The $\mathtt{H}$-cell $\!^i \hc^j$ is equal to the two element set $\{u_{ij}, w_{ij}\}$, if both $i,j \neq 0$ or $i=j=0$, and to the singleton $\{u_{ij}=w_{ij}\}$, otherwise. We present the Bruhat graph of the penultimate cell $\jc$ in Figure \ref{fig:Bruhat_J_B}. There, the cell $\!^i \hc^j$ is the gray square placed in the $i$-th row and $j$-th column. The left cell $\mathcal{L}^j$ consists of all gray squares in the $j$-th column, and the right cell ${}^i\mathcal{R}$ consists of all gray squares in the $i$-th row.

\begin{figure}
    \centering
\begin{tikzpicture}[scale=.9, yscale=-1]
\fill [gray!20] (-0.45,-0.45) rectangle (1.45,1.45);
\fill [gray!20] (-0.45,1.55) rectangle (1.45,3.45);
\fill [gray!20] (-0.45,3.55) rectangle (1.45,5.45);
\fill [gray!20] (-0.45,5.55) rectangle (1.45,7.45);
\fill [gray!20] (-0.45,9.55) rectangle (1.45,11.45);
\fill [gray!20] (1.55,-0.45) rectangle (3.45,1.45);
\fill [gray!20] (1.55,1.55) rectangle (3.45,3.45);
\fill [gray!20] (1.55,3.55) rectangle (3.45,5.45);
\fill [gray!20] (1.55,5.55) rectangle (3.45,7.45);
\fill [gray!20] (1.55,9.55) rectangle (3.45,11.45);
\fill [gray!20] (3.55,-0.45) rectangle (5.45,1.45);
\fill [gray!20] (3.55,1.55) rectangle (5.45,3.45);
\fill [gray!20] (3.55,3.55) rectangle (5.45,5.45);
\fill [gray!20] (3.55,5.55) rectangle (5.45,7.45);
\fill [gray!20] (3.55,9.55) rectangle (5.45,11.45);
\fill [gray!20] (5.55,-0.45) rectangle (7.45,1.45);
\fill [gray!20] (5.55,1.55) rectangle (7.45,3.45);
\fill [gray!20] (5.55,3.55) rectangle (7.45,5.45);
\fill [gray!20] (5.55,5.55) rectangle (7.45,7.45);
\fill [gray!20] (5.55,9.55) rectangle (7.45,11.45);
\fill [gray!20] (9.55,-0.45) rectangle (11.45,1.45);
\fill [gray!20] (9.55,1.55) rectangle (11.45,3.45);
\fill [gray!20] (9.55,3.55) rectangle (11.45,5.45);
\fill [gray!20] (9.55,5.55) rectangle (11.45,7.45);
\fill [gray!20] (9.55,9.55) rectangle (11.45,11.45);
\newcommand{\arrow}[4]{\draw [->] (0.82*#1+0.17*#3,0.82*#2+0.17*#4) -- (0.17*#1+0.82*#3,0.17*#2+0.82*#4);}
\node[] at (0,0) {$u_{00}$};
\node[] at (1,1) {$w_{00}$};
\node[] at (2.50,0.5) {$u_{01}$};
\node[] at (4.50,0.5) {$u_{02}$};
\node[] at (6.50,0.5) {$u_{03}$};
\node[] at (0.5,2.50) {$u_{10}$};
\node[] at (0.5,4.50) {$u_{20}$};
\node[] at (0.5,6.50) {$u_{30}$};
\node[] at (2.00,2.00) {$u_{11}$};
\node[] at (3.00,3.00) {$w_{11}$};
\node[] at (2.00,4.00) {$u_{21}$};
\node[] at (3.00,5.00) {$w_{21}$};
\node[] at (2.00,6.00) {$u_{31}$};
\node[] at (3.00,7.00) {$w_{31}$};
\node[] at (4.00,2.00) {$u_{12}$};
\node[] at (5.00,3.00) {$w_{12}$};
\node[] at (4.00,4.00) {$u_{22}$};
\node[] at (5.00,5.00) {$w_{22}$};
\node[] at (4.00,6.00) {$u_{32}$};
\node[] at (5.00,7.00) {$w_{32}$};
\node[] at (6.00,2.00) {$u_{13}$};
\node[] at (7.00,3.00) {$w_{13}$};
\node[] at (6.00,4.00) {$u_{23}$};
\node[] at (7.00,5.00) {$w_{23}$};
\node[] at (6.00,6.00) {$u_{33}$};
\node[] at (7.00,7.00) {$w_{33}$};
\node[] at (8.50,0.5) {$\cdots$};
\node[] at (0.5,8.50) {$\vdots$};
\node[] at (8.00,2.00) {$\cdots$};
\node[] at (9.00,3.00) {$\cdots$};
\node[] at (8.00,4.00) {$\cdots$};
\node[] at (9.00,5.00) {$\cdots$};
\node[] at (8.00,6.00) {$\cdots$};
\node[] at (9.00,7.00) {$\cdots$};
\node[] at (2.00,8.00) {$\vdots$};
\node[] at (3.00,9.00) {$\vdots$};
\node[] at (4.00,8.00) {$\vdots$};
\node[] at (5.00,9.00) {$\vdots$};
\node[] at (6.00,8.00) {$\vdots$};
\node[] at (7.00,9.00) {$\vdots$};
\node[] at (8.00,8.00) {$\ddots$};
\node[] at (9.00,9.00) {$\ddots$};
\node[] at (10.50,0.5) {$u_{0n}$};
\node[] at (0.5,10.50) {$u_{n0}$};
\node[] at (10.00,2.00) {$u_{1n}$};
\node[] at (11.00,3.00) {$w_{1n}$};
\node[] at (10.00,4.00) {$u_{2n}$};
\node[] at (11.00,5.00) {$w_{2n}$};
\node[] at (10.00,6.00) {$u_{3n}$};
\node[] at (11.00,7.00) {$w_{3n}$};
\node[] at (2.00,10.00) {$u_{n1}$};
\node[] at (3.00,11.00) {$w_{n1}$};
\node[] at (4.00,10.00) {$u_{n2}$};
\node[] at (5.00,11.00) {$w_{n2}$};
\node[] at (6.00,10.00) {$u_{n3}$};
\node[] at (7.00,11.00) {$w_{n3}$};
\node[] at (8.00,10.00) {$\cdots$};
\node[] at (9.00,11.00) {$\cdots$};
\node[] at (10.00,8.00) {$\vdots$};
\node[] at (11.00,9.00) {$\vdots$};
\node[] at (10.00,10.00) {$u_{nn}$};
\node[] at (11.00,11.00) {$w_{nn}$};
\arrow{0}{0}{0.5}{2.5};
\arrow{0}{0}{2.5}{0.5};
\arrow{0.5}{2.5}{1}{1};
\arrow{2.5}{0.5}{1}{1};
\arrow{4.50}{0.5}{2.50}{0.5};
\arrow{6.50}{0.5}{4.50}{0.5};
\arrow{8.50}{0.5}{6.50}{0.5};
\arrow{10.50}{0.5}{8.50}{0.5};
\arrow{0.5}{4.50}{0.5}{2.50};
\arrow{0.5}{6.50}{0.5}{4.50};
\arrow{0.5}{8.50}{0.5}{6.50};
\arrow{0.5}{10.50}{0.5}{8.50};
\arrow{5.00}{3.00}{3.00}{3.00};
\arrow{7.00}{3.00}{5.00}{3.00};
\arrow{9.00}{3.00}{7.00}{3.00};
\arrow{11.00}{3.00}{9.00}{3.00};
\arrow{4.00}{2.00}{2.00}{2.00};
\arrow{6.00}{2.00}{4.00}{2.00};
\arrow{8.00}{2.00}{6.00}{2.00};
\arrow{10.00}{2.00}{8.00}{2.00};
\arrow{2.00}{2.00}{0.50}{2.50};
\arrow{0.50}{2.50}{3.00}{3.00};
\arrow{7.00}{5.00}{5.00}{5.00};
\arrow{9.00}{5.00}{7.00}{5.00};
\arrow{11.00}{5.00}{9.00}{5.00};
\arrow{3.00}{5.00}{5.00}{5.00};
\arrow{4.00}{4.00}{2.00}{4.00};
\arrow{6.00}{4.00}{4.00}{4.00};
\arrow{8.00}{4.00}{6.00}{4.00};
\arrow{10.00}{4.00}{8.00}{4.00};
\arrow{2.00}{4.00}{0.50}{4.50};
\arrow{0.50}{4.50}{3.00}{5.00};
\arrow{9.00}{7.00}{7.00}{7.00};
\arrow{11.00}{7.00}{9.00}{7.00};
\arrow{3.00}{7.00}{5.00}{7.00};
\arrow{5.00}{7.00}{7.00}{7.00};
\arrow{4.00}{6.00}{2.00}{6.00};
\arrow{6.00}{6.00}{4.00}{6.00};
\arrow{8.00}{6.00}{6.00}{6.00};
\arrow{10.00}{6.00}{8.00}{6.00};
\arrow{2.00}{6.00}{0.50}{6.50};
\arrow{0.50}{6.50}{3.00}{7.00};
\arrow{3.00}{5.00}{3.00}{3.00};
\arrow{3.00}{7.00}{3.00}{5.00};
\arrow{3.00}{9.00}{3.00}{7.00};
\arrow{3.00}{11.00}{3.00}{9.00};
\arrow{2.00}{4.00}{2.00}{2.00};
\arrow{2.00}{6.00}{2.00}{4.00};
\arrow{2.00}{8.00}{2.00}{6.00};
\arrow{2.00}{10.00}{2.00}{8.00};
\arrow{2.00}{2.00}{2.50}{0.50};
\arrow{2.50}{0.50}{3.00}{3.00};
\arrow{5.00}{7.00}{5.00}{5.00};
\arrow{5.00}{9.00}{5.00}{7.00};
\arrow{5.00}{11.00}{5.00}{9.00};
\arrow{5.00}{3.00}{5.00}{5.00};
\arrow{4.00}{4.00}{4.00}{2.00};
\arrow{4.00}{6.00}{4.00}{4.00};
\arrow{4.00}{8.00}{4.00}{6.00};
\arrow{4.00}{10.00}{4.00}{8.00};
\arrow{4.00}{2.00}{4.50}{0.50};
\arrow{4.50}{0.50}{5.00}{3.00};
\arrow{7.00}{9.00}{7.00}{7.00};
\arrow{7.00}{11.00}{7.00}{9.00};
\arrow{7.00}{3.00}{7.00}{5.00};
\arrow{7.00}{5.00}{7.00}{7.00};
\arrow{6.00}{4.00}{6.00}{2.00};
\arrow{6.00}{6.00}{6.00}{4.00};
\arrow{6.00}{8.00}{6.00}{6.00};
\arrow{6.00}{10.00}{6.00}{8.00};
\arrow{6.00}{2.00}{6.50}{0.50};
\arrow{6.50}{0.50}{7.00}{3.00};
\arrow{10.00}{4.00}{10.00}{2.00};
\arrow{11.00}{3.00}{11.00}{5.00};
\arrow{10.00}{6.00}{10.00}{4.00};
\arrow{11.00}{5.00}{11.00}{7.00};
\arrow{10.00}{8.00}{10.00}{6.00};
\arrow{11.00}{7.00}{11.00}{9.00};
\arrow{10.00}{10.00}{10.00}{8.00};
\arrow{11.00}{9.00}{11.00}{11.00};
\arrow{2.00}{10.00}{0.50}{10.50};
\arrow{0.50}{10.50}{3.00}{11.00};
\arrow{4.00}{10.00}{2.00}{10.00};
\arrow{3.00}{11.00}{5.00}{11.00};
\arrow{6.00}{10.00}{4.00}{10.00};
\arrow{5.00}{11.00}{7.00}{11.00};
\arrow{8.00}{10.00}{6.00}{10.00};
\arrow{7.00}{11.00}{9.00}{11.00};
\arrow{10.00}{10.00}{8.00}{10.00};
\arrow{9.00}{11.00}{11.00}{11.00};
\arrow{10.00}{2.00}{10.50}{0.50};
\arrow{10.50}{0.50}{11.00}{3.00};
\end{tikzpicture}
    \caption{Bruhat graph of the penultimate two-sided cell in type $B_{n+1}$.}
    \label{fig:Bruhat_J_B}
\end{figure}

\subsubsection{Some Kazhdan-Lusztig computation in type B}\label{ssKLB}

Note that $u_{nn} \in {}^n\hc^n$ is the maximal element in the parabolic subgroup generated by $I=\{0,1,\cdots,n-1\}\subset S$. From this, we have 
$$\mathbf{a}(\mathcal{J}) = \mathbf{a}(u_{nn})=\ell(u_{nn})=\ell(w_0)-2n-1 = n^2.$$

\begin{proposition}\label{outerB}
Let $y\in \!^i\hc^j\subset\jc$. If $i=n$ or $j=n$, then $p_{e,y}= v^{\ell(y)}$.
\end{proposition}
\begin{proof}
The equality $p_{e,u_{nn}} = v^{n^2}$ follows from the fact that $u_{nn}$ is the maximal element in a parabolic subgroup. Moreover, we have $p_{e,u_{n-1,n}} = v^{n^2+1}$, because of the formulae
$\ell(u_{n-1,n})=n^2+1$ and $\mathbf{a}(u_{n-1,n})=n^2$, Equation~\eqref{pdegrees} and the parity condition.
From Lemma \ref{prop2p}, we get a recursion $(v+v^{-1})p_{e,b} = p_{e,a} + p_{e,c}$, for any consecutive $a \to b \to c$ in the chain  \[ u_{nn} \to u_{n-1,n} \to \ldots \to u_{0n} \to w_{1n} \to w_{2n} \to \ldots \to w_{nn} \]
(see Figure \ref{fig:Bruhat_J_B}).

By a two-step induction, we get the claim of the proposition in the case $j=n$. In the case $i=n$, the claim follows by recalling that $\left( \mathcal{L}^n \right)^{-1} = {}^n\mathcal{R}$, and that $p_{e,w} = p_{e,w^{-1}}$.
\end{proof}

%Therefore, we are in a situation where \eqref{2p} determines $p_{e,y}$ inductively. 

\begin{proposition}\label{Bbut1}
Let $y\in \!^i\hc^j\subseteq \jc$ be such that $(i,j)\neq (0,0)$. 
Then we have 
\begin{equation}\label{want}
    p_{e,y}=v^{\ell(y)}+v^{\ell(y)-2}+\cdots + v^{\ell(y)-2d(y)},
\quad\text{ where }\quad    
d(y):=\operatorname{min}(n-i,n-j).
\end{equation}
\end{proposition}
\begin{proof}
The case $d(y)=0$ is considered in Proposition \ref{outerB}, which gives us the basis of the induction.

Take $y \in {}^i\mathcal{H}^j$ with $d:=d(y) \in \{1,\ldots,n-1\}$, and assume that the assertion of the proposition is true for all $y' \in \mathcal{J}$, for which $d(y')<d$. Assume that $i=n-d$ and $j \leq n-d$ (the opposite case will follow by symmetry). Note that $y':= (i+1) \cdot y \in {}^{i+1}\mathcal{H}^j$ (see Equation (\ref{align:wu}) and Figure \ref{fig:Bruhat_J_B}). Applying Lemma \ref{prop2p} to $y'$ (observe that $i+1 \neq j$), we get
\[ (v+v^{-1} ) \cdot p_{e,y'} = \begin{cases} p_{e,y}, & d=2; \\ p_{e,y} + p_{e,y''},  & d \geq 2; \end{cases} \]
for $y'':=(i+2) \cdot y' \in {}^{i+2}\mathcal{H}^j$.
Applying the inductive assumption to $p_{e,y'}$ and $p_{e,y''}$, we get that $p_{e,y}$ has the desired form.
%
\begin{comment} % Older version of the proof
The proof of induction steps within the $(n,n)$-minor consisting of $\hc$ with 2 elements is identical to that in type A (see \cite[Proposition 11]{kmm2}).
We are left with the cases $s=0$ or $t=0$.
We consider the case $t=0$, that is, $y\in \lc^0$. The other case follows by symmetry.
Let $s>0$. 
We consider the neighboring elements $y0,y1\in \!^s\hc^{1}$.
We have $d(y0)=d(y1)=d(y)=n-s$ and the claim
\eqref{want} is established for both $y0$ and $y1$ above.
Considering Lemma \ref{prop2p} and Proposition \ref{muJ}), we get 
\[(v+v\inv)p_{e,y}=p_{e,y1}+p_{e,y0}.\]
The claim follows.
\end{comment}
\end{proof}

We now consider the remaining cases, namely the two-sided cell ${}^0\mathcal{H}^0$, which has the following two elements: $w_{00}=0 \, w_0$ and $u_{00}=010 \, w_0$
(note that here the symbols $0$ and $1$ denote the simple reflections corresponding to the respective verticies of the Dynkin diagram).

\begin{proposition}\label{specialB}
We have 
\begin{equation}\label{sumB}
p_{e,w_{00}}+p_{e,u_{00}}=v^{\ell(w_0)-1}+v^{\ell(w_0)-3}+ \cdots +v^{\ell(w_0)-2n-1}.    
\end{equation}
More specifically, if $n+1$ is odd, then $w_{00}$ is the Duflo element in
${}^0\hc^0$ and we have \begin{equation}
\label{eq:odd}
    \begin{split}
        p_{e,w_{00}}=v^{\ell(w_0)-1}+v^{\ell(w_0)-5}+\cdots +v^{\ell(w_0)-2n-1}\\
        p_{e,u_{00}}=v^{\ell(w_0)-3}+v^{\ell(w_0)-7}+\cdots +v^{\ell(w_0)-2n+1};
    \end{split}
\end{equation}
%n=2m+1 then first has m+1 terms second m terms
if $n+1$ is even, then $u_{00}$ is the Duflo element in $\!^0\hc^0$ and we have 
\begin{equation}
\label{eq:even}
    \begin{split}
        p_{e,w_{00}}=v^{\ell(w_0)-1}+v^{\ell(w_0)-5}+\cdots +v^{\ell(w_0)-2n+1}\\
        p_{e,u_{00}}=v^{\ell(w_0)-3}+v^{\ell(w_0)-7}+\cdots +v^{\ell(w_0)-2n-1}.
    \end{split}
\end{equation}%if n= 2m then both have m terms
\end{proposition}
\begin{proof}
Applying Lemma \ref{prop2p} to $y=u_{10}$ and $sy=w_{00}$ (see also Equation (\ref{align:wu}) and Figure \ref{fig:Bruhat_J_B}), we get the equation
\[ (v+v^{-1}) \cdot p_{e,u_{10}} = p_{e,w_{00}} + p_{e,u_{00}} +p_{e,u_{20}} . \]
Plugging in $p_{e,u_{10}}$ and $p_{e,u_{20}}$ from Proposition \ref{Bbut1}, we obtain \eqref{sumB}.

Note that $x\in\jc$ is a Duflo element if and only if $p_{e,x}$ contains the term $v^{\ell(w_0)-2n-1}$, so it remains to prove Formulae~(\ref{eq:odd}) and (\ref{eq:even}).

Applying Lemma \ref{prop2p} to $y=w_{00}$, $s=0$, and, respectively, to $y=u_{00}$, $s=0$, we get:
\begin{align}
\label{al:pw}   & v \cdot p_{e,w_{00}} + p_{0,w_{00}} = v^{\ell(w_0)} + v^{\ell(w_0)-2} + v^{\ell(w_0)-4} + \cdots + v^{\ell(w_0)-2n} ,   \\
\label{al:pu}   & v \cdot p_{e,u_{00}} + p_{0,u_{00}} = \phantom{v^{\ell(w_0)} +} \ \ v^{\ell(w_0)-2} + v^{\ell(w_0)-4} + \cdots + v^{\ell(w_0)-2n} .
\end{align}
From (\ref{eq:p_Delta}) and the fact that $\Delta_s\langle -1 \rangle\subset \Delta_e$, it follows that, for $s \in S$, $x \in W$, the polynomial $p_{e,x} - v\cdot p_{s,x}$ has non-negative coefficients. In particular, we have
\begin{equation}
\label{eq:p_s<p_e}
p_{s,x}(1) \leq  p_{e,x}(1).
\end{equation}
%
%NOT NEEDED Moreover, from the multiplicity freeness in Proposition \ref{Bbut1} and (\ref{sumB}), it follows that all the coefficients in $p_{e,x} - v\cdot p_{s,x}$ are equal to $1$.

Assume that $n=2k$, the other case being similar and therefore omitted. By evaluating $v=1$ in (\ref{al:pw}) and (\ref{al:pu}), and using (\ref{eq:p_s<p_e}), it follows that $p_{e,w_{00}}(1) \geq k+1$, and that $p_{e,u_{00}}(1) \geq k$. Combining this with (\ref{sumB}), we see that $p_{e,w_{00}}(1)=k+1$, and that $p_{e,u_{00}}(1)=k$. Going back to (\ref{al:pu}), we see that $p_{0,u_{00}}(1)=k$. So the polynomial $p_{e,u_{00}} - v\cdot p_{0,u_{00}}$, with non-negative coefficients, vanishes at $v=1$. It follows that $p_{e,u_{00}} = v\cdot p_{0,u_{00}}$. This makes (\ref{al:pu}) into
\[ (v + v^{-1})\cdot p_{e,u_{00}} = v^{\ell(w_0)-2} + v^{\ell(w_0)-4} + \cdots + v^{\ell(w_0)-4k},   \]
which, in turn, gives (\ref{eq:odd}).
%
%\hk{Reminder: There is an argument using twisting functors.}
%
\end{proof}

We visualize the results in Figure \ref{fig:notoctahedron} by depicting each simple subquotient $L_{w}\langle - k\rangle$ of $\Delta_e$, $w \in {}^i\mathcal{H}^j$, as $(i,j,k) \in \mathbb{Z}^3$. If $|{}^i\mathcal{H}^j|=1$ or if $w$ is the shorter one of the two elements there, we denote this point by $\begin{tikzpicture}
\draw[fill=white, line width=0.2mm] (0,0,0) circle (2pt);
\end{tikzpicture}$; otherwise we denote it by $\begin{tikzpicture}
\draw (1,1,0) node[cross=2pt, thick]{};
\end{tikzpicture}$. Note that both $\begin{tikzpicture}
\draw[fill=white, line width=0.2mm] (0,0,0) circle (2pt);
\end{tikzpicture}$ and $\begin{tikzpicture}
\draw (1,1,0) node[cross=2pt, thick]{};
\end{tikzpicture}$ can appear at the same coordinate. The bottom of the picture consists of the points that come from the elements $w_{ii} = i \cdot w_0$, for $i \in S$, and the top of the picture comes from the Duflo elements.
\begin{figure}[ht]
    \centering
\begin{minipage}{0.23\textwidth}
\tdplotsetmaincoords{80}{120}
\begin{tikzpicture}[tdplot_main_coords, scale=1]
\draw[dashed] (0, 0, 0) -- (0, 0, 2);
\draw[dashed] (0, 0, 0) -- (0, 1, 1);
\draw[dashed] (0, 0, 2) -- (0, 1, 1);
\draw[] (0, 0, 0) -- (1, 1, 0);
\draw[] (0, 0, 2) -- (1, 1, 2);
\draw[] (0, 1, 1) -- (1, 0, 1);
\draw[] (0, 0, 0) -- (1, 0, 1);
\draw[] (0, 0, 2) -- (1, 0, 1);
\draw[] (1, 1, 0) -- (1, 0, 1);
\draw[] (1, 1, 2) -- (1, 0, 1);
\draw[] (1, 1, 0) -- (0, 1, 1);
\draw[] (1, 1, 2) -- (0, 1, 1);
\draw[fill=white, line width=0.2mm] (0,0,2) circle (2pt) node[anchor=west] {};
\draw (0,0,0) node[cross=2pt, thick]{};
\draw[fill=white, line width=0.2mm] (0,1,1) circle (2pt) node[anchor=west] {};
\draw[fill=white, line width=0.2mm] (1,0,1) circle (2pt) node[anchor=west] {};
\draw[fill=white, line width=0.2mm] (1,1,2) circle (2pt) node[anchor=west] {};
\draw (1,1,0) node[cross=2pt, thick]{};
\node[left=2pt] at (0,0,0) {{\scriptsize $(0,0,3)$}};
\node[above=2pt] at (0,0,2) {{\scriptsize $(0,0,1)$}};
\node[below=2pt] at (1,1,0) {{\scriptsize $(1,1,3)$}};
\node[right=2pt] at (1,1,2) {{\scriptsize $(1,1,1)$}};
\node[left] at (1,0,1) {{\scriptsize $(1,0,2)$}};
\node[right] at (0,1,1) {{\scriptsize $(0,1,2)$}};
\end{tikzpicture}
\end{minipage}
\begin{minipage}{0.32\textwidth}
\tdplotsetmaincoords{80}{120}
\begin{tikzpicture}[tdplot_main_coords, scale=1]
\draw[dashed] (0, 0, 0) -- (0, 0, 4);
\draw[dashed] (0, 0, 0) -- (0, 2, 2);
\draw[dashed] (0, 0, 4) -- (0, 2, 2);
\draw[] (0, 0, 0) -- (2, 2, 0);
\draw[] (0, 0, 4) -- (2, 2, 4);
\draw[] (0, 2, 2) -- (2, 0, 2);
\draw[] (0, 0, 0) -- (2, 0, 2);
\draw[] (0, 0, 4) -- (2, 0, 2);
\draw[] (2, 2, 0) -- (2, 0, 2);
\draw[] (2, 2, 4) -- (2, 0, 2);
\draw[] (2, 2, 0) -- (0, 2, 2);
\draw[] (2, 2, 4) -- (0, 2, 2);
\draw (0,0,4) node[cross=2pt, thick]{};
\draw[fill=white, line width=0.2mm] (0,0,2) circle (2pt) node[anchor=west] {};
\draw (0,0,0) node[cross=2pt, thick]{};
\draw[fill=white, line width=0.2mm] (0,1,3) circle (2pt) node[anchor=west] {};
\draw[fill=white, line width=0.2mm] (0,1,1) circle (2pt) node[anchor=west] {};
\draw[fill=white, line width=0.2mm] (0,2,2) circle (2pt) node[anchor=west] {};
\draw[fill=white, line width=0.2mm] (1,0,3) circle (2pt) node[anchor=west] {};
\draw[fill=white, line width=0.2mm] (1,0,1) circle (2pt) node[anchor=west] {};
\draw[fill=white, line width=0.2mm] (1,1,4) circle (2pt) node[anchor=west] {};
\draw[fill=white, line width=0.2mm] (1,1,2) circle (2pt) node[anchor=west] {};
\draw (1,1,2) node[cross=2pt, thick]{};
\draw (1,1,0) node[cross=2pt, thick]{};
\draw[fill=white, line width=0.2mm] (1,2,3) circle (2pt) node[anchor=west] {};
\draw (1,2,1) node[cross=2pt, thick]{};
\draw[fill=white, line width=0.2mm] (2,0,2) circle (2pt) node[anchor=west] {};
\draw[fill=white, line width=0.2mm] (2,1,3) circle (2pt) node[anchor=west] {};
\draw (2,1,1) node[cross=2pt, thick]{};
\draw[fill=white, line width=0.2mm] (2,2,4) circle (2pt) node[anchor=west] {};
\draw (2,2,0) node[cross=2pt, thick]{};
\node[left=2pt] at (0,0,0) {{\scriptsize $(0,0,8)$}};
\node[above=2pt] at (0,0,4) {{\scriptsize $(0,0,4)$}};
\node[below=2pt] at (2,2,0) {{\scriptsize $(2,2,8)$}};
\node[right=2pt] at (2,2,4) {{\scriptsize $(2,2,4)$}};
\node[left] at (2,0,2) {{\scriptsize $(2,0,6)$}};
\node[right] at (0,2,2) {{\scriptsize $(0,2,6)$}};
\end{tikzpicture}
\end{minipage}
\begin{minipage}{0.40\textwidth}
\tdplotsetmaincoords{80}{120}
\begin{tikzpicture}[tdplot_main_coords, scale=1]
\draw[dashed] (0, 0, 0) -- (0, 0, 6);
\draw[dashed] (0, 0, 0) -- (0, 3, 3);
\draw[dashed] (0, 0, 6) -- (0, 3, 3);
\draw[] (0, 0, 0) -- (3, 3, 0);
\draw[] (0, 0, 6) -- (3, 3, 6);
\draw[] (0, 3, 3) -- (3, 0, 3);
\draw[] (0, 0, 0) -- (3, 0, 3);
\draw[] (0, 0, 6) -- (3, 0, 3);
\draw[] (3, 3, 0) -- (3, 0, 3);
\draw[] (3, 3, 6) -- (3, 0, 3);
\draw[] (3, 3, 0) -- (0, 3, 3);
\draw[] (3, 3, 6) -- (0, 3, 3);
\draw[fill=white, line width=0.2mm] (0,0,6) circle (2pt) node[anchor=west] {};
\draw (0,0,4) node[cross=2pt, thick]{};
\draw[fill=white, line width=0.2mm] (0,0,2) circle (2pt) node[anchor=west] {};
\draw (0,0,0) node[cross=2pt, thick]{};
\draw[fill=white, line width=0.2mm] (0,1,5) circle (2pt) node[anchor=west] {};
\draw[fill=white, line width=0.2mm] (0,1,3) circle (2pt) node[anchor=west] {};
\draw[fill=white, line width=0.2mm] (0,1,1) circle (2pt) node[anchor=west] {};
\draw[fill=white, line width=0.2mm] (0,2,4) circle (2pt) node[anchor=west] {};
\draw[fill=white, line width=0.2mm] (0,2,2) circle (2pt) node[anchor=west] {};
\draw[fill=white, line width=0.2mm] (0,3,3) circle (2pt) node[anchor=west] {};
\draw[fill=white, line width=0.2mm] (1,0,5) circle (2pt) node[anchor=west] {};
\draw[fill=white, line width=0.2mm] (1,0,3) circle (2pt) node[anchor=west] {};
\draw[fill=white, line width=0.2mm] (1,0,1) circle (2pt) node[anchor=west] {};
\draw[fill=white, line width=0.2mm] (1,1,6) circle (2pt) node[anchor=west] {};
\draw[fill=white, line width=0.2mm] (1,1,4) circle (2pt) node[anchor=west] {};
\draw[fill=white, line width=0.2mm] (1,1,2) circle (2pt) node[anchor=west] {};
\draw (1,1,4) node[cross=2pt, thick]{};
\draw (1,1,2) node[cross=2pt, thick]{};
\draw (1,1,0) node[cross=2pt, thick]{};
\draw[fill=white, line width=0.2mm] (1,2,5) circle (2pt) node[anchor=west] {};
\draw[fill=white, line width=0.2mm] (1,2,3) circle (2pt) node[anchor=west] {};
\draw (1,2,3) node[cross=2pt, thick]{};
\draw (1,2,1) node[cross=2pt, thick]{};
\draw[fill=white, line width=0.2mm] (1,3,4) circle (2pt) node[anchor=west] {};
\draw (1,3,2) node[cross=2pt, thick]{};
\draw[fill=white, line width=0.2mm] (2,0,4) circle (2pt) node[anchor=west] {};
\draw[fill=white, line width=0.2mm] (2,0,2) circle (2pt) node[anchor=west] {};
\draw[fill=white, line width=0.2mm] (2,1,5) circle (2pt) node[anchor=west] {};
\draw[fill=white, line width=0.2mm] (2,1,3) circle (2pt) node[anchor=west] {};
\draw (2,1,3) node[cross=2pt, thick]{};
\draw (2,1,1) node[cross=2pt, thick]{};
\draw[fill=white, line width=0.2mm] (2,2,6) circle (2pt) node[anchor=west] {};
\draw[fill=white, line width=0.2mm] (2,2,4) circle (2pt) node[anchor=west] {};
\draw (2,2,2) node[cross=2pt, thick]{};
\draw (2,2,0) node[cross=2pt, thick]{};
\draw[fill=white, line width=0.2mm] (2,3,5) circle (2pt) node[anchor=west] {};
\draw (2,3,1) node[cross=2pt, thick]{};
\draw[fill=white, line width=0.2mm] (3,0,3) circle (2pt) node[anchor=west] {};
\draw[fill=white, line width=0.2mm] (3,1,4) circle (2pt) node[anchor=west] {};
\draw (3,1,2) node[cross=2pt, thick]{};
\draw[fill=white, line width=0.2mm] (3,2,5) circle (2pt) node[anchor=west] {};
\draw (3,2,1) node[cross=2pt, thick]{};
\draw[fill=white, line width=0.2mm] (3,3,6) circle (2pt) node[anchor=west] {};
\draw (3,3,0) node[cross=2pt, thick]{};
\node[left=2pt] at (0,0,0) {{\scriptsize $(0,0,15)$}};
\node[above=2pt] at (0,0,6) {{\scriptsize $(0,0,9)$}};
\node[below=2pt] at (3,3,0) {{\scriptsize $(3,3,15)$}};
\node[right=2pt] at (3,3,6) {{\scriptsize $(3,3,9)$}};
\node[left] at (3,0,3) {{\scriptsize $(3,0,12)$}};
\node[right] at (0,3,3) {{\scriptsize $(0,3,12)$}};
\end{tikzpicture}
\end{minipage}
\begin{minipage}{0.49\textwidth}
\tdplotsetmaincoords{80}{120}
\begin{tikzpicture}[tdplot_main_coords, scale=1]
\draw[dashed] (0, 0, 0) -- (0, 0, 8);
\draw[dashed] (0, 0, 0) -- (0, 4, 4);
\draw[dashed] (0, 0, 8) -- (0, 4, 4);
\draw[] (0, 0, 0) -- (4, 4, 0);
\draw[] (0, 0, 8) -- (4, 4, 8);
\draw[] (0, 4, 4) -- (4, 0, 4);
\draw[] (0, 0, 0) -- (4, 0, 4);
\draw[] (0, 0, 8) -- (4, 0, 4);
\draw[] (4, 4, 0) -- (4, 0, 4);
\draw[] (4, 4, 8) -- (4, 0, 4);
\draw[] (4, 4, 0) -- (0, 4, 4);
\draw[] (4, 4, 8) -- (0, 4, 4);
\draw (0,0,8) node[cross=2pt, thick]{};
\draw[fill=white, line width=0.2mm] (0,0,6) circle (2pt) node[anchor=west] {};
\draw (0,0,4) node[cross=2pt, thick]{};
\draw[fill=white, line width=0.2mm] (0,0,2) circle (2pt) node[anchor=west] {};
\draw (0,0,0) node[cross=2pt, thick]{};
\draw[fill=white, line width=0.2mm] (0,1,7) circle (2pt) node[anchor=west] {};
\draw[fill=white, line width=0.2mm] (0,1,5) circle (2pt) node[anchor=west] {};
\draw[fill=white, line width=0.2mm] (0,1,3) circle (2pt) node[anchor=west] {};
\draw[fill=white, line width=0.2mm] (0,1,1) circle (2pt) node[anchor=west] {};
\draw[fill=white, line width=0.2mm] (0,2,6) circle (2pt) node[anchor=west] {};
\draw[fill=white, line width=0.2mm] (0,2,4) circle (2pt) node[anchor=west] {};
\draw[fill=white, line width=0.2mm] (0,2,2) circle (2pt) node[anchor=west] {};
\draw[fill=white, line width=0.2mm] (0,3,5) circle (2pt) node[anchor=west] {};
\draw[fill=white, line width=0.2mm] (0,3,3) circle (2pt) node[anchor=west] {};
\draw[fill=white, line width=0.2mm] (0,4,4) circle (2pt) node[anchor=west] {};
\draw[fill=white, line width=0.2mm] (1,0,7) circle (2pt) node[anchor=west] {};
\draw[fill=white, line width=0.2mm] (1,0,5) circle (2pt) node[anchor=west] {};
\draw[fill=white, line width=0.2mm] (1,0,3) circle (2pt) node[anchor=west] {};
\draw[fill=white, line width=0.2mm] (1,0,1) circle (2pt) node[anchor=west] {};
\draw[fill=white, line width=0.2mm] (1,1,8) circle (2pt) node[anchor=west] {};
\draw[fill=white, line width=0.2mm] (1,1,6) circle (2pt) node[anchor=west] {};
\draw[fill=white, line width=0.2mm] (1,1,4) circle (2pt) node[anchor=west] {};
\draw[fill=white, line width=0.2mm] (1,1,2) circle (2pt) node[anchor=west] {};
\draw (1,1,6) node[cross=2pt, thick]{};
\draw (1,1,4) node[cross=2pt, thick]{};
\draw (1,1,2) node[cross=2pt, thick]{};
\draw (1,1,0) node[cross=2pt, thick]{};
\draw[fill=white, line width=0.2mm] (1,2,7) circle (2pt) node[anchor=west] {};
\draw[fill=white, line width=0.2mm] (1,2,5) circle (2pt) node[anchor=west] {};
\draw[fill=white, line width=0.2mm] (1,2,3) circle (2pt) node[anchor=west] {};
\draw (1,2,5) node[cross=2pt, thick]{};
\draw (1,2,3) node[cross=2pt, thick]{};
\draw (1,2,1) node[cross=2pt, thick]{};
\draw[fill=white, line width=0.2mm] (1,3,6) circle (2pt) node[anchor=west] {};
\draw[fill=white, line width=0.2mm] (1,3,4) circle (2pt) node[anchor=west] {};
\draw (1,3,4) node[cross=2pt, thick]{};
\draw (1,3,2) node[cross=2pt, thick]{};
\draw[fill=white, line width=0.2mm] (1,4,5) circle (2pt) node[anchor=west] {};
\draw (1,4,3) node[cross=2pt, thick]{};
\draw[fill=white, line width=0.2mm] (2,0,6) circle (2pt) node[anchor=west] {};
\draw[fill=white, line width=0.2mm] (2,0,4) circle (2pt) node[anchor=west] {};
\draw[fill=white, line width=0.2mm] (2,0,2) circle (2pt) node[anchor=west] {};
\draw[fill=white, line width=0.2mm] (2,1,7) circle (2pt) node[anchor=west] {};
\draw[fill=white, line width=0.2mm] (2,1,5) circle (2pt) node[anchor=west] {};
\draw[fill=white, line width=0.2mm] (2,1,3) circle (2pt) node[anchor=west] {};
\draw (2,1,5) node[cross=2pt, thick]{};
\draw (2,1,3) node[cross=2pt, thick]{};
\draw (2,1,1) node[cross=2pt, thick]{};
\draw[fill=white, line width=0.2mm] (2,2,8) circle (2pt) node[anchor=west] {};
\draw[fill=white, line width=0.2mm] (2,2,6) circle (2pt) node[anchor=west] {};
\draw[fill=white, line width=0.2mm] (2,2,4) circle (2pt) node[anchor=west] {};
\draw (2,2,4) node[cross=2pt, thick]{};
\draw (2,2,2) node[cross=2pt, thick]{};
\draw (2,2,0) node[cross=2pt, thick]{};
\draw[fill=white, line width=0.2mm] (2,3,7) circle (2pt) node[anchor=west] {};
\draw[fill=white, line width=0.2mm] (2,3,5) circle (2pt) node[anchor=west] {};
\draw (2,3,3) node[cross=2pt, thick]{};
\draw (2,3,1) node[cross=2pt, thick]{};
\draw[fill=white, line width=0.2mm] (2,4,6) circle (2pt) node[anchor=west] {};
\draw (2,4,2) node[cross=2pt, thick]{};
\draw[fill=white, line width=0.2mm] (3,0,5) circle (2pt) node[anchor=west] {};
\draw[fill=white, line width=0.2mm] (3,0,3) circle (2pt) node[anchor=west] {};
\draw[fill=white, line width=0.2mm] (3,1,6) circle (2pt) node[anchor=west] {};
\draw[fill=white, line width=0.2mm] (3,1,4) circle (2pt) node[anchor=west] {};
\draw (3,1,4) node[cross=2pt, thick]{};
\draw (3,1,2) node[cross=2pt, thick]{};
\draw[fill=white, line width=0.2mm] (3,2,7) circle (2pt) node[anchor=west] {};
\draw[fill=white, line width=0.2mm] (3,2,5) circle (2pt) node[anchor=west] {};
\draw (3,2,3) node[cross=2pt, thick]{};
\draw (3,2,1) node[cross=2pt, thick]{};
\draw[fill=white, line width=0.2mm] (3,3,8) circle (2pt) node[anchor=west] {};
\draw[fill=white, line width=0.2mm] (3,3,6) circle (2pt) node[anchor=west] {};
\draw (3,3,2) node[cross=2pt, thick]{};
\draw (3,3,0) node[cross=2pt, thick]{};
\draw[fill=white, line width=0.2mm] (3,4,7) circle (2pt) node[anchor=west] {};
\draw (3,4,1) node[cross=2pt, thick]{};
\draw[fill=white, line width=0.2mm] (4,0,4) circle (2pt) node[anchor=west] {};
\draw[fill=white, line width=0.2mm] (4,1,5) circle (2pt) node[anchor=west] {};
\draw (4,1,3) node[cross=2pt, thick]{};
\draw[fill=white, line width=0.2mm] (4,2,6) circle (2pt) node[anchor=west] {};
\draw (4,2,2) node[cross=2pt, thick]{};
\draw[fill=white, line width=0.2mm] (4,3,7) circle (2pt) node[anchor=west] {};
\draw (4,3,1) node[cross=2pt, thick]{};
\draw[fill=white, line width=0.2mm] (4,4,8) circle (2pt) node[anchor=west] {};
\draw (4,4,0) node[cross=2pt, thick]{};
\node[left=2pt] at (0,0,0) {{\scriptsize $(0,0,24)$}};
\node[above=2pt] at (0,0,8) {{\scriptsize $(0,0,16)$}};
\node[below=2pt] at (4,4,0) {{\scriptsize $(4,4,24)$}};
\node[right=2pt] at (4,4,8) {{\scriptsize $(4,4,16)$}};
\node[left] at (4,0,4) {{\scriptsize $(4,0,20)$}};
\node[right] at (0,4,4) {{\scriptsize $(0,4,20)$}};
\end{tikzpicture}
\end{minipage}
    \caption{Composition factors of $\Delta_e$ from the penultimate cell, for $n = 1, 2, 3, 4$.}
    \label{fig:notoctahedron}
\end{figure}
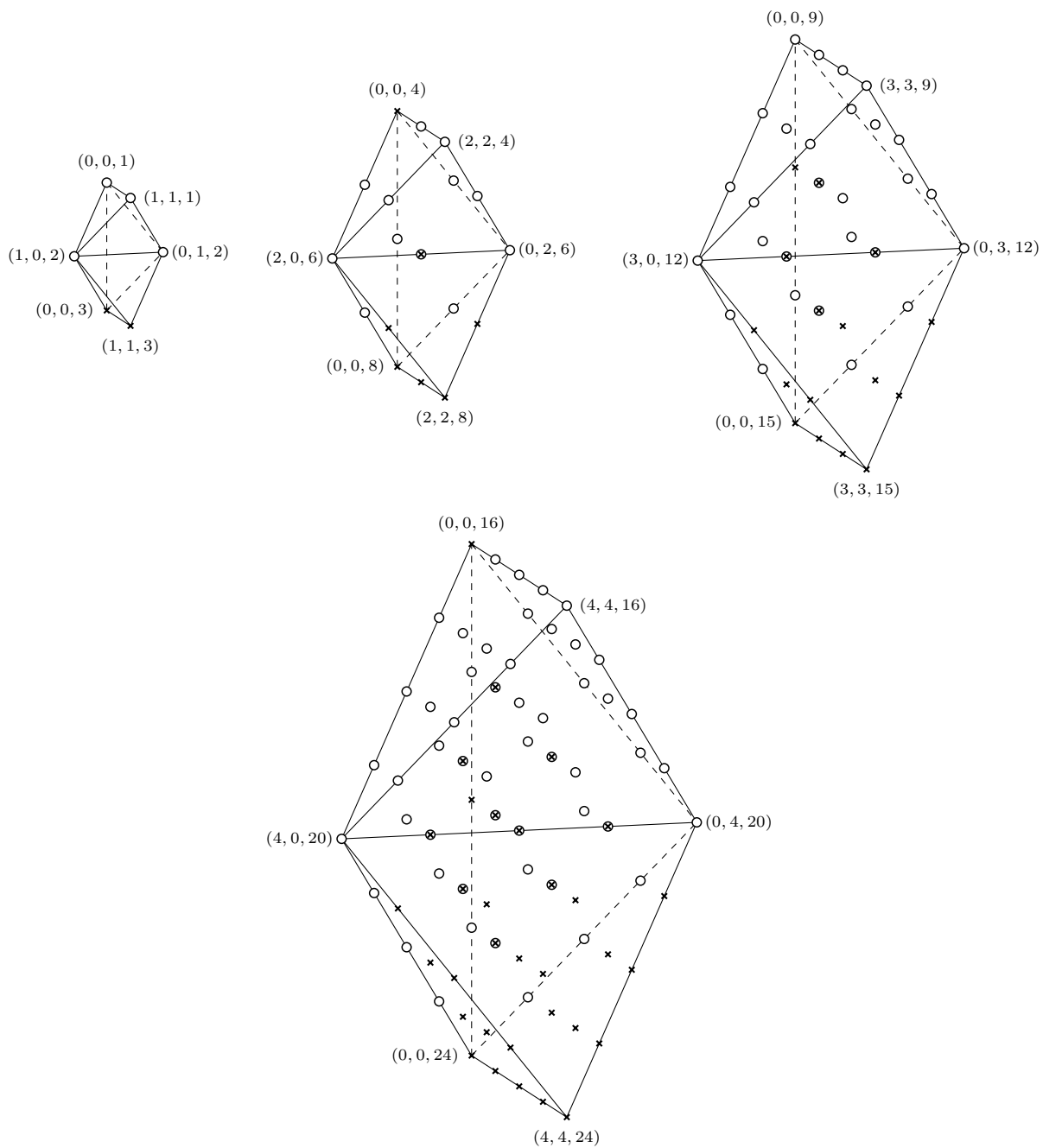 

%by the following rule:
%
%\begin{align*}
%    L_{u_{ij}}\langle - k\rangle &\mapsto (i,j,k), \ \text{ denoted by } \begin{tikzpicture} \draw[fill=white, line width=0.2mm] (0,0,0) circle (2pt) node[anchor=west] {}; \end{tikzpicture}; \\
%    L_{w_{ij}}\langle - k\rangle &  \mapsto \begin{cases}    (i,j,k), \ \text{ denoted by } \begin{tikzpicture} \draw (3,4,1) node[cross=2pt, thick]{}; \end{tikzpicture} & \colon i=j=0; \\
%
%    (n+1-i,n+1-j,k), \ \text{ denoted by } \begin{tikzpicture} \draw (3,4,1) node[cross=2pt, thick]{}; \end{tikzpicture} & \colon \text{otherwise}.
%    \end{cases}
%\end{align*}
%
\begin{remark}\label{octahedron}

%It follows from (\ref{want}) and (\ref{sumB}) that 
The number of graded composition factors in $\Delta_e$ isomorphic to $L_w$, for $w\in \jc$, adds up to the octahedral number $\frac{(2n^2+4n+3)(n+1)}{3}$. %This sequence is given in \cite[A005900]{OEIS}, under the name ``Octahedral numbers''. 
If we depict $L_{w_{ij}}\langle - k\rangle$, for $i,j>0 $, as a point with 
the coordinates  $(n+1-i,n+1-j,k)$, we indeed get a ``combinatorial'' octahedron inscribed in the cuboid $[0,n] \times [0,n] \times [n^2,n^2+2n]$. We present it in Figure \ref{fig:octahedron} for small ranks.
\begin{figure}[ht]
    \centering
\begin{minipage}{0.22\textwidth}
\tdplotsetmaincoords{80}{120}
\begin{tikzpicture}[tdplot_main_coords, scale=1]
\draw[dashed] (0, 0, 0) -- (0, 0, 2);
\draw[dashed] (0, 0, 0) -- (0, 1, 1);
\draw[dashed] (0, 0, 2) -- (0, 1, 1);
\draw[] (0, 0, 0) -- (1, 1, 0);
\draw[] (0, 0, 2) -- (1, 1, 2);
\draw[] (1, 1, 0) -- (1, 1, 2);
\draw[] (0, 0, 0) -- (1, 0, 1);
\draw[] (0, 0, 2) -- (1, 0, 1);
\draw[] (1, 1, 0) -- (1, 0, 1);
\draw[] (1, 1, 2) -- (1, 0, 1);
\draw[] (1, 1, 0) -- (0, 1, 1);
\draw[] (1, 1, 2) -- (0, 1, 1);
\draw[fill=white, line width=0.2mm] (0,0,2) circle (2pt) node[anchor=west] {};
\draw (0,0,0) node[cross=2pt, thick]{};
\draw[fill=white, line width=0.2mm] (0,1,1) circle (2pt) node[anchor=west] {};
\draw[fill=white, line width=0.2mm] (1,0,1) circle (2pt) node[anchor=west] {};
\draw[fill=white, line width=0.2mm] (1,1,2) circle (2pt) node[anchor=west] {};
\draw (1,1,0) node[cross=2pt, thick]{};
\node[left=2pt] at (0,0,0) {{\scriptsize $(0,0,3)$}};
\node[above=2pt] at (0,0,2) {{\scriptsize $(0,0,1)$}};
\node[below=2pt] at (1,1,0) {{\scriptsize $(1,1,3)$}};
\node[right=2pt] at (1,1,2) {{\scriptsize $(1,1,1)$}};
\node[left] at (1,0,1) {{\scriptsize $(1,0,2)$}};
\node[right] at (0,1,1) {{\scriptsize $(0,1,2)$}};
\end{tikzpicture}
\end{minipage}
\begin{minipage}{0.31\textwidth}
\tdplotsetmaincoords{80}{120}
\begin{tikzpicture}[tdplot_main_coords, scale=1]
\draw[dashed] (0, 0, 0) -- (0, 0, 4);
\draw[dashed] (0, 0, 0) -- (0, 2, 2);
\draw[dashed] (0, 0, 4) -- (0, 2, 2);
\draw[] (0, 0, 0) -- (2, 2, 0);
\draw[] (0, 0, 4) -- (2, 2, 4);
\draw[] (2, 2, 0) -- (2, 2, 4);
\draw[] (0, 0, 0) -- (2, 0, 2);
\draw[] (0, 0, 4) -- (2, 0, 2);
\draw[] (2, 2, 0) -- (2, 0, 2);
\draw[] (2, 2, 4) -- (2, 0, 2);
\draw[] (2, 2, 0) -- (0, 2, 2);
\draw[] (2, 2, 4) -- (0, 2, 2);
\draw (0,0,4) node[cross=2pt, thick]{};
\draw[fill=white, line width=0.2mm] (0,0,2) circle (2pt) node[anchor=west] {};
\draw (0,0,0) node[cross=2pt, thick]{};
\draw[fill=white, line width=0.2mm] (0,1,3) circle (2pt) node[anchor=west] {};
\draw[fill=white, line width=0.2mm] (0,1,1) circle (2pt) node[anchor=west] {};
\draw[fill=white, line width=0.2mm] (0,2,2) circle (2pt) node[anchor=west] {};
\draw[fill=white, line width=0.2mm] (1,0,3) circle (2pt) node[anchor=west] {};
\draw[fill=white, line width=0.2mm] (1,0,1) circle (2pt) node[anchor=west] {};
\draw[fill=white, line width=0.2mm] (1,1,4) circle (2pt) node[anchor=west] {};
\draw[fill=white, line width=0.2mm] (1,1,2) circle (2pt) node[anchor=west] {};
\draw (2,2,2) node[cross=2pt, thick]{};
\draw (2,2,0) node[cross=2pt, thick]{};
\draw[fill=white, line width=0.2mm] (1,2,3) circle (2pt) node[anchor=west] {};
\draw (2,1,1) node[cross=2pt, thick]{};
\draw[fill=white, line width=0.2mm] (2,0,2) circle (2pt) node[anchor=west] {};
\draw[fill=white, line width=0.2mm] (2,1,3) circle (2pt) node[anchor=west] {};
\draw (1,2,1) node[cross=2pt, thick]{};
\draw[fill=white, line width=0.2mm] (2,2,4) circle (2pt) node[anchor=west] {};
\draw (1,1,0) node[cross=2pt, thick]{};
\node[left=2pt] at (0,0,0) {{\scriptsize $(0,0,8)$}};
\node[above=2pt] at (0,0,4) {{\scriptsize $(0,0,4)$}};
\node[below=2pt] at (2,2,0) {{\scriptsize $(2,2,8)$}};
\node[right=2pt] at (2,2,4) {{\scriptsize $(2,2,4)$}};
\node[left] at (2,0,2) {{\scriptsize $(2,0,6)$}};
\node[right] at (0,2,2) {{\scriptsize $(0,2,6)$}};
\end{tikzpicture}
\end{minipage}
\begin{minipage}{0.40\textwidth}
\tdplotsetmaincoords{80}{120}
\begin{tikzpicture}[tdplot_main_coords, scale=1]
\draw[dashed] (0, 0, 0) -- (0, 0, 6);
\draw[dashed] (0, 0, 0) -- (0, 3, 3);
\draw[dashed] (0, 0, 6) -- (0, 3, 3);
\draw[] (0, 0, 0) -- (3, 3, 0);
\draw[] (0, 0, 6) -- (3, 3, 6);
\draw[] (3, 3, 0) -- (3, 3, 6);
\draw[] (0, 0, 0) -- (3, 0, 3);
\draw[] (0, 0, 6) -- (3, 0, 3);
\draw[] (3, 3, 0) -- (3, 0, 3);
\draw[] (3, 3, 6) -- (3, 0, 3);
\draw[] (3, 3, 0) -- (0, 3, 3);
\draw[] (3, 3, 6) -- (0, 3, 3);
\draw[fill=white, line width=0.2mm] (0,0,6) circle (2pt) node[anchor=west] {};
\draw (0,0,4) node[cross=2pt, thick]{};
\draw[fill=white, line width=0.2mm] (0,0,2) circle (2pt) node[anchor=west] {};
\draw (0,0,0) node[cross=2pt, thick]{};
\draw[fill=white, line width=0.2mm] (0,1,5) circle (2pt) node[anchor=west] {};
\draw[fill=white, line width=0.2mm] (0,1,3) circle (2pt) node[anchor=west] {};
\draw[fill=white, line width=0.2mm] (0,1,1) circle (2pt) node[anchor=west] {};
\draw[fill=white, line width=0.2mm] (0,2,4) circle (2pt) node[anchor=west] {};
\draw[fill=white, line width=0.2mm] (0,2,2) circle (2pt) node[anchor=west] {};
\draw[fill=white, line width=0.2mm] (0,3,3) circle (2pt) node[anchor=west] {};
\draw[fill=white, line width=0.2mm] (1,0,5) circle (2pt) node[anchor=west] {};
\draw[fill=white, line width=0.2mm] (1,0,3) circle (2pt) node[anchor=west] {};
\draw[fill=white, line width=0.2mm] (1,0,1) circle (2pt) node[anchor=west] {};
\draw[fill=white, line width=0.2mm] (1,1,6) circle (2pt) node[anchor=west] {};
\draw[fill=white, line width=0.2mm] (1,1,4) circle (2pt) node[anchor=west] {};
\draw[fill=white, line width=0.2mm] (1,1,2) circle (2pt) node[anchor=west] {};
\draw (3,3,4) node[cross=2pt, thick]{};
\draw (3,3,2) node[cross=2pt, thick]{};
\draw (3,3,0) node[cross=2pt, thick]{};
\draw[fill=white, line width=0.2mm] (1,2,5) circle (2pt) node[anchor=west] {};
\draw[fill=white, line width=0.2mm] (1,2,3) circle (2pt) node[anchor=west] {};
\draw (3,2,3) node[cross=2pt, thick]{};
\draw (3,2,1) node[cross=2pt, thick]{};
\draw[fill=white, line width=0.2mm] (1,3,4) circle (2pt) node[anchor=west] {};
\draw (3,1,2) node[cross=2pt, thick]{};
\draw[fill=white, line width=0.2mm] (2,0,4) circle (2pt) node[anchor=west] {};
\draw[fill=white, line width=0.2mm] (2,0,2) circle (2pt) node[anchor=west] {};
\draw[fill=white, line width=0.2mm] (2,1,5) circle (2pt) node[anchor=west] {};
\draw[fill=white, line width=0.2mm] (2,1,3) circle (2pt) node[anchor=west] {};
\draw (2,3,3) node[cross=2pt, thick]{};
\draw (2,3,1) node[cross=2pt, thick]{};
\draw[fill=white, line width=0.2mm] (2,2,6) circle (2pt) node[anchor=west] {};
\draw[fill=white, line width=0.2mm] (2,2,4) circle (2pt) node[anchor=west] {};
\draw (2,2,2) node[cross=2pt, thick]{};
\draw (2,2,0) node[cross=2pt, thick]{};
\draw[fill=white, line width=0.2mm] (2,3,5) circle (2pt) node[anchor=west] {};
\draw (2,1,1) node[cross=2pt, thick]{};
\draw[fill=white, line width=0.2mm] (3,0,3) circle (2pt) node[anchor=west] {};
\draw[fill=white, line width=0.2mm] (3,1,4) circle (2pt) node[anchor=west] {};
\draw (1,3,2) node[cross=2pt, thick]{};
\draw[fill=white, line width=0.2mm] (3,2,5) circle (2pt) node[anchor=west] {};
\draw (1,2,1) node[cross=2pt, thick]{};
\draw[fill=white, line width=0.2mm] (3,3,6) circle (2pt) node[anchor=west] {};
\draw (1,1,0) node[cross=2pt, thick]{};
\node[left=2pt] at (0,0,0) {{\scriptsize $(0,0,15)$}};
\node[above=2pt] at (0,0,6) {{\scriptsize $(0,0,9)$}};
\node[below=2pt] at (3,3,0) {{\scriptsize $(3,3,15)$}};
\node[right=2pt] at (3,3,6) {{\scriptsize $(3,3,9)$}};
\node[left] at (3,0,3) {{\scriptsize $(3,0,12)$}};
\node[right] at (0,3,3) {{\scriptsize $(0,3,12)$}};
\end{tikzpicture}
\end{minipage}
    \caption{The octahedron for $n = 1, 2, 3$.}
    \label{fig:octahedron}
\end{figure}
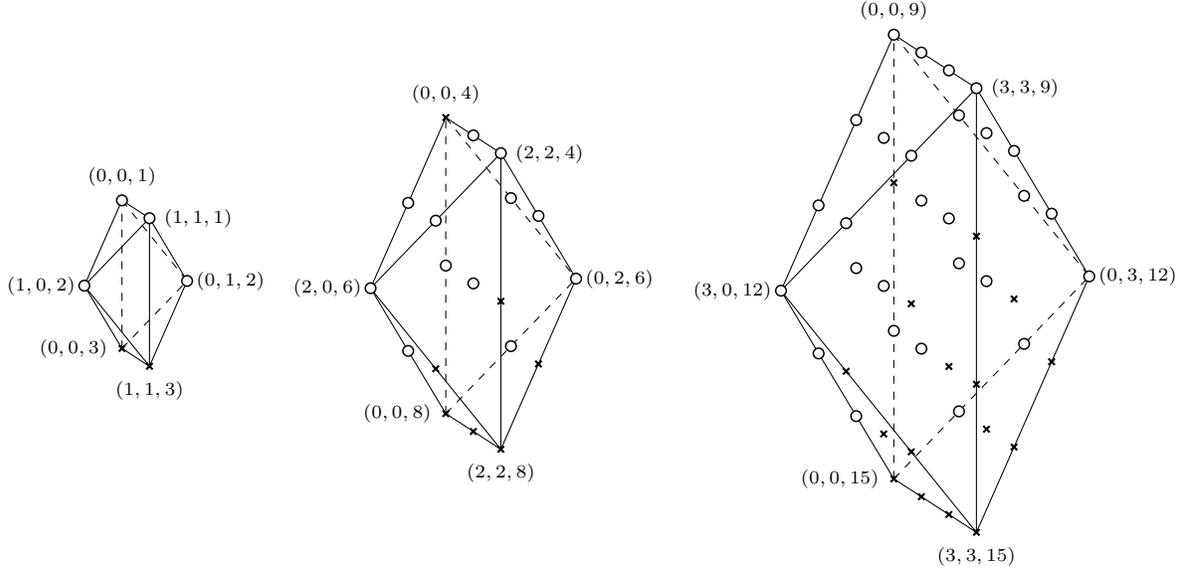 
%
%As usual in depicting positively graded algebras, the $z$-axis is reversed in the pictures, so that 
%
\end{remark}

\subsubsection{Join-irreducibles in type B}\label{ss:JIB}

Join-irreducible elements in $W$ of type $B$ are explicitly determined in \cite[Section 4]{GeK}, % where $JI$ is called the base. 
and, in a different way, in \cite[Section 2]{anderson}. 
We describe these and determine the Bruhat order on each $\JI{i}{j}$. By symmetry, we only consider the case $i\leq j$.

The reduced expressions appearing here are rather lengthy. So, we also present elements in $W$ as signed permutations and provide the corresponding pictures. Recall that a signed permutation is a permutation $w$ of $(-n-1, -n, \ldots, -1,0,1, \ldots, n, n+1)$ with the property that $w(-k)=-w(k)$. We identify such $w$ with the sequence $(w(1),w(2),\ldots, w(n))$, which uniquely determines $w$. We write $\overline{k}$ instead of $-k$ (only in this subsection).

%The conversion for simple reflections is as follows, 
For example, when $n=2$, i.e., in type $B_3$, we have $s_0 = (\overline{1},2,3)$, $s_1 = (2,1,3)$, $s_0s_1 = (2,\overline{1},3)$  and $s_1s_0 = (\overline{2},1,3)$
depicted, respectively, as follows:
%
%\[ \scalebox{0.4}{\input{s0.pgf}},\ \ \ \ \ \ \ \ \ \ \ \scalebox{0.4}{\input{s3.pgf}} .\]
%
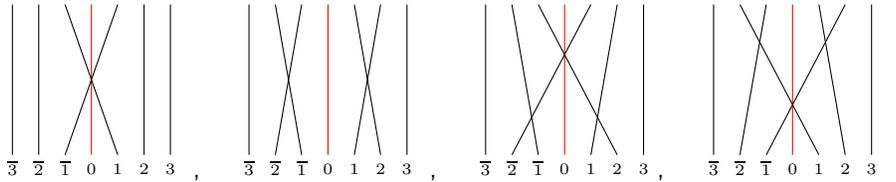
\[ \scalebox{1}{
\begin{tikzpicture}[scale=1, baseline=-3mm]
\draw[red,fill=white] (0,0) -- (0,2.000000);
\node[below] at (0,0) {\tiny $0$};
\draw (0.350000,0) -- (-0.350000,2.000000);
\draw (-0.350000,0) -- (0.350000,2.000000);
\node[below] at (0.350000,0) {\tiny $1$};
\node[below=-0.45mm] at (-0.350000,0) {\tiny $\overline{1}$};
\draw (0.700000,0) -- (0.700000,2.000000);
\draw (-0.700000,0) -- (-0.700000,2.000000);
\node[below] at (0.700000,0) {\tiny $2$};
\node[below=-0.45mm] at (-0.700000,0) {\tiny $\overline{2}$};
\draw (1.050000,0) -- (1.050000,2.000000);
\draw (-1.050000,0) -- (-1.050000,2.000000);
\node[below] at (1.050000,0) {\tiny $3$};
\node[below=-0.45mm] at (-1.050000,0) {\tiny $\overline{3}$};
\end{tikzpicture} , \quad
\begin{tikzpicture}[scale=1, baseline=-3mm]
\draw[red,fill=white] (0,0) -- (0,2.000000);
\node[below] at (0,0) {\tiny $0$};
\draw (0.350000,0) -- (0.700000,2.000000);
\draw (-0.350000,0) -- (-0.700000,2.000000);
\node[below] at (0.350000,0) {\tiny $1$};
\node[below=-0.45mm] at (-0.350000,0) {\tiny $\overline{1}$};
\draw (0.700000,0) -- (0.350000,2.000000);
\draw (-0.700000,0) -- (-0.350000,2.000000);
\node[below] at (0.700000,0) {\tiny $2$};
\node[below=-0.45mm] at (-0.700000,0) {\tiny $\overline{2}$};
\draw (1.050000,0) -- (1.050000,2.000000);
\draw (-1.050000,0) -- (-1.050000,2.000000);
\node[below] at (1.050000,0) {\tiny $3$};
\node[below=-0.45mm] at (-1.050000,0) {\tiny $\overline{3}$};
\end{tikzpicture} , \quad
\begin{tikzpicture}[scale=1, baseline=-3mm]
\draw[red,fill=white] (0,0) -- (0,2.000000);
\node[below] at (0,0) {\tiny $0$};
\draw (0.350000,0) -- (0.700000,2.000000);
\draw (-0.350000,0) -- (-0.700000,2.000000);
\node[below] at (0.350000,0) {\tiny $1$};
\node[below=-0.45mm] at (-0.350000,0) {\tiny $\overline{1}$};
\draw (0.700000,0) -- (-0.350000,2.000000);
\draw (-0.700000,0) -- (0.350000,2.000000);
\node[below] at (0.700000,0) {\tiny $2$};
\node[below=-0.45mm] at (-0.700000,0) {\tiny $\overline{2}$};
\draw (1.050000,0) -- (1.050000,2.000000);
\draw (-1.050000,0) -- (-1.050000,2.000000);
\node[below] at (1.050000,0) {\tiny $3$};
\node[below=-0.45mm] at (-1.050000,0) {\tiny $\overline{3}$};
\end{tikzpicture}, \quad
\begin{tikzpicture}[scale=1, baseline=-3mm]
\draw[red,fill=white] (0,0) -- (0,2.000000);
\node[below] at (0,0) {\tiny $0$};
\draw (0.350000,0) -- (-0.700000,2.000000);
\draw (-0.350000,0) -- (0.700000,2.000000);
\node[below] at (0.350000,0) {\tiny $1$};
\node[below=-0.45mm] at (-0.350000,0) {\tiny $\overline{1}$};
\draw (0.700000,0) -- (0.350000,2.000000);
\draw (-0.700000,0) -- (-0.350000,2.000000);
\node[below] at (0.700000,0) {\tiny $2$};
\node[below=-0.45mm] at (-0.700000,0) {\tiny $\overline{2}$};
\draw (1.050000,0) -- (1.050000,2.000000);
\draw (-1.050000,0) -- (-1.050000,2.000000);
\node[below] at (1.050000,0) {\tiny $3$};
\node[below=-0.45mm] at (-1.050000,0) {\tiny $\overline{3}$};
\end{tikzpicture} . }  \]
%
%In general, concatenating braids from top down corresponds to multiplying the corresponding elements in $W$ from left to right, i.e. we consider the right action of $W$.
We read this picture from top to bottom and read a reduced expression from left to right.

%Recall that $s_{ij} \in W$ is defined to be the product $s_i \cdots s_j$ of simple reflections along the unique shortest path starting in $i$ and ending in $j$ in the Dynkin diagram.

\begin{comment}
% Not needed included in the next proposition.
\begin{proposition}\label{BJI00}
The subposet $\JI{0}{0} \subset (W,\leq)$ is a chain 
\[ b(1,1,1) \leq b(2,1,1) \leq \ldots \leq b(n,1,1) \leq b(n+1,1,1) . \]
\end{proposition}
\begin{proof}
The Bruhat relations above follow from the easily checked fact that a reduced expression for $b(k,1,1)$ is given by
$b(k,1,1) = (012\cdots (k-1))\cdots(012)(01)0$.
\end{proof}
\end{comment}

\begin{proposition}
\label{BJI0a}
Let $0 \leq j\leq n$. 
Then the set $\JI{0}{j}$ consists of the elements
\begin{equation}
\label{eq:B_JI_0j}
b^{0j}_k = b_{k} := s_{00} \cdot s_{10} \cdot \ldots \cdot s_{k-1,0} \cdot s_{k1} \cdot s_{k+1,2} \cdot\ldots \cdot s_{k+j-1,j} , \quad 1 \leq k \leq n+1-j, %= \big(01 \cdots (j+k-1)  \big) \big(01 \cdots (j+k-2)\big) \cdots \big(01 \cdots j\big)
\end{equation}
with $b_{k}$ having the following picture:
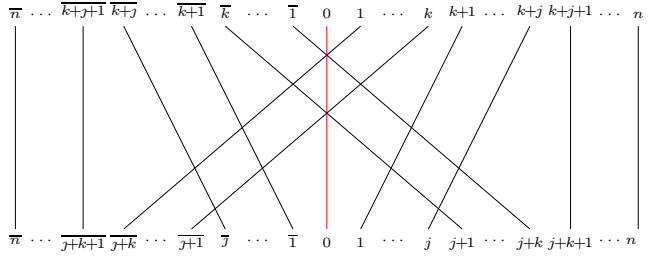
\[\scalebox{0.9}{
\medmuskip=-2mu
\begin{tikzpicture}[scale=1, baseline=-3mm]
\draw[red,fill=white] (0,0) -- (0,3.000000);
\node[below] at (0,0) {\tiny $0$};
\node[above] at (0,3.000000) {\tiny $0$};
\draw (0.500000,0) -- (2.000000,3.000000);
\draw (-0.500000,0) -- (-2.000000,3.000000);
\node[below] at (0.500000,0) {\tiny $1$};
\node[below=-0.45mm] at (-0.500000,0) {\tiny $\overline{1}$};
\node[above] at (0.500000,3.000000) {\tiny $1$};
\node[above] at (-0.500000,3.000000) {\tiny $\overline{1}$};
\node[below] at (1.000000,0) {\tiny $\cdots$};
\node[below] at (-1.000000,0) {\tiny $\cdots$};
\node[above] at (1.000000,3.000000) {\tiny $\cdots$};
\node[above] at (-1.000000,3.000000) {\tiny $\cdots$};
\draw (1.500000,0) -- (3.000000,3.000000);
\draw (-1.500000,0) -- (-3.000000,3.000000);
\node[below] at (1.500000,0) {\tiny $j$};
\node[below=-0.45mm] at (-1.500000,0) {\tiny $\overline{\jmath}$};
\node[above] at (1.500000,3.000000) {\tiny $k$};
\node[above] at (-1.500000,3.000000) {\tiny $\overline{k}$};
\draw (2.000000,0) -- (-1.500000,3.000000);
\draw (-2.000000,0) -- (1.500000,3.000000);
\node[below] at (2.000000,0) {\tiny $j+1$};
\node[below=-0.45mm] at (-2.000000,0) {\tiny $\overline{\jmath+1}$};
\node[above] at (2.000000,3.000000) {\tiny $k+1$};
\node[above] at (-2.000000,3.000000) {\tiny $\overline{k+1}$};
\node[below] at (2.500000,0) {\tiny $\cdots$};
\node[below] at (-2.500000,0) {\tiny $\cdots$};
\node[above] at (2.500000,3.000000) {\tiny $\cdots$};
\node[above] at (-2.500000,3.000000) {\tiny $\cdots$};
\draw (3.000000,0) -- (-0.500000,3.000000);
\draw (-3.000000,0) -- (0.500000,3.000000);
\node[below] at (3.000000,0) {\tiny $j+k$};
\node[below=-0.45mm] at (-3.000000,0) {\tiny $\overline{\jmath+k}$};
\node[above] at (3.000000,3.000000) {\tiny $k+j$};
\node[above] at (-3.000000,3.000000) {\tiny $\overline{k+\jmath}$};

\draw (3.600000,0) -- (3.600000,3.000000);
\draw (-3.600000,0) -- (-3.600000,3.000000);
\node[below] at (3.600000,0) {\tiny $j+k+1$};
\node[below=-0.45mm] at (-3.600000,0) {\tiny $\overline{\jmath+k+1}$};
\node[above] at (3.600000,3.000000) {\tiny $k+j+1$};
\node[above] at (-3.600000,3.000000) {\tiny $\overline{k+\jmath+1}$};
\node[below] at (4.200000,0) {\tiny $\cdots$};
\node[below] at (-4.200000,0) {\tiny $\cdots$};
\node[above] at (4.200000,3.000000) {\tiny $\cdots$};
\node[above] at (-4.200000,3.000000) {\tiny $\cdots$};
\draw (4.600000,0) -- (4.600000,3.000000);
\draw (-4.600000,0) -- (-4.600000,3.000000);
\node[below] at (4.500000,0) {\tiny $n$};
\node[below=-0.45mm] at (-4.600000,0) {\tiny $\overline{n}$};
\node[above] at (4.600000,3.000000) {\tiny $n$};
\node[above] at (-4.600000,3.000000) {\tiny $\overline{n}$};
\end{tikzpicture}} . \]
There is no cancellation in the above product, i.e. one obtains a reduced expressions for $b_k$ by concatenating the reduced expressions for the factors.
Moreover, as a subposet of $(W,\leq)$, $\JI{0}{j}$ is a chain
\begin{equation}
\label{eq:00_chain} \{b_1 < b_2 < \ldots < b_{n+1-j}\} .
\end{equation}
\end{proposition}
\begin{proof}
The first part follows directly from \cite[Theorem 4.6]{GeK}. The Bruhat relations in (\ref{eq:00_chain}) follow from the reduced expressions obtained from (\ref{eq:B_JI_0j}).
%
%That $w_{0,a,r}\in BG(0,a)$ and their Bruhat order is seen from the definition.  By Lemma %\ref{BG=JI B} $w_{0,a,r}\in JI(0,a)$.
%That there is no element in $JI(0,a)$ not of the form $w_{0,a,r}$ follows from Proposition \ref{Bbut1} and Lemma \ref{Bnumber}.
\end{proof}

\begin{proposition}\label{BJIab}
Let $1\leq i\leq j\leq n$. Then the set $\JI{i}{j}$ consists of the following three types of elements:
\begin{itemize}
    \item (type $O_A$)  for $1\leq k\leq \min \{i,n+1-j\}$:
    \[b^{ij}_{\circ,k} = b_{\circ,k} := s_{i,i-k+1} \cdot s_{i+1,i-k+2} \cdot \ldots \cdot s_{k+j-1,j} , \] %=\big(i\cdots (j+k-1)\big)\big((i-1)\cdots (j+k-2)\big)\cdots \big((i-k+1)\cdots j\big)
    with the picture
    \[\scalebox{0.9}{
\medmuskip=-2mu
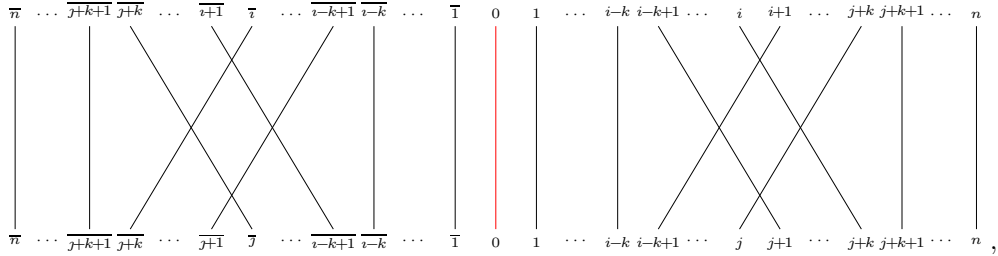
\begin{tikzpicture}[scale=1, baseline=-3mm]
\draw[red,fill=white] (0,0) -- (0,3.000000);
\node[below] at (0,0) {\tiny $0$};
\node[above] at (0,3.000000) {\tiny $0$};
\draw (0.600000,0) -- (0.600000,3.000000);
\draw (-0.600000,0) -- (-0.600000,3.000000);
\node[below] at (0.600000,0) {\tiny $1$};
\node[below=-0.45mm] at (-0.600000,0) {\tiny $\overline{1}$};
\node[above] at (0.600000,3.000000) {\tiny $1$};
\node[above] at (-0.600000,3.000000) {\tiny $\overline{1}$};
\node[below] at (1.200000,0) {\tiny $\cdots$};
\node[below] at (-1.200000,0) {\tiny $\cdots$};
\node[above] at (1.200000,3.000000) {\tiny $\cdots$};
\node[above] at (-1.200000,3.000000) {\tiny $\cdots$};
\draw (1.800000,0) -- (1.800000,3.000000);
\draw (-1.800000,0) -- (-1.800000,3.000000);
\node[below] at (1.800000,0) {\tiny $i-k$};
\node[below=-0.45mm] at (-1.800000,0) {\tiny $\overline{\imath-k}$};
\node[above] at (1.800000,3.000000) {\tiny $i-k$};
\node[above] at (-1.800000,3.000000) {\tiny $\overline{\imath-k}$};
\draw (2.400000,0) -- (4.200000,3.000000);
\draw (-2.400000,0) -- (-4.200000,3.000000);
\node[below] at (2.400000,0) {\tiny $i-k+1$};
\node[below=-0.45mm] at (-2.400000,0) {\tiny $\overline{\imath-k+1}$};
\node[above] at (2.400000,3.000000) {\tiny $i-k+1$};
\node[above] at (-2.400000,3.000000) {\tiny $\overline{\imath-k+1}$};
\node[below] at (3.000000,0) {\tiny $\cdots$};
\node[below] at (-3.000000,0) {\tiny $\cdots$};
\node[above] at (3.000000,3.000000) {\tiny $\cdots$};
\node[above] at (-3.000000,3.000000) {\tiny $\cdots$};
\draw (3.600000,0) -- (5.400000,3.000000);
\draw (-3.600000,0) -- (-5.400000,3.000000);
\node[below] at (3.600000,0) {\tiny $j$};
\node[below=-0.45mm] at (-3.600000,0) {\tiny $\overline{\jmath}$};
\node[above] at (3.600000,3.000000) {\tiny $i$};
\node[above] at (-3.600000,3.000000) {\tiny $\overline{\imath}$};
\draw (4.200000,0) -- (2.400000,3.000000);
\draw (-4.200000,0) -- (-2.400000,3.000000);
\node[below] at (4.200000,0) {\tiny $j+1$};
\node[below=-0.45mm] at (-4.200000,0) {\tiny $\overline{\jmath +1}$};
\node[above] at (4.200000,3.000000) {\tiny $i+1$};
\node[above] at (-4.200000,3.000000) {\tiny $\overline{\imath+1}$};
\node[below] at (4.800000,0) {\tiny $\cdots$};
\node[below] at (-4.800000,0) {\tiny $\cdots$};
\node[above] at (4.800000,3.000000) {\tiny $\cdots$};
\node[above] at (-4.800000,3.000000) {\tiny $\cdots$};
\draw (5.400000,0) -- (3.600000,3.000000);
\draw (-5.400000,0) -- (-3.600000,3.000000);
\node[below] at (5.400000,0) {\tiny $j+k$};
\node[below=-0.45mm] at (-5.400000,0) {\tiny $\overline{\jmath+k}$};
\node[above] at (5.400000,3.000000) {\tiny $j+k$};
\node[above] at (-5.400000,3.000000) {\tiny $\overline{\jmath+k}$};

\draw (6.000000,0) -- (6.000000,3.000000);
\draw (-6.000000,0) -- (-6.000000,3.000000);
\node[below] at (6.000000,0) {\tiny $j+k+1$};
\node[below=-0.45mm] at (-6.000000,0) {\tiny $\overline{\jmath+k+1}$};
\node[above] at (6.000000,3.000000) {\tiny $j+k+1$};
\node[above] at (-6.000000,3.000000) {\tiny $\overline{\jmath+k+1}$};
\node[below] at (6.600000,0) {\tiny $\cdots$};
\node[below] at (-6.600000,0) {\tiny $\cdots$};
\node[above] at (6.600000,3.000000) {\tiny $\cdots$};
\node[above] at (-6.600000,3.000000) {\tiny $\cdots$};
\draw (7.100000,0) -- (7.100000,3.000000);
\draw (-7.100000,0) -- (-7.100000,3.000000);
\node[below] at (7.100000,0) {\tiny $n$};
\node[below=-0.45mm] at (-7.100000,0) {\tiny $\overline{n}$};
\node[above] at (7.100000,3.000000) {\tiny $n$};
\node[above] at (-7.100000,3.000000) {\tiny $\overline{n}$};
\end{tikzpicture}} , \]
    \item (type $O_B$)  for $i< k\leq n+1-j$:
    \[b^{ij}_{\circ,k} = b_{\circ,k}:= s_{i0} \cdot s_{i+1,0} \cdot \ldots \cdot s_{k-1,0}  \cdot s_{k,1}  \cdot s_{k+1,2} \cdot \ldots \cdot s_{k+j-1,j} , \] %=\big(i\cdots (j+k-1)\big)\big((i-1)\cdots (j+k-2)\big)\cdots \big(0\cdots (j+k-1-i)\big)\cdots \big(0\cdots j\big)
    with the picture
    \[\scalebox{0.9}{
\medmuskip=-2mu
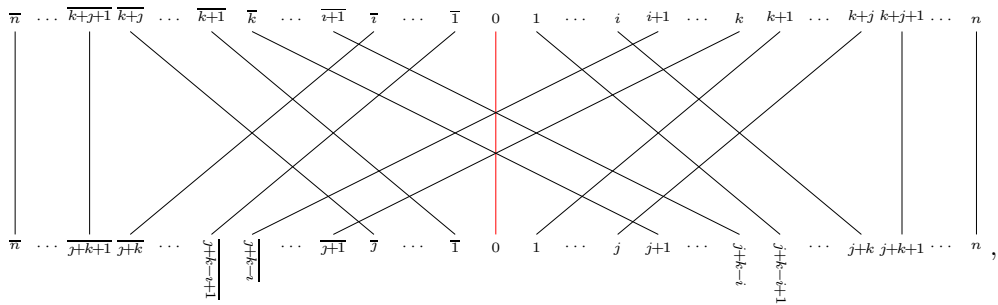
\begin{tikzpicture}[scale=1, baseline=-3mm]
\draw[red,fill=white] (0,0) -- (0,3.000000);
\node[below] at (0,0) {\tiny $0$};
\node[above] at (0,3.000000) {\tiny $0$};
\draw (0.600000,0) -- (4.200000,3.000000);
\draw (-0.600000,0) -- (-4.200000,3.000000);
\node[below] at (0.600000,0) {\tiny $1$};
\node[below=-0.45mm] at (-0.600000,0) {\tiny $\overline{1}$};
\node[above] at (0.600000,3.000000) {\tiny $1$};
\node[above] at (-0.600000,3.000000) {\tiny $\overline{1}$};
\node[below] at (1.200000,0) {\tiny $\cdots$};
\node[below] at (-1.200000,0) {\tiny $\cdots$};
\node[above] at (1.200000,3.000000) {\tiny $\cdots$};
\node[above] at (-1.200000,3.000000) {\tiny $\cdots$};
\draw (1.800000,0) -- (5.400000,3.000000);
\draw (-1.800000,0) -- (-5.400000,3.000000);
\node[below] at (1.800000,0) {\tiny $j$};
\node[below=-0.45mm] at (-1.800000,0) {\tiny $\overline{\jmath}$};
\node[above] at (1.800000,3.000000) {\tiny $i$};
\node[above] at (-1.800000,3.000000) {\tiny $\overline{\imath}$};
\draw (2.400000,0) -- (-3.600000,3.000000);
\draw (-2.400000,0) -- (3.600000,3.000000);
\node[below] at (2.400000,0) {\tiny $j+1$};
\node[below=-0.45mm] at (-2.400000,0) {\tiny $\overline{\jmath+1}$};
\node[above] at (2.400000,3.000000) {\tiny $i+1$};
\node[above] at (-2.400000,3.000000) {\tiny $\overline{\imath+1}$};
\node[below] at (3.000000,0) {\tiny $\cdots$};
\node[below] at (-3.000000,0) {\tiny $\cdots$};
\node[above] at (3.000000,3.000000) {\tiny $\cdots$};
\node[above] at (-3.000000,3.000000) {\tiny $\cdots$};
\draw (3.600000,0) -- (-2.400000,3.000000);
\draw (-3.600000,0) -- (2.400000,3.000000);
\node[below] at (3.600000,0) {\rotatebox[origin=c]{270}{\tiny $j+k-i$}};
\node[below=-0.45mm] at (-3.600000,0) {\rotatebox[origin=c]{270}{\tiny $\overline{\jmath+k-\imath}$}};
\node[above] at (3.600000,3.000000) {\tiny $k$};
\node[above] at (-3.600000,3.000000) {\tiny $\overline{k}$};
\draw (4.200000,0) -- (0.600000,3.000000);
\draw (-4.200000,0) -- (-0.600000,3.000000);
\node[below] at (4.200000,0) {\rotatebox[origin=c]{270}{\tiny $j+k-i+1$}};
\node[below=-0.45mm] at (-4.200000,0) {\rotatebox[origin=c]{270}{\tiny $\overline{\jmath+k-\imath+1}$}};
\node[above] at (4.200000,3.000000) {\tiny $k+1$};
\node[above] at (-4.200000,3.000000) {\tiny $\overline{k+1}$};
\node[below] at (4.800000,0) {\tiny $\cdots$};
\node[below] at (-4.800000,0) {\tiny $\cdots$};
\node[above] at (4.800000,3.000000) {\tiny $\cdots$};
\node[above] at (-4.800000,3.000000) {\tiny $\cdots$};
\draw (5.400000,0) -- (1.800000,3.000000);
\draw (-5.400000,0) -- (-1.800000,3.000000);
\node[below] at (5.400000,0) {\tiny $j+k$};
\node[below=-0.45mm] at (-5.400000,0) {\tiny $\overline{\jmath+k}$};
\node[above] at (5.400000,3.000000) {\tiny $k+j$};
\node[above] at (-5.400000,3.000000) {\tiny $\overline{k+\jmath}$};

\draw (6.000000,0) -- (6.000000,3.000000);
\draw (-6.000000,0) -- (-6.000000,3.000000);
\node[below] at (6.000000,0) {\tiny $j+k+1$};
\node[below=-0.45mm] at (-6.000000,0) {\tiny $\overline{\jmath+k+1}$};
\node[above] at (6.000000,3.000000) {\tiny $k+j+1$};
\node[above] at (-6.000000,3.000000) {\tiny $\overline{k+\jmath+1}$};
\node[below] at (6.600000,0) {\tiny $\cdots$};
\node[below] at (-6.600000,0) {\tiny $\cdots$};
\node[above] at (6.600000,3.000000) {\tiny $\cdots$};
\node[above] at (-6.600000,3.000000) {\tiny $\cdots$};
\draw (7.100000,0) -- (7.100000,3.000000);
\draw (-7.100000,0) -- (-7.100000,3.000000);
\node[below] at (7.100000,0) {\tiny $n$};
\node[below=-0.45mm] at (-7.100000,0) {\tiny $\overline{n}$};
\node[above] at (7.100000,3.000000) {\tiny $n$};
\node[above] at (-7.100000,3.000000) {\tiny $\overline{n}$};
\end{tikzpicture}} , \]
    \item(type $X$)  for $1\leq k\leq n+1-j$:
    \[b^{ij}_{\times,k} = b_{\times,k} := %t_{i,j+k-1} \cdot \ldots \cdot t_{i,j+1}\cdot t_{i,j}% 
    t_{ii} \cdot t_{i+1,i} \cdot \ldots \cdot t_{i+k-1,i} \cdot s_{i+k,i+1} \cdot s_{i+k+1,i+2} \cdot \ldots \cdot s_{j+k-1,j} ,
    \] % =\big(i \cdots 101 \cdots (j+k-1)\big) \big(i \cdots 101 \cdots (j+k-2)\big) \cdots \big(i \cdots 101 \cdots j\big)
    with the picture
    \[ \scalebox{0.9}{
\medmuskip=-2mu
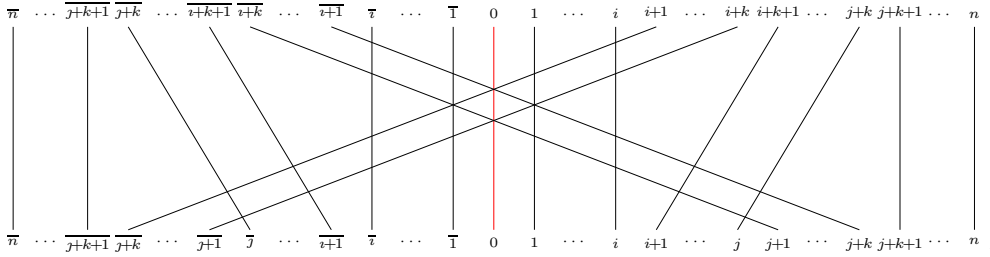
\begin{tikzpicture}[scale=1, baseline=-3mm]
\draw[red,fill=white] (0,0) -- (0,3.000000);
\node[below] at (0,0) {\tiny $0$};
\node[above] at (0,3.000000) {\tiny $0$};
\draw (0.600000,0) -- (0.600000,3.000000);
\draw (-0.600000,0) -- (-0.600000,3.000000);
\node[below] at (0.600000,0) {\tiny $1$};
\node[below=-0.45mm] at (-0.600000,0) {\tiny $\overline{1}$};
\node[above] at (0.600000,3.000000) {\tiny $1$};
\node[above] at (-0.600000,3.000000) {\tiny $\overline{1}$};
\node[below] at (1.200000,0) {\tiny $\cdots$};
\node[below] at (-1.200000,0) {\tiny $\cdots$};
\node[above] at (1.200000,3.000000) {\tiny $\cdots$};
\node[above] at (-1.200000,3.000000) {\tiny $\cdots$};
\draw (1.800000,0) -- (1.800000,3.000000);
\draw (-1.800000,0) -- (-1.800000,3.000000);
\node[below] at (1.800000,0) {\tiny $i$};
\node[below=-0.45mm] at (-1.800000,0) {\tiny $\overline{\imath}$};
\node[above] at (1.800000,3.000000) {\tiny $i$};
\node[above] at (-1.800000,3.000000) {\tiny $\overline{\imath}$};
\draw (2.400000,0) -- (4.200000,3.000000);
\draw (-2.400000,0) -- (-4.200000,3.000000);
\node[below] at (2.400000,0) {\tiny $i+1$};
\node[below=-0.45mm] at (-2.400000,0) {\tiny $\overline{\imath+1}$};
\node[above] at (2.400000,3.000000) {\tiny $i+1$};
\node[above] at (-2.400000,3.000000) {\tiny $\overline{\imath+1}$};
\node[below] at (3.000000,0) {\tiny $\cdots$};
\node[below] at (-3.000000,0) {\tiny $\cdots$};
\node[above] at (3.000000,3.000000) {\tiny $\cdots$};
\node[above] at (-3.000000,3.000000) {\tiny $\cdots$};
\draw (3.600000,0) -- (5.400000,3.000000);
\draw (-3.600000,0) -- (-5.400000,3.000000);
\node[below] at (3.600000,0) {\tiny $j$};
\node[below=-0.45mm] at (-3.600000,0) {\tiny $\overline{\jmath}$};
\node[above] at (3.600000,3.000000) {\tiny $i+k$};
\node[above] at (-3.600000,3.000000) {\tiny $\overline{\imath+k}$};
\draw (4.200000,0) -- (-3.600000,3.000000);
\draw (-4.200000,0) -- (3.600000,3.000000);
\node[below] at (4.200000,0) {\tiny $j+1$};
\node[below=-0.45mm] at (-4.200000,0) {\tiny $\overline{\jmath+1}$};
\node[above] at (4.200000,3.000000) {\tiny $i+k+1$};
\node[above] at (-4.200000,3.000000) {\tiny $\overline{\imath+k+1}$};
\node[below] at (4.800000,0) {\tiny $\cdots$};
\node[below] at (-4.800000,0) {\tiny $\cdots$};
\node[above] at (4.800000,3.000000) {\tiny $\cdots$};
\node[above] at (-4.800000,3.000000) {\tiny $\cdots$};
\draw (5.400000,0) -- (-2.400000,3.000000);
\draw (-5.400000,0) -- (2.400000,3.000000);
\node[below] at (5.400000,0) {\tiny $j+k$};
\node[below=-0.45mm] at (-5.400000,0) {\tiny $\overline{\jmath+k}$};
\node[above] at (5.400000,3.000000) {\tiny $j+k$};
\node[above] at (-5.400000,3.000000) {\tiny $\overline{\jmath+k}$};

\draw (6.000000,0) -- (6.000000,3.000000);
\draw (-6.000000,0) -- (-6.000000,3.000000);
\node[below] at (6.000000,0) {\tiny $j+k+1$};
\node[below=-0.45mm] at (-6.000000,0) {\tiny $\overline{\jmath+k+1}$};
\node[above] at (6.000000,3.000000) {\tiny $j+k+1$};
\node[above] at (-6.000000,3.000000) {\tiny $\overline{\jmath+k+1}$};
\node[below] at (6.600000,0) {\tiny $\cdots$};
\node[below] at (-6.600000,0) {\tiny $\cdots$};
\node[above] at (6.600000,3.000000) {\tiny $\cdots$};
\node[above] at (-6.600000,3.000000) {\tiny $\cdots$};
\draw (7.100000,0) -- (7.100000,3.000000);
\draw (-7.100000,0) -- (-7.100000,3.000000);
\node[below] at (7.100000,0) {\tiny $n$};
\node[below=-0.45mm] at (-7.100000,0) {\tiny $\overline{n}$};
\node[above] at (7.100000,3.000000) {\tiny $n$};
\node[above] at (-7.100000,3.000000) {\tiny $\overline{n}$};
\end{tikzpicture}} .\]
\end{itemize}
There are no cancellations in the above products, i.e. one obtains reduced expressions for $b_{\circ,k}$ and $b_{\times,k}$ by concatenating the reduced expressions for the factors.
Moreover, the partial order on $\JI{i}{j}$
is generated by
the following relations (see Figure \ref{fig:BD_JI_gen}):
\begin{align}
\label{relo}    & b_{\circ,k} < b_{\circ,k+1}, & & 1 \leq k \leq n-j, \\
\label{relx}    & b_{\times,k} < b_{\times,k+1}, & & 1 \leq k \leq n-j,\\
\label{relox}    & b_{\circ,k} < b_{\times,k}, & & 1 \leq k \leq n+1-j, \\
\label{relxo}    & b_{\times,k} < b_{\circ,k+i}, & & 1 \leq k \leq n+1-j-i.
\end{align}
%\begin{enumerate}
%    \item $b_{\circ,k} < b_{\circ,k+1}$, %for each $k$;
%    \item $b_{\times,k} < b_{\times,k+1}$, % for each $k$;%for each $1\leq k\leq n+1-j$;
%    \item $b_{\circ,k} < b_{\times,k}$,   %for each $k$;%for each $1\leq k\leq n+1-j$;
%    \item $b_{\times,k} < b_{\circ,k+i}$, % for each $k$. 
%\end{enumerate}
\end{proposition}
\begin{proof}
The fact that $\JI{i}{j}$ consists of the elements as above follows from \cite[Theorem 4.6]{GeK}. The reduced expressions we choose to give agrees with the expressions used in \cite{GeK}. The case $l=0$ in \cite[Theorem 4.6]{GeK} is our type $O_A$; $l_1=0$ or $l=c$ is our type $X$ (for $i=j$ or $i\neq j$), and $c = m+l$ is our type $O_B$. The Bruhat relations (\ref{relo}), (\ref{relx}) and (\ref{relox}) follow directly from the reduced expressions given above. For the relation (\ref{relxo}), we see from the reduced expressions above that it is enough to prove it in the case $i=j$, which follows easily from e.g. \cite[Theorem 8.1.8]{BB2}.

To prove that (\ref{relo})--(\ref{relxo}) are generating relations in $\JI{i}{j}$, it is enough to observe that $b_{\circ,k} \not\le b_{\times,k-1}$, since  $b_{\circ,k}$ involves the simple reflection $(j+k-1)$ in its reduced expressions while $b_{\times,k-1}$ does not, and, moreover, that $b_{\times,k} \not\le b_{\circ,k'}$ for $k'<k+i$, by comparing the number of occurrences of $s_0$ in their reduced expressions.
%
%Note that the elements of type $O_A$ belong to the parabolic subgroup $(W_I,I=\{1,\cdots,n\})$ of type $A_n$, so are join irreducible.\hk{I do not see an easy argument for these w being JI, though it is easy to see they are big. So use andersen or geck to show it. After that it is KL and Lemma \ref{Bnumber}} The relations \eqref{relo}, \eqref{relx}, \eqref{relox}, \eqref{relxo} are readable from the braid presentations or from the reduced expressions.
\end{proof}

%number of $0\in S$ appearing in  reduced expressions of $x\in W$ is unique, and the index $r$ denotes this number for the join irreducible elements, except in type $O$, the number for $w_{\circ,r}$ is $r-a$ (and zero for $r-a<0$, i.e., for type $O_A$).

\begin{corollary}\label{Bnumber}
Let $i,j\in S$.
Then
\[ \displaystyle|\JI{i}{j}|= \sum_{y\in {}^i\mathcal{H}^j}p_{e,y}(1) = \begin{cases} n+1-\max\{i,j\}, & i=0 \text{ or } j=0; \\ 
2(n+1-\max\{i,j\}), &\text{otherwise.}\end{cases}. \]
\end{corollary}
\begin{proof}
Follows from comparing Propositions \ref{BJI0a} and \ref{BJIab} with Propositions  \ref{Bbut1} and \ref{specialB}.
\end{proof}

We will need also to use some information on $\BG{i}{j}$. However, neither a complete enumeration of $\BG{i}{j}\setminus\JI{i}{j}$ nor determination of the 
corresponding poset structure is needed.
The following proposition is enough for our purposes.

\begin{proposition}\label{newelementsofimportance}
Let $1\leq i\leq j \leq n$. The elements
\[f^{ij}_{k} = b^{ij}_{\circ,k+1} \vee b^{ij}_{\times,k} \in \BG{i}{j}\setminus \JI{i}{j}\]
are well-defined, for each $1\leq k\leq n-j$, and 
\begin{equation}\label{eqJMf}
    \JM(f^{ij}_k)=\{b^{ij}_{\circ,k+1}, b^{ij}_{\times,k}\}.
\end{equation} 
Moreover, we have \begin{equation}\label{f<bx}
    f^{ij}_k < b^{ij}_{\times,k+1},
\end{equation} for each $1\leq k\leq n-j$.
\end{proposition}

\begin{proof}
Fix $i,j$ and $1\leq k\leq n-j$. 
We define the elements $f = f^{ij}_k$ as follows.
Take the reduced expression for $b^{ij}_{\circ,k+1}$ given in Proposition~\ref{BJIab}. 
Note that the $k$-th factor is $s_{i+k-1,i-1}$, both for $k\leq i$ and $i<k$.
Let $f$ be given by the expression obtained from 
the reduced expression of $b^{ij}_{\circ,k+1}$
by replacing the first $k$ factors  of the form $s_{m,l}$ by $t_{m,i-1}$. 
The result is an expression of the form used in Section 4 in \cite{GeK}.
In particular, it is reduced and the element belongs to $\BG{i}{j}$.
More explicitly, we have
   \[f = 
    t_{i,i-1} \cdot t_{i+1,i-1} \cdot \ldots \cdot t_{i+k-1,i-1} \cdot s_{i+k,i} \cdot s_{i+k+1,i+1} \cdot \ldots \cdot s_{j+k,j} .    \]

To show that $f = b^{ij}_{\circ,k+1} \vee b^{ij}_{\times,k} $,
we use \cite[Proposition 4.5]{GeK}, which states that any $f\in\BG{i}{j}$ is written as $f=f_- \vee f_+$, whenever $f\geq f_+,f_- \in\BG{i}{j}$ are such that 
\begin{enumerate}
    \item $f$ and $f_-$ have the same number of $s_0$ appearing in their reduced expressions (note that the number is independent of reduced expressions);
    \item $f, f_+ \in W_m\setminus W_{m-1}$, for some $m\leq n$, where $W_m\subset W$ is the parabolic subgroup generated by $\{s_0,\cdots s_m\}\subset S$.
\end{enumerate}
By construction, the pair $f_-=b^{ij}_{\times,k}$ and $f_+=b^{ij}_{\circ,k+1}$ satisfies the conditions, with the numbers $k$ and $m=j+k$, and establishes the first claim. Dissectivity, Proposition~\ref{prop:dissJM}, and Corollary~\ref{cor:JM=JM} give \eqref{eqJMf}.

The relation \eqref{f<bx} follows from the given reduced expressions since the two are comparable factor by factor.
\end{proof}

\begin{lemma}\label{BY}
The Bruhat relations in (\ref{eq:00_chain}), (\ref{relo}), (\ref{relx}) and (\ref{relox}), as well as \eqref{f<bx}, are socle-killing. 

%
%For each $a,b\in S$, $r<r'$ and $?=\circ,\times$ or empty, there is $z\in JI(a\pm 1,b)$ such that $w_{a,b,?,r}< z<w_{a,b,?,r'}$. For each $a,b\geq 1$ and $r$, there is $z\in JI(a\pm 1,b)$ such that $w_{a,b,\times,r}<z<w_{a,b,\circ,r+a}$.
\end{lemma}
\begin{proof}
This is a routine check of all the cases, e.g., using the given reduced expressions.
%\hk{It is probably acceptable to just provide some examples with pictures? But we could also use the general form to do all cases.}
\end{proof}
%In other words, $w_{a,b,?,r}<w_{a,b,?,r'}$ and $w_{a,b,\times,r}<w_{a,b,\circ,r+a}$ are socle-killing.
%

\begin{proposition}\label{Bbgchain}
For any $0\leq i,j\leq n$, there exists a chain of length $\sum_{y\in {}^i\mathcal{H}^j}p_{e,y}(1)$ in $\BG{i}{j}$.
\end{proposition}

\begin{proof}
When $i=0$ or $j=0$, the elements in $\JI{i}{j}$ form a chain, 
as shown in Proposition~\ref{BJI0a}. This chain is of the desired length.

When $i\neq 0 \neq j$, then Proposition~\ref{BJIab} and Proposition~\ref{newelementsofimportance} provide the chain
\begin{equation}\label{bbfbfb}
  b_{\circ,1}<b_{\times,1}<f_1<b_{\times,2}<f_2<\cdots < f_{n-\max \{i,j\}} < b_{\times,n-\max \{i,j\}+1}  
\end{equation}
which has the desired length by Corollary~\ref{Bnumber}.
\end{proof}

\begin{corollary}\label{halfsosumB}
Let $w\in W$. The socle of $\co_w$ is contained in the sum of $\so_{x}$ taken
over $x\in \JM(w)$.
\end{corollary}
\begin{proof}
This follows from Proposition~\ref{Bbgchain} and Proposition~\ref{cor:sum}.
\end{proof}

\subsubsection{The join-irreducible socles in type B}\label{ss:jisocleB}

% I think (R) not needed, we already have notation for these elements:
%When the descents $a,b\in S$ are fixed, we denote by the shorter element in $\hc(a,b)$ by o and the longer element by x. If $|\hc(a,b)|=1$ we have $\circ = \times$.

Recall from Subsection~\ref{ssKLB} the structure of the penultimate cell $\mathcal{J}$.

\begin{proposition}\label{Bsoc00}
%The socle of $\Delta_e/\Delta_{w_{0,0,r}}$ is the isotypic subquotient $L_x\langle -\ell(w_0)+2n+3-2r\rangle$ of $\Delta_e$ for $x\in \hc(0,0)$, where $x=$o if $r\equiv n$ modulo $2$ and $x=$x otherwise. 
For $b_k \in \JI{0}{0}$, where $1 \leq k \leq n+1$, we have
\[ \soc \Delta_e / \Delta_{b_k} \cong \begin{cases} L_{u_{00}}\langle -\ell(u_{00})+2(n-k)\rangle, & k \equiv n \text{ (mod $2$)};   \\[.75em]
L_{w_{00}}\langle -\ell(w_{00})+2(n+1-k)\rangle,   &   k \not\equiv n \text{ (mod $2$)} . \end{cases}  \]
\end{proposition}
\begin{proof}
From Proposition \ref{BJI0a} and Lemma \ref{BY} we know that $\JI{0}{0}$ is a socle-killing chain. By Proposition \ref{prop3} and Proposition \ref{prop6}, the socles corresponding to the elements $b_k$ are direct sums of shifts of $L_{u_{00}}$ and $L_{w_{00}}$. 
Proposition~\ref{specialB} gives the degrees in $\Delta_e$ in which these composition factors appear. The claim now follows from Lemma \ref{socledie}, Lemma \ref{Bnumber} and the pigeonhole principle.
\end{proof}

\begin{proposition}\label{Bsocle0a}
For $1 \leq j \leq n$ and $b_k \in \JI{0}{j}$, where $1 \leq k \leq n+1-j$, we have
\[ \soc \Delta_e/\Delta_{b_{k}} \cong L_{u_{0j}}\langle -\ell(u_{0j})+2(n+1-j-k) \rangle . \]
\end{proposition}

\begin{proof}
Analogous to the proof of Proposition \ref{Bsoc00}.
\end{proof}

To determine the case $1\leq i\leq j\leq n$, we need some further computation of KL polynomials.
The following lemma, which is derivable from Lemma~\ref{lem:KL_sz_y} directly, determines all $p_{x,y}$, for $x\in W$ and $y\in\jc$ inductively. Note that the induction also works when we restrict to $y\in \lc\subset \jc$.

\begin{lemma}\label{KLBlemma}
Let $z\in W$ and $s\in S$ be such that $zs>s$.
\begin{enumerate}
\item 
For $i>0$ and $y_{i,j}\in{}^i\hc^j$, we have
\begin{equation}
    p_{zs,y_{i,j}}=\begin{cases}
                        v\inv p_{z,y_{i,j}}, &\text{if $s\neq s_j$};\\
                        -vp_{z,y_{i,j}}+p_{z,y_{i,j+1}}+p_{z,y_{i,j-1}}+\delta_{y_{i,j}.s_{j},w_0}p_{z,w_0}, &\text{if $s=s_j$}.\\
                        \end{cases}
\end{equation}
Here, by convention, $u_{i,j}=w_{i,-j}$, if $j<0$, %w_{i,j} with j<0 does not appear 
and $y_{i,j}=0$, if $j>n$. Note that $\delta_{y_{i,j}.s_{j},w_0}=\delta_{y_{i,j},w_{i,i}}$.
\item For $y_{0,j}\in{}^0\hc^j$, we have
\begin{equation}
    p_{zs,w_{0,j}}=\begin{cases}
                        v\inv p_{z,w_{0,j}}, &\text{if $s\neq s_j$};\\
                        -vp_{z,w_{0,j}}+p_{z,w_{0,j+1}}+p_{z,w_{0,j}.s_j}+\delta_{1,j}p_{z,u_{0,0}},& \text{if $s=s_j$};\\
                        \end{cases}
\end{equation}
and
\begin{equation}
    p_{z.s,u_{0,0}}=\begin{cases}
                        v\inv p_{z,u_{0,0}},& \text{if $s\neq s_0$};\\
                        -vp_{z,u_{0,0}}+p_{z,u_{0,1}},& \text{if $s=s_0$}.\\
                        \end{cases}
\end{equation}
\end{enumerate}
\end{lemma}

\begin{lemma}
\label{lem:B_KL}
For $0 \leq i \leq j \leq n$ and $0<j$, we have:
\begin{align}
    p_{s_{ij}, u_{ij}} & = v^{\ell(w_0) - 2(j+1)} + v^{\ell(w_0) - 2(j+2)} + \cdots + v^{\ell(w_0) - 2n}  %= p_{j,u_{jj}}  = v^{\ell(w_0) - 2(j+1)} + v^{\ell(w_0) - 2(j+2)} + \cdots + v^{\ell(w_0) - 2n} 
    , \\
    p_{s_{ij}, w_{ij}} &= v^{\ell(w_0) - 2(j-i+1)} + v^{\ell(w_0) - 2(j-i+2)} + \cdots + v^{\ell(w_0) - 2(n+1-i)} = v^{2(i-1)}(v^{\ell(w_0) - 2j} +p_{s_{ij}, u_{ij}}) %v^{-j+i} \cdot  p_{e,w_{i-1,j}} 
    , \\
    p_{t_{ij}, u_{ij}} &= v^{-2i} \cdot p_{s_{ij}, u_{ij}}% v^{-2i} \cdot p_{j,u_{jj}} 
    , \\
    p_{t_{ij}, w_{ij}} &= p_{s_{ij}, u_{ij}} . %p_{j,u_{jj}}
\end{align}
%
\begin{comment}
\begin{align}
    p_{s_{ij}, u_{ij}} &= p_{j,u_{jj}}  = v^{\ell(w_0) - 2(j+1)} + v^{\ell(w_0) - 2(j+2)} + \cdots + v^{\ell(w_0) - 2n} , \\
    p_{s_{ij}, w_{ij}} &= v^{-j+i} \cdot  p_{e,w_{i-1,j}} , \\
    p_{t_{ij}, u_{ij}} &= v^{-2i} \cdot p_{j,u_{jj}} , \\
    p_{t_{ij}, w_{ij}} &= p_{j,u_{jj}}
\end{align}
\end{comment}
\end{lemma}
\begin{proof}
Follows from Lemma \ref{KLBlemma} by various inductions.
\end{proof}

\begin{corollary}\label{Bk=1}
For $0 \leq i \leq j \leq n$ and $0<j$, we have %(recall that the graded lift of $\co_x$ is $\Delta_e/(\Delta_x\langle \ell(x)\rangle)$)
\[[\co_{b_{\times,1}}:L_{u_{ij}}] = 0 \qquad \text{ and } \qquad [\co_{b_{\times,1}}:L_{w_{ij}}] = 1,\]
and, moreover, the composition factor $L_{w_{ij}}$ appearing in  $\co_{b_{\times,1}}$ is in the minimal degree among the composition factors $L_{w_{ij}}$ in $\Delta_e$.
\end{corollary}
\begin{proof}
Comparing Proposition~\ref{lem:B_KL} with Proposition~\ref{Bbut1} gives the claim.
\end{proof}

\begin{proposition}\label{Bsocab}
Let $1\leq i\leq j\leq n$. For each $1 \leq k \leq n+1-j$, and for
the corresponding join-irreducible elements 
$b_{\circ,k}$ and $b_{\times,k}$ in $\JI{i}{j}$, we have
\begin{align}
\label{soc:B_OA}
    \soc \Delta_e / \Delta_{b_{\circ,k}} &\cong \begin{cases}
    L_{u_{ij}} \langle -\ell(u_{ij}) +2(n+1-j-k)\rangle,
    & b_{\circ,k} \text{ is of type } O_A; \\[.75em]
    (L_{u_{ij}} \oplus L_{w_{ij}}) \langle -\ell(u_{ij}) +2(n+1-j-k)\rangle,     &  b_{\circ,k} \text{ is of type } O_B;    \end{cases} %\\[.75em]
\end{align}
and
\begin{equation}\label{bbx}
    \soc \Delta_e / \Delta_{b_{\times,k}} \cong L_k \langle -\ell(w_{ij}) +2(n+1-j-k)\rangle
\end{equation}
for either $L_k \cong L_{w_{ij}}$ or $L_k \cong L_{u_{ij}}$. 
% 
%\begin{enumerate}
%    \item The socle of $\co_{w_{\circ,r}}$ of type $O_A$ is the isotypic subquotient $L_\circ\langle -\ell(\circ)+2(n-a+1-r) \rangle$ of $\Delta_e$;
%\item The socle of $\co_{w_{\times,r}}$ of type $X$ is the isotypic subquotient $L_\times\langle -\ell(\times)+2(n-a+1-r) \rangle$ of $\Delta_e$;
%\item The socle of $\co_{w_{\circ,r}}$ of type $O_B$ is the sum of two isotypic subquotient $(L_\circ\oplus L_\times)\langle -\ell(\circ)+2(n-a+1-r) \rangle$ of $\Delta_e$.
%\end{enumerate}
\end{proposition}

\begin{remark}
%\hk{if we do not prove the final proposition we remark here the following}
For $k=1$ and $k>n+1-i-j$, we have $L_k\cong L_{w_{ij}}$ by Corollary~\ref{Bk=1} and for degree reasons. 
We conjecture that $L_k\cong L_{w_{ij}}$ is the case for all $k$, but to prove this in general, 
a computation of KL polynomials similar to Lemma~\ref{lem:B_KL}, e.g., 
using Lemma~\ref{KLBlemma}, seems required. 
\end{remark}

\begin{remark}\label{soclesumfk}
In the proof, we also prove 
\begin{equation}\label{eqf}
    \so_{f_k} = L'_k\langle\ell(u_{ij})+2(n+1-j-k)\rangle \oplus L_k \langle -\ell(w_{ij}) +2(n+1-j-k)\rangle %= L_{u_{ij}}\langle\ell(u_{ij})+2(n+1-j-k)\rangle \oplus \so_{b_{\times,k}},
\end{equation}
with $L'_k \cong L_{w_{ij}}$, if $L_{k-i}\cong L_{u_{ij}}$, and $L'_k \cong L_{u_{ij}}$, if $L_{k-i}\cong L_{w_{ij}}$, where $f_k$ is from Proposition~\ref{newelementsofimportance}. Thus the socle-sum property does not holds for $f_k$ when $k>i$.
\end{remark}

\begin{proof}
We use Proposition~\ref{Bbut1} and Lemma~\ref{socledie} throughout the proof without referring to it.

Consider the chain \eqref{bbfbfb} which we rename $x_1<\cdots <x_{2n-2j+2}$. Since the total multiplicity of composition factors in $\Delta_e$ isomorphic to shifts of $L_y$ with $y\in {}^i\hc^j$ agrees with $2n-2j+2$ by Proposition~\ref{Bnumber}, we have
\[[\co_{x_m}: L_u ]+[\co_{x_m}: L_w ] = m,\] 
for each $1\leq m\leq 2n-2j+2$ (see the proof of Proposition~\ref{newbgchain}), where $u= u_{ij}$ and $w=w_{ij}$.
Here we use that both subquotients $L_u$ and $L_w$ of $\Delta_e$ are graded multiplicity free.
By Proposition~\ref{BY}, for each even number $m$, the relation $x_{m-1}<x_{m}$ is socle-killing. 
Thus, by Lemma \ref{socledie}, the socle of $\co_{x_m}$ is simple, for each even number $m$.
That is, each $\so_{b_{\times,k}}$ is simple.
Now ,Corollary~\ref{Bk=1} gives \eqref{bbx}, for $k=1$. Since $b_{\times,k}<b_{\times,k+1}$ is socle-killing, Lemma~\ref{lem:bruh_soc}\eqref{sk m<m} establishes \eqref{bbx} for all $k$.

Similarly, each subquotient $\Delta_{x_{m-2}}/\Delta_{x_{m}}$ contain exactly two penultimate composition factors (i.e., composition factors isomorphic, up to shift, either to $L_u$ or to $L_w$). By the socle-killing relations and Lemma~\ref{socledie}, this subquotient must contain the maximal degree penultimate component in $\Delta_{e}/\Delta_{x_m}$.
Also, the socle-killing property shows that each $\so_{f_k}$ has length at most two. 
By \eqref{bbx} and degree comparison, we conclude that 
\begin{equation}\label{fkprelim}
\so_{f_k} = L'\langle - d \rangle  \oplus L\langle -\ell(w) +2(n+1-j-k )\rangle = L'\langle - d \rangle  \oplus \so_{b_{\times,k}}.    
\end{equation}
Here $L$ and $L'$ are just notation for the simple subquotients in the socle.
Then \eqref{eqJMf} and Proposition~\ref{halfsosumB} implies that the component $L'\langle -d\rangle$ is contained in $\so_{b_{\circ,k+1}}$.

We prove, by induction on $l = n+1-j-k = 1,\cdots ,n-j$, that 
\begin{equation}\label{eqbo}
    \max\deg(\so_{b_{\circ,k+1}}) = \ell(u) -2(l-1).
\end{equation}
Note that Lemma~\ref{skandmaxd} gives the ``$\geq$'' inequality in Formula~\eqref{eqbo}. 
Let $l=1$. We claim that the ``$\leq$'' inequality in Formula~\eqref{eqbo} follows from the following facts:
\begin{itemize}
\item that there is at most one (graded) penultimate composition factor of $\Delta_e$ in 
each degree $> \ell(u)-2(l-1)$,
\item that such composition factors are the socle of $\co_{b_{\times,k'}}$, for some $k'\geq n+2-i-j$, by \eqref{bbx}, 
\item that, for each  $k'\geq n+2-i-j$ such that $b_{\circ,k+1}>b_{\times, k'}$, 
the latter is socle-killing by Proposition~\ref{BY}.
\end{itemize}
Indeed, let $l>1$, then the socle-killing relation $b_{\circ,k+1}<b_{\circ,k+2}$ and 
induction provide the desired bound.

From Formula~\eqref{eqbo}, we obtain Formula~\eqref{soc:B_OA}, for type $O_A$, immediately.
For type $O_B$, the (non-socle-killing) relation $b_{\circ,k}>b_{\times,k-i}$, Formula~\eqref{bbx}, and Lemma~\ref{lem:bruh_soc} \eqref{bru mm}, together with Formula~\eqref{eqbo}, imply Formula~\eqref{soc:B_OA}.

Finally, Formula~\eqref{eqbo} and the second equality in Formula~\eqref{fkprelim}
implies $d=\ell(u) - 2(n+1-j-k )$ in Formula~\eqref{fkprelim}.
Now, the socle-killing relation $b_{\times,k-i}<b_{\times,k}$ shows that $L'\not\cong L_k$, 
as claimed in Remark~\ref{soclesumfk}.
\end{proof}

\subsection{Type $D$}

We assume $(W,S)$ is of type $D_{n+2}$. 
We use the following labeling of $S$:
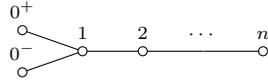
\[\begin{tikzpicture}[scale=0.4,baseline=-3]
\protect\draw (4 cm,0) -- (2 cm,0);
\protect\draw (2 cm,0) -- (0 cm,0);
\protect\draw (0 cm,0) -- (-2 cm,0);
\protect\draw (-2 cm,0) -- (-4 cm,0.7 cm);
\protect\draw (-2 cm,0) -- (-4 cm,-0.7 cm);
\protect\draw[fill=white] (4 cm, 0 cm) circle (.15cm) node[above=1pt]{\scriptsize $n$};
\protect\draw[fill=white] (2 cm, 0 cm) circle (0cm) node[above=1pt]{\scriptsize $\cdots$};
\protect\draw[fill=white] (0 cm, 0 cm) circle (.15cm) node[above=1pt]{\scriptsize $2$};
\protect\draw[fill=white] (-2 cm, 0 cm) circle (.15cm) node[above=1pt]{\scriptsize $1$};
\protect\draw[fill=white] (-4 cm, 0.7 cm) circle (.15cm) node[above=1pt]{\scriptsize $\Ou{}$};
\protect\draw[fill=white] (-4 cm, -0.7 cm) circle (.15cm) node[above=1pt]{\scriptsize $\Od{}$};
\end{tikzpicture}
\]

Since $\jc$ is strongly regular, i.e., all $\mathtt{H}$-cells inside $\mathcal{J}$ are singletons, we denote the unique element in ${}^i\hc^j$ as $w_{ij}$. These elements can be described as follows.
Denote by $\widehat{\cdot}$ the identity on $S$, if $n$ is even, and the unique automorphism of the Dynkin diagram that swaps $\Ou{} \leftrightarrow \Od{}$, if $n$ is odd. For simple reflections $i,j$, denote by $s_{ij} \in W$ the product $i \cdots j$ of simple reflections along the unique shortest path starting in $i$ and ending in $j$ in the Dynkin diagram.
Then
\[w_{ij} = s_{i \, \widehat{\jmath}} \cdot w_0. \]

The Bruhat graph of $\jc$, for $n$ even, is given on Figure \ref{fig:Bruhat_J_D_even}. The Bruhat graph of $\jc$, for $n$ odd, is obtained from Figure \ref{fig:Bruhat_J_D_even} by reversing all the arrows that start from or end in one of the following elements: $w_{\Od{}\Od{}}$, $w_{\Od{}\Ou{}}$, $w_{\Ou{}\Od{}}$, $w_{\Ou{}\Ou{}}$. Both for even and odd $n$, the rows are right cells, and the columns are left cells.

\begin{figure}

    \[ \scalebox{.9}{\xymatrix{
    w_{\Od{}\Od{}} & w_{\Od{}\Ou{}} \ar[r] \ar@/^1pc/[dd] & w_{\Od{}1} \ar@/^1pc/[ll]  \ar@/^1pc/[dd] & w_{\Od{}2} \ar@/^1pc/[dd] \ar[l] & w_{\Od{}3} \ar@/^1pc/[dd] \ar[l] & \cdots \ar[l] & w_{\Od{}n} \ar@/^1pc/[dd] \ar[l] \\
    w_{\Ou{}\Od{}} \ar[d] \ar@/_1pc/[rr] & w_{\Ou{}\Ou{}} & w_{\Ou{}1} \ar[l] \ar[d] & w_{\Ou{}2} \ar[l] \ar[d] & w_{\Ou{}3} \ar[l] \ar[d] & \cdots \ar[l] & w_{\Ou{}n} \ar[d] \ar[l] \\
    w_{1\Od{}} \ar@/_1pc/[rr] \ar@/_1pc/[uu] & w_{1\Ou{}} \ar[r] \ar[u] & w_{11} & w_{12} \ar[l] \ar[d] & w_{13} \ar[l] \ar[d] & \cdots \ar[l] & w_{1n} \ar[l] \ar[d] \\
    w_{2\Od{}} \ar[u] \ar@/_1pc/[rr] & w_{2\Ou{}} \ar[r] \ar[u] & w_{21} \ar[u] \ar[r] & w_{22} & w_{23} \ar[l] \ar[d] & \cdots \ar[l] & w_{2n} \ar[l] \ar[d] \\
    w_{3\Od{}} \ar[u] \ar@/_1pc/[rr] & w_{3\Ou{}} \ar[r] \ar[u] & w_{31} \ar[u] \ar[r] & w_{32} \ar[u] & w_{33} \ar[l] & \cdots \ar[l] & w_{3n} \ar[l] \ar[d] \\
    \vdots \ar[u] & \vdots \ar[u] & \vdots \ar[u] & \vdots \ar[u] & \vdots \ar[u] & \ddots & \vdots \ar[d] \\
   w_{n\Od{}} \ar[u] \ar@/_1pc/[rr] &  w_{n\Ou{}} \ar[r] \ar[u] & w_{n1} \ar[r] \ar[u] & w_{n2} \ar[r] \ar[u] & w_{n3} \ar[r] \ar[u] & \cdots \ar[r] & w_{nn}  
    }} \]
    \caption{Bruhat graph of the penultimate two-sided cell in type $D_{n+2}$ with $n$ even.}
    \label{fig:Bruhat_J_D_even}
\end{figure}
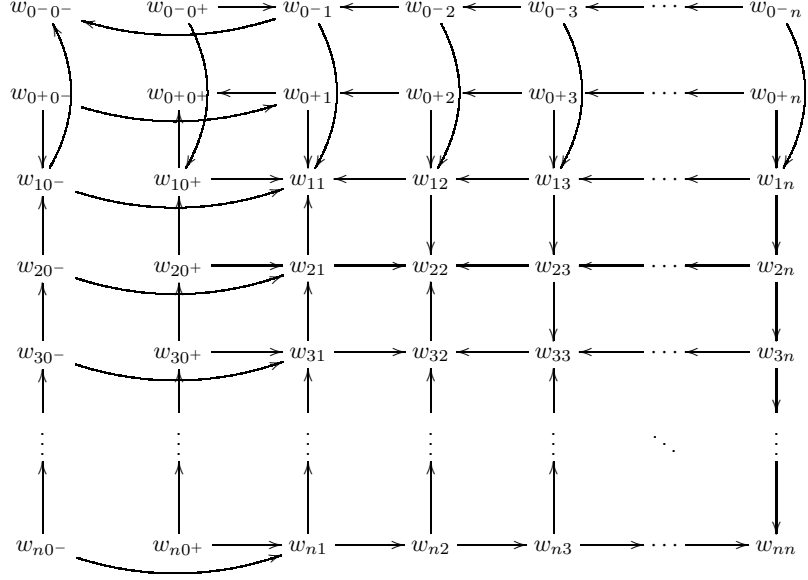

\subsubsection{Some Kazhdan-Lusztig computation in type D} \label{ssKLD}

Since the Coxeter number for $D_{n+2}$ is $h=2n+2$, we have $\ell(w_0)=(n+2)h/2=(n+2)(n+1)=n^2+3n+2$. We do not know any reference for an explicit formula for the value of the $\mathbf{a}$-function on $\jc$, so we will compute  this value below.
We start with the following  estimate.

\begin{lemma} \label{lem:D_a_half}
We have $\mathbf{a}(\jc)\geq \ell(w_0)-2n-1$.
\end{lemma}

\begin{proof}
Consider $I=\{\Od{},\Ou{},1,\cdots n-1\}\subset S$. We have $\ell(w_I)=(n+1)2n/2=n^2+n=\ell(w_0)-2n-2$. We want to prove that $w_I$ is not in $\jc$. If it were, then it would have to be in ${}^n \mathcal{H}^n$. But ${}^n \mathcal{H}^n$ has only one element $w_{nn} = n \cdot w_0$, whose length is $\ell(w_0)-1$, a contradiction.
Now, $w_I<_J \jc$ implies $\mathbf{a}(w_I) = \ell(w_I)  = \ell(w_0)-2n-2<\mathbf a(\jc)$, proving the claim.
\end{proof}

We will simplify the notation and use $p_{ij} := p_{e,w_{ij}}$, for $i,j \in S$.
The following is the base case of the inductive computation of the $p_{e,y}$, for $y\in \jc$.

\begin{proposition}\label{baseD}
For $y = w_{i\Ou{}}$, $w_{i\Od{}}$, $w_{\Ou{}i}$, or $w_{\Od{}i}$, where $i\geq 1$, we have
\begin{equation}
\label{eq:baseD}
    p_{e,y}=v^{\ell(y)}+v^{\ell(y)-2}+\cdots +v^{\ell(y)-2(n-i)}.
\end{equation}
\end{proposition}

\begin{proof}
Applying Lemma~\ref{prop2p} to $w_{i\Ou{}} \to w_{i-1,\Ou{}}$, for, respectively, $i=2,3,\ldots,n$ (see also Figure~\ref{fig:Bruhat_J_D_even}), we get a series of equalities:
\begin{equation} \label{eq:D_pt0_ind}
\begin{split}
    &(v+v^{-1}) \cdot p_{i\Ou{}} = p_{i-1,\Ou{}} + p_{i+1,\Ou{}}, \qquad i=2,3, \ldots, n-1, \\
    &(v+v^{-1}) \cdot p_{n\Ou{}} = p_{n-1,\Ou{}}.
\end{split}
\end{equation}

Suppose that $v^m$ is the smallest monomial that appears in $p_{n\Ou{}}$, and note that $m \leq \ell(w_{n\Ou{}}) = \ell(w_0) - n-1$. By induction, it is easy to see, using (\ref{eq:D_pt0_ind}), that $p_{1\Ou{}}$ must contain $v^{m-n+1}$ as a monomial. Since $w_{1\Ou{}}$ is not a Duflo element, we have $m-n+1 > \mathbf{a}(\jc) \geq \ell(w_0) - 2n-1$ (see Lemma \ref{lem:D_a_half}). This gives us $m=\ell(w_0) - n-1$, and therefore $p_{n\Ou{}} = v^{\ell(w_{n\Ou{}})}$. The statement of the proposition, for $y=w_{t\Ou{}}$, now follows from (\ref{eq:D_pt0_ind}) by a two-step induction.

The statement for the remaining elements follows from the symmetry of the Dynkin diagram, and from the fact that $p_{ij} = p_{ji}$, for $i,j \in S$.
%
%By \eqref{2p}, it is enough to show that $p_{e,y}$ is trivial for $y=(\Ou{},n)$. Suppose it is not.
%By applying \eqref{2p} for $(\Ou{},n-t)$, inductively for $t$, at the last step $p_{e,(\Ou{},1)}$ necessarily has a term $v^m$ for $m<l(\Ou{},1)-2n+2=l(w_0)-2n$.
%But since $(\Ou{},1)$ is not a Duflo element, $m>\mathbf{a}(\jc)\geq l(w_0)-2n-1$, which is a contradiction.
\end{proof}

From the proof of Proposition \ref{baseD} we also get:
\begin{corollary}
We have $\mathbf{a}(\jc) = \ell(w_0)-2n-1 = n^2+n+1$.
\end{corollary}

\begin{proposition}\label{Dmost}
For $1 \leq i \leq j \leq n$ and $y=w_{ij}$, 
%Letting $d(y):= n-j$ and $p_{d}:=1+v^{-2}+\cdots v^{-2d}$,
we have
\begin{equation}
\begin{aligned}\label{mostD}
% p_{e,y}=v^{\ell(y)}p_{d(y)}+v^{\ell(y)-2i} p_{d(y)}.
%    
    p_{ij} = p_{ji} = v^{\ell(y)} + v^{\ell(y)-2} + &\cdots + v^{\ell(y)-2(n-j)}  + \\
% R: It's maybe better to keep this in two lines. Both lines are monotonous in degree, and their intersection consist of higher multiplicity monomials.
    &+ v^{\ell(y)-2i} + v^{\ell(y)-2i-2} + \cdots + v^{\ell(y)-2(n-j+i)} .
\end{aligned}
\end{equation}
\end{proposition}
Note that higher multiplicities appear in (\ref{mostD}) if and only if $i+j \leq n$. These multiplicities are at most $2$.
\begin{proof}
We apply Lemma \ref{prop2p} to $w_{\Ou{}j} \to w_{1j}$ to get
\begin{equation*}
%\label{Dfirst}
    (v+v\inv) \cdot p_{\Ou{}j}=p_{1j}.
\end{equation*}
Proposition \ref{baseD} applied to $p_{\Ou{}j}$ above proves our claim (\ref{mostD}) for $i=1$.
Applying again Lemma \ref{prop2p} to $w_{1j} \to w_{2j}$ and $j \geq 2$, we get
\begin{equation*}
%\label{s1}
    (v+v\inv) \cdot p_{1j}=p_{2j} + p_{\Ou{}j}+ p_{\Od{}j}.
\end{equation*}
Here we apply Proposition \ref{baseD} and (\ref{mostD}) for $i=1$, which proves \eqref{mostD} for $i=2$. For $i > 2$ we apply Lemma \ref{prop2p} to $w_{i-1,j} \to w_{ij}$ and $j \geq i$ to get
\begin{equation*}
%\label{s2}
    (v+v\inv) \cdot p_{i-1,j}=p_{ij} + p_{i-2,j},
\end{equation*}
where the induction hypothesis establishes \eqref{mostD}.
The proof is complete.
\end{proof}

Finally, we treat the remaining four cases where $i,j\in\{\Ou{},\Od{}\}$. The result depends on the parity of the rank.

\begin{proposition}\label{Dspecial}
For $n$ even, we have
\begin{equation}\label{00even}
    p_{\Ou{}\Ou{}} = p_{\Od{}\Od{}}=v^{\ell(w_0)-1}+v^{\ell(w_0)-5} + \cdots +v^{\ell(w_0)-2n-1},
\end{equation}
\begin{equation}\label{00-even}
    p_{\Ou{}\Od{}}=p_{\Od{}\Ou{}}=v^{\ell(w_0)-3}+v^{\ell(w_0)-7} + \cdots +v^{\ell(w_0)-2n+1}.
\end{equation}
%where the former polynomial has $\frac{n}{2}+1$ terms and the latter one $\frac{n}{2}$ terms.
%
For $n$ odd, we have
\begin{equation}\label{00odd}
    p_{\Ou{}\Ou{}} = p_{\Od{}\Od{}} =v^{\ell(w_0)-3}+v^{\ell(w_0)-7} + \cdots +v^{\ell(w_0)-2n-1},
\end{equation}
\begin{equation}\label{00-odd}
 p_{\Ou{}\Od{}}=p_{\Od{}\Ou{}} =v^{\ell(w_0)-1}+v^{\ell(w_0)-5} + \cdots +v^{\ell(w_0)-2n+1}.
\end{equation}
%where both polynomials have $\frac{n+1}{2}$ terms.
\end{proposition}

\begin{proof}
The equalities $p_{\Ou{}\Ou{}}=p_{\Od{}\Od{}}$ and $p_{\Ou{}\Od{}}=p_{\Od{}\Ou{}}$ follow from the symmetry of the Dynkin diagram.

Assume first that $n$ is even. Applying Lemma \ref{prop2p} to $w_{\Od{}\Ou{}} \to w_{1\Ou{}}$ yields $(v+v^{-1}) \cdot p_{\Od{}\Ou{}} = p_{1\Ou{}}$ (see Figure \ref{fig:Bruhat_J_D_even}), where the right hand side is known from Proposition \ref{baseD}. This gives (\ref{00-even}).

Applying Lemma \ref{prop2p} to $w_{1\Ou{}} \to w_{\Ou{}\Ou{}}$ yields $(v+v^{-1}) \cdot p_{1\Ou{}} = p_{\Ou{}\Ou{}} + p_{2\Ou{}} + p_{\Od{}\Ou{}}$ (see Figure \ref{fig:Bruhat_J_D_even}). Now we use Proposition~\ref{baseD} and (\ref{00-even}) to get (\ref{00even}).

Assume now that $n$ is odd. We apply Lemma \ref{prop2p} to $w_{\Ou{}\Od{}} \to w_0 = \Ou{} \, w_{\Ou{}\Od{}}$, which gives $(v+v^{-1} )\cdot p_{\Ou{}\Od{}} = v^{\ell(w_0)} + p_{1\Od{}}$. From this, one easily gets (\ref{00-odd}).

Finally, we apply Lemma \ref{prop2p} to $w_{1\Ou{}} \to w_{\Od{}\Ou{}}$, which gives $(v+v^{-1}) \cdot p_{1\Ou{}} = p_{\Od{}\Ou{}} + p_{2\Ou{}} + p_{\Ou{}\Ou{}}$. From this, one easily gets (\ref{00odd}).
%
\begin{comment}
%Old version:
Applying \eqref{2p} we get
\begin{equation*}
    (v+v\inv)p_{e,(1,\Ou{})}=p_{e,(\Ou{},\Ou{})} + p_{e,(\Od{},\Ou{})} + p_{e,(2,\Ou{})},
\end{equation*}
which, by Proposition \ref{baseD} yields 
\begin{equation}\label{10}
    p_{e,(\Ou{},\Ou{})} + p_{e,(\Od{},\Ou{})}=v^{l(w_0)-1}+v^{l(w_0)-3}\cdots +v^{l(w_0)-2n-1}.
\end{equation}
By the same argument we get the same formula for the sums
\begin{equation*}
    p_{e,(\Ou{},\Od{})} + p_{e,(\Od{},\Od{})}=p_{e,(\Od{},\Ou{})} + p_{e,(\Od{},\Od{})}=p_{e,(\Ou{},\Od{})} + p_{e,(\Ou{},\Od{})}, 
\end{equation*}
from which we conclude 
\[ p_{e,(\Ou{},\Od{})} = p_{e,(\Od{},\Ou{})},\ \ \ \ p_{e,(\Ou{},\Ou{})} = p_{e,(\Od{},\Od{})}\]
(this can also be seen by symmetry of the Dynkin diagram).

Now consider \eqref{2p} at $y=(\Ou{},\Od{})$.
Note that we have $l(\Ou{},\Ou{})=l(\Od{},\Od{})$ and $l(\Ou{},\Od{})=l(\Od{},\Ou{})$ and the two lengths differ by 2.
If $n$ is even then $l(\Ou{},\Ou)=l(\Od{},\Od{})=l(w_0)-1>l(w_0)-3=l(\Ou{},\Od{})=l(\Od{},\Ou{})$; if $n$ is odd then they are swapped.
If $n$ is odd, then $y=0w_0$ and we have
\begin{equation*}
    (v+v\inv)p_{e,(\Ou{},\Od{})}=p_{e,w_0}+p_{e,(1,\Od{})}=v^{l(w_0)}+v^{l(w_0)-2}+\cdots +v^{l(w_0)-2n},
\end{equation*}
proving \eqref{00-odd}. Then \eqref{00odd} follows from \eqref{10}.
If $n$ is even, then
\begin{equation*}
    (v+v\inv)p_{e,(\Ou{},\Od{})}=p_{e,(1,\Od{})}=v^{l(w_0)-2}+v^{l(w_0)-2}+\cdots +v^{l(w_0)-2n},
\end{equation*}
implying \eqref{00-even}.
From this and \eqref{10}, we obtain \eqref{00even}.
\end{comment}
%
\end{proof}

%\begin{remark}
%One can easily check using (\ref{eq:baseD}),  (\ref{mostD}) and Proposition \ref{Dspecial} that the number of graded composition factors in $\Delta_e$ isomorphic to $L_w$ for $w\in \jc$ adds up to $\frac{(n+1)(n+2)(2n+3)}{3}$, see \cite[A006331]{OEIS}. There does not seem to be any direct geometric explanation similar to the one in Remark \ref{octahedron} and \cite[4.1]{kmm2}.
%\end{remark}

\subsubsection{Join-irreducibles in type D}\label{ss:JID}

Consider the type $B$ and type $D$ Weyl groups $\WB$ and $\WD$ with Dynkin diagrams as in Subsection~\ref{ssKLB} and Subsection~\ref{ssKLD}.
Define a map
\[\phi \colon \WB\to \WD\]
by $0\mapsto \Ou{}\Od{}=\Od{}\Ou{}$ and $i\mapsto i$ for $1\leq i\leq n$. Since $\Ou{}\Od{}1\Ou{}\Od{}1=1\Ou{}\Od{}1\Ou{}\Od{}$, this defines a group homomorphism.
%\todo{and reduced expressions go to reduced expressions?}. If $0\not\in LD(y)\cup RD(y)$, then 
%\begin{equation}\label{BtoDdescents}
%LD(\phi(y))=LD(y)\ \ \ \ \ \text{    and    }\ \ \ \ \ RD(\phi(y))=RD(y), 
%\end{equation} where we identify the Coxeter generators $1,\ldots, n$ in $\WB$ and in $\WD$.

We denote by $\JI{}{}$ the join-irreducibles in $\WD$, by $\JI{}{}^B$ the join-irreducibles in $\WB$. % and by $\JI{i}{j}$ and $\JI{i}{j}^B$ the sub(po)sets of $\JI{}{}$, respectively $\JI{}{}^B$, with left and right descents $\{i\}$ and $\{j\}$ respectively. By $\BG{}{}$ we denote the bigrassmannians in $\WD$, and  analogous sub/super-script versions.
From the braids of the elements in $\JI{}{}^B$ given in Subsection~ \ref{ss:JIB}, we see that no element in $\JI{}{}^B$ has a reduced expression that contains consecutive simple reflections $0101$, and so, in particular, no two reduced expressions of an element in $\JI{}{}^B$ are related by the type $B$ braid relation $0101=1010$. Therefore, we have two well-defined maps
\[\phi^{\pm} \colon \JI{}{}^B \to \WD\]
which are given by taking a reduced expression and replacing every other appearance of `$0$' by $\Ou{}$ and $\Od{}$ alternatingly. The result is an expression %\todo{Why reduced?}
of an element in $\WD$. We let $\phi^+(w)$ has $\Ou{}$ at the first (leftmost) appearance of $0$ in a reduced expression of $w$, and $\phi^-(w)$ has $\Od{}$.
%Then we have for $y\in \JI{i}{j}^B$ 
%\begin{equation}\label{pmdescents}
%    LD(\phi^\pm(y))=\begin{cases}
%    \{0^\pm\}\ &\text{ if }j=0\\
%    \{j\} &\text{ otherwise}.
%    \end{cases}
%\end{equation}
%
For example,
\begin{align*}
&\phi(0102103210) = \Ou{}\Od{}1\Ou{}\Od{}21\Ou{}\Od{}321\Ou{}\Od{} , \\
&\phi^+(0102103210) = \Ou{}1\Od{}21\Ou{}321\Od{}, \\
&\phi^-(0102103210) = \Od{}1\Ou{}21\Od{}321\Ou{} .
\end{align*}
%
%Note that, by construction, the maps $\phi$ and $\phi^\pm$ preserve the Bruhat order.
Note that $\phi^+$ and $\phi^-$ differ by the automorphism of the Dynkin diagram which swaps $\Ou{} \leftrightarrow \Od{}$ and preserves the other simple roots.

\begin{proposition}
\label{BtoD0}
For $j\in S$, define the following elements in $\WD$:
\[ d_k^{0^{\pm}j} := \phi^{\pm} ( b_k^{0j}), \qquad   1 \leq k \leq n+1-j. \]
\begin{enumerate}
\item Then we have finite sets
\begin{align}
\label{al:JI0+0+}    & \JI{\Ou{}}{\Ou} = \big\{ d_1^{0^{+}0} < d_3^{0^{+}0} < d_5^{0^{+}0} < \ldots \big\}, \\
\label{al:JI0+0-}    & \JI{\Ou{}}{\Od} = \big\{ d_2^{0^{+}0} < d_4^{0^{+}0} < d_6^{0^{+}0} < \ldots \big\},
\end{align}
where the last elements above are $d_{n}^{0^{+}0}$ and $d_{n+1}^{0^{+}0}$, depending on the parity of $n$. 
\item For $j\geq 1$, we have
\begin{equation}
\label{al:JI0+j}    \JI{\Ou{}}{j} = \big\{ d_1^{0^{+}j} < d_2^{0^{+}j} < d_3^{0^{+}j} < \ldots < d_{n+1-j}^{0^{+}j} \big\}.
\end{equation}
\item Analogous statements hold for $\JI{\Od{}}{j}$.
\end{enumerate}
\end{proposition}

\begin{proof}
Note that by applying $\phi^\pm$ to the reduced expressions of elements $b^{0j}_k$ described in Proposition \ref{BJI0a}, we get exactly the reduced expressions of the elements from \cite[Theorem 5.7]{GeK}, from which the rest follows easily.
\end{proof}
%
%Usually we will fix left and right descents of join-irreducible elements, and omit writing the upper indices over $d_k$ to ease the notation.

%When $1\leq i\leq j\leq n$, we set for all $1\leq k \leq n+1-j$
%\begin{equation}
%    d_{\circ,k} := \phi ( b_{\circ,k}^{ij} ) \in \BG{i}{j}.
%\end{equation}

\begin{proposition}
\label{JIabD}
Let $1\leq i\leq j\leq n$. The set $\JI{i}{j}$ consists of the following three types of elements:
\begin{itemize}
    \item (type $O_A$)  for $1\leq k\leq \min \{i,n+1-j\}$:
    \[    d_{\circ,k}^{ij}  = d_{\circ,k} = \phi ( b_{\circ,k}^{ij} ) = \phi^+ ( b_{\circ,k}^{ij} ) =\phi^- ( b_{\circ,k}^{ij} ) , \] 
    \item (type $O_B$)  for $i< k\leq n+1-j$:
    \[d_{\pm ,k}^{ij} = d_{\pm ,k} := \phi^{\pm} ( b_{\circ,k}^{ij})  , \] 
    \item(type $X$)  for $1\leq k\leq n+1-j$:
    \[ d_{\times ,k}^{ij} = d_{\times ,k}:= \phi ( b_{\times,k}^{ij}).    \] 
\end{itemize}
Moreover, the partial order on $\JI{i}{j}$ 
is generated by the following relations (see Figure \ref{fig:BD_JI_gen}):
\begin{align}
\label{al:D_OAOA}    & d_{\circ,k} < d_{\circ,k+1}, & & 1 \leq k \leq \min \{i-1,n-j\}, \\
   & d_{\circ,i} < d_{\pm,i+1}, && \text{if } i\leq n-j,    \\
    & d_{\pm, k} < d_{\pm, k+1}, \ d_{\pm, k} < d_{\mp, k+1}, && i < k \leq n-j,   \\
   & d_{\times,k} < d_{\times,k+1}, & & 1 \leq k \leq n-j, \\
\label{al:D_OAX}    & d_{\circ,k} < d_{\times,k}, & & 1 \leq k \leq \min\{i,n+1-j\}, \\
\label{al:D_OBX}    & d_{\pm,k} < d_{\times,k} && i < k \leq n+1-j, \\
\label{al:D_XOB}   & d_{\times,k} < d_{\pm,k+i+1}, & & 1 \leq k \leq n-j-i.
\end{align}
\end{proposition}
\begin{proof}
The fact that $\JI{i}{j}$ consists of the elements as above follows from \cite[Theorem 5.7]{GeK}, as well as the fact that applying the maps $\phi$ and $\phi^\pm$ to the reduced expression of the corresponding element, we get a reduced expression. The Bruhat relations (\ref{al:D_OAOA})--(\ref{al:D_OBX}) follow directly the construction of the elements. 
For the relation  (\ref{al:D_XOB}), we see from the reduced expressions above that it is enough to prove it in the case $i=j$, which follows from e.g. \cite[Theorem 8.2.8]{BB2}.

To prove that (\ref{al:D_OAOA})--(\ref{al:D_XOB}) are essentially all relations in $\JI{i}{j}$, it is enough to observe the following:
\begin{itemize}
    \item $d_{\circ,k} \not\le d_{\times,k-1}$, for $1 \leq k \leq \min\{ i,n+1-j \} $, since the former one involves the simple reflection $(j+k-1)$ in its reduced expressions while the latter does not,
    \item $d_{\pm,k} \not\le d_{\times,k-1}$, for $i \leq k \leq n+1-j$, for the same reason as above,
    \item $d_{\times,k} \not\leq d_{\pm,k+i}$, for $1 \leq k \leq n+1-j-i$, which can be checked using \cite[Theorem 8.2.8]{BB2}.
    \qedhere
\end{itemize}
\end{proof}

\begin{lemma}\label{DY}
The Bruhat relations in (\ref{al:JI0+0+})--(\ref{al:JI0+j}) and (\ref{al:D_OAOA})--(\ref{al:D_OBX}) are socle-killing.
\end{lemma}

\begin{proof}
This is a routine check of all the cases.
\end{proof}

We note that the formulation of Lemma~\ref{DY} lists all  the
relations from Poposition~\ref{JIabD} except \eqref{al:D_XOB}.

As in type $B$, it helps to consider some of $\BG{i}{j}$ in addition to $\JI{i}{j}$:

\begin{proposition}\label{Dnewelementsofimportance}
Let $1\leq i\leq j$. Then the elements
\[g^{ij}_k = \phi(f^{ij}_k) \in W(D_{n+2})\]
belongs to $\BG{i}{j}$ and we have in $W(D_{n+2})$
\[d^{ij}_{\times,k}<g^{ij}_k <d^{ij}_{\times,k+1}.\]
\end{proposition}
\begin{proof}
The left and right descents of $\phi(f^{ij}_k)$ are $\{i\}$ and $\{j\}$, respectively, by construction of $\phi$.
The second claim follows from $b^{ij}_{\times,k}<f^{ij}_k<b^{ij}_{\times,k}$ in $W(B_{n+1})$ since the map $\phi$ preserves the order relations.
\end{proof}

\begin{lemma}\label{gdsk}
Let $1\leq i\leq j$. The relation $g^{ij}_k <d^{ij}_{\times,k+1}$ is socle-killing.
\end{lemma}
\begin{proof}
This follows from the fact that $f=f^{ij}_k <b^{ij}_{\times,k+1}=b$ is socle-killing.
In fact, the former implies either $f<s_ib$ or $f<bs_j$. If $f<s_ib$, 
then \[g^{ij}_k=\phi(f)<\phi(s_ib)=\phi(s_i)\phi(b)=s_id^{ij}_{\times,k+1},\]
and the second case is similar.
\end{proof}

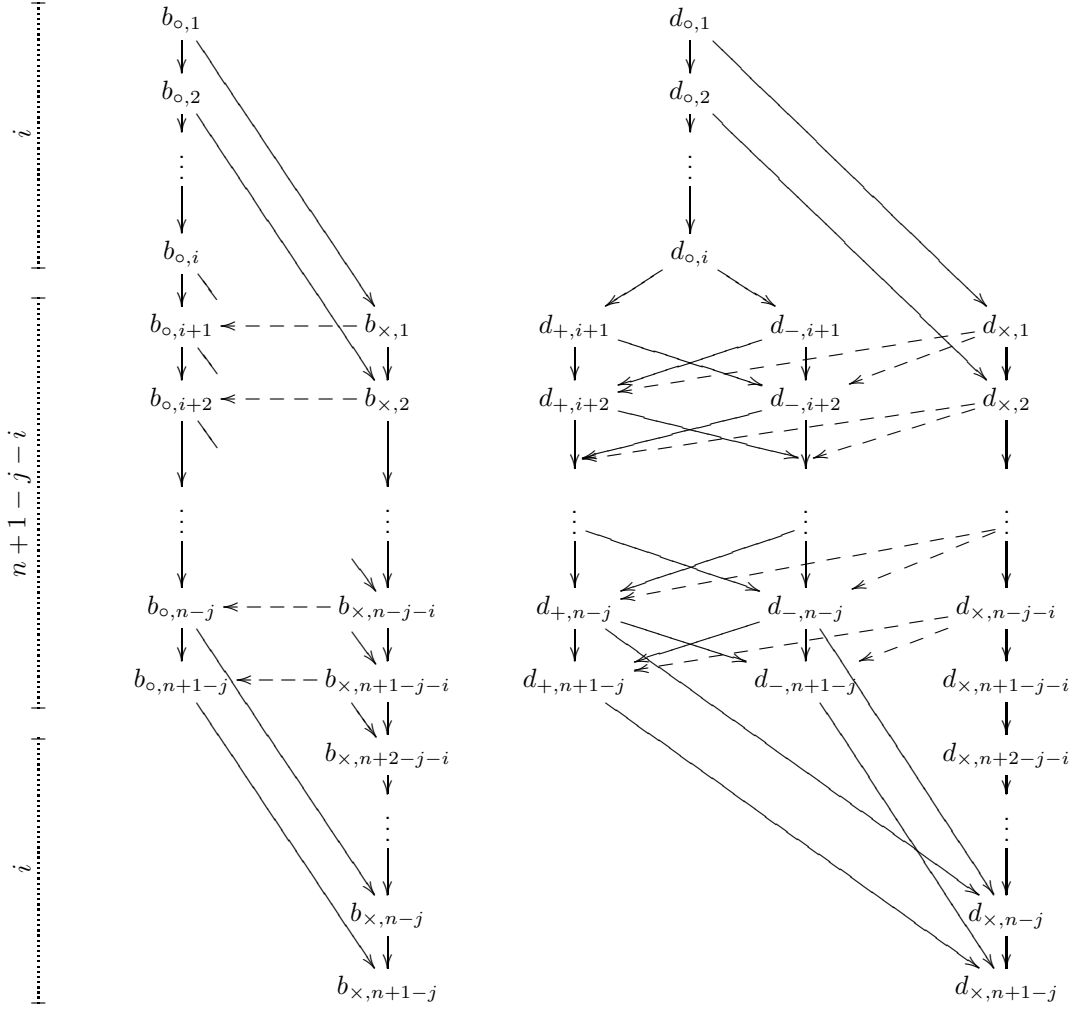
\begin{figure}
\centering
\begin{subfigure}{.4\textwidth}
  \centering
\[  \xymatrix@C=3em@R=1.2em{
\ar@{}[ddd]^(-0.1){}="a"^(1.1){}="b" \ar@{|..|}_{\rotatebox[origin=c]{90}{$i$}} "a";"b" &    b_{\circ,1} \ar[d] \ar[rdddd]  & \\
&   b_{\circ,2} \ar[d] \ar[rdddd]  & \\
&   \vdots \ar[d]  &   \\
&   b_{\circ,i} \ar[d] \ar@{}[rdddd]^(0.05){}="a"^(0.2){}="b" \ar@{-} "a";"b"   &   \\
\ar@{}[ddddd]^(-0.1){}="a"^(1.1){}="b" \ar@{|..|}_{\rotatebox[origin=c]{90}{$n+1-j-i$}} "a";"b" &   b_{\circ,i+1} \ar[d]  \ar@{}[rdddd]^(0.05){}="a"^(0.2){}="b" \ar@{-} "a";"b" & b_{\times,1} \ar@{-->}[l]\ar[d] \\
&   b_{\circ,i+2} \ar[dd] \ar@{}[rdddd]^(0.05){}="a"^(0.2){}="b" \ar@{-} "a";"b" & b_{\times,2} \ar@{-->}[l]\ar[dd] \\
&    &   \\
&   \vdots \ar[d] & \vdots \ar[d]  \\
&   b_{\circ,n-j} \ar[d] \ar[rdddd] & b_{\times,n-j-i} \ar@{-->}[l] \ar[d] \ar@{}[luuuu]^(0.03){}="a"^(0.2){}="b" \ar@{<-} "a";"b" \\
&   b_{\circ,n+1-j} \ar[rdddd] & b_{\times,n+1-j-i} \ar@{-->}[l] \ar[d] \ar@{}[luuuu]^(0.03){}="a"^(0.2){}="b" \ar@{<-} "a";"b" \\
\ar@{}[ddd]^(-0.1){}="a"^(1.1){}="b" \ar@{|..|}_{\rotatebox[origin=c]{90}{$i$}} "a";"b" &   & b_{\times,n+2-j-i}  \ar[d] \ar@{}[luuuu]^(0.03){}="a"^(0.2){}="b" \ar@{<-} "a";"b" \\
&   & \vdots \ar[d]  \\
&   & b_{\times,n-j} \ar[d] \\
&   & b_{\times,n+1-j} }  \]
%  \caption{}
%  \label{fig:B_JI_gen}
\end{subfigure}%
\begin{subfigure}{.5\textwidth}
  \centering
\[  \xymatrix@C=1em@R=1.2em{
& d_{\circ,1} \ar[d] \ar[rrrdddd]  &&& \\
& d_{\circ,2} \ar[d] \ar[rrrdddd]  &&& \\
& \vdots \ar[d]  &&&   \\
& d_{\circ,i}  \ar[rd] \ar[ld]  &&&   \\
d_{+,i+1} \ar[d] \ar[drr]  && d_{-,i+1} \ar[d] \ar[dll]    && d_{\times,1} \ar@{-->}[lllld] \ar@{-->}[lld] \ar[d] \\
d_{+,i+2} \ar@{}[dd]^(0.1){}="a"^(0.6){}="b" \ar@{->} "a";"b" \ar[drr] && d_{-,i+2} \ar@{}[dd]^(0.1){}="a"^(0.6){}="b" \ar@{->} "a";"b" \ar[dll]  && d_{\times,2} \ar@{}[dd]^(0.1){}="a"^(0.6){}="b" \ar@{->} "a";"b" \ar@{-->}[lllld] \ar@{-->}[lld] \\
 &&  && \\
\vdots \ar[d] \ar[drr] && \vdots \ar[d] \ar[dll] && \vdots \ar[d] \ar@{-->}[lllld] \ar@{-->}[lld] \\
d_{+,n-j} \ar[d] \ar[drr] \ar[rrrrdddd] && d_{-,n-j} \ar[d] \ar[dll] \ar[rrdddd] && d_{\times,n-j-i} \ar[d] \ar@{-->}[lllld] \ar@{-->}[lld]\\
d_{+,n+1-j} \ar[rrrrdddd] && d_{-,n+1-j} \ar[rrdddd] && d_{\times,n+1-j-i} \ar[d] \\
&&&& d_{\times,n+2-j-i}  \ar[d] \\
&&&& \vdots \ar[d]  \\
&&&& d_{\times,n-j} \ar[d] \\
&&&& d_{\times,n+1-j} }  \]
%  \caption{}
%  \label{fig:D_JI_gen}
\end{subfigure}
\caption{The Bruhat graph of $\JI{i}{j}$ in types $B_{n+1}$ (left) and $D_{n+2}$ (right), for $1 \leq i \leq j \leq n$ and $i \leq n-j$, with solid arrows socle-killing. In the right-hand side diagram, all but the first two arrows from (\ref{al:D_OAX}) are omitted, as well as all but the last four from (\ref{al:D_OBX}). The case $i > n-j$ is similar, but without type $O_B$ elements (i.e., no dashed arrows).}
\label{fig:BD_JI_gen}
\end{figure}

\begin{corollary}\label{Dnumber}
Let $i,j\in S$.
Then
\[ \displaystyle|\JI{i}{j}|= \begin{cases}  p_{e,w_{ij}}(1), &  i=0^\pm \text{ or } j=0^\pm ; \\ p_{e,w_{ij}}(1)  + c_{ij}, & 1 \leq i,j \leq n; \end{cases}\]
where $c_{ij} := \max\{0, n+1-j-i\}$.
\end{corollary}
\begin{proof}
The claim follows from comparing Proposition \ref{BtoD0} and \ref{JIabD} with Proposition  \ref{baseD}, \ref{Dspecial}, and \ref{Dmost}. Note that $c_{ij}$ is equal to the one half of the number of type $O_B$ elements, and also to the number of monomials in $p_{ij}$ with coefficient $2$.
\end{proof}

\begin{proposition}\label{Dbgchain}
For any $i,j\in S$, there exists a chain of length $p_{e,w_{ij}}(1)$ in $\BG{i}{j}$.
\end{proposition}
\begin{proof}
When $i=0^\pm$ or $j=0^\pm$, the elements in $\JI{i}{j}$ form a chain, as shown in Proposition~\ref{BtoD0}, of the desired length.

Otherwise, Proposition~\ref{JIabD} and Proposition~\ref{Dnewelementsofimportance} provides the chain
\begin{equation}\label{ddgdgd}
  d_{\circ,1}<d_{\times,1}<g_1<d_{\times,2}<g_2<\cdots < g_{n-\max \{i,j\}} < d_{\times,n-\max \{i,j\}+1}  
\end{equation}
which has the desired length by Corollary~\ref{Dnumber}.
\end{proof}

\begin{corollary}\label{halfsosumD}
Let $w\in W$. The socle of $\co_w$ is contained in the sum of $\so_{x}$ taken  over $x\in \JM(w)$.
\end{corollary}
\begin{proof}
This follows from Proposition~\ref{Dbgchain} and Proposition~\ref{cor:sum}.
\end{proof}

\subsubsection{The join-irreducible socles}

\begin{proposition}\label{socle0D}
Let $i =0^{\pm}$ and $j\in S$. Then, for $d_k \in \JI{0^{\pm}}{j}$, we have
%\[\soc\co_{\phi^\pm (b_{k})} = L_{w_{ij}}\langle -\ell(w_{ij})+2(n+1-k) \rangle\]
%if $j=0^\pm$, or
\[\soc\co_{d_k} \cong \begin{cases} L_{w_{ij}} \langle -\ell(w_{0})+1+2(n+1-k) \rangle, &  j = 0^\pm;  \\
L_{w_{ij}} \langle -\ell(w_{ij})+2(n+1-j-k) \rangle, &  j\geq 1 . \end{cases} \]
%if $j\geq 1$.
\end{proposition}

\begin{proof}
Since $\JI{i}{j}$ is a chain with socle-killing relations by Proposition \ref{BtoD0} and Lemma \ref{DY}, and since $|\JI{i}{j}| = p_{ij}(1)$ by Proposition  \ref{Dspecial} and Propostion \ref{baseD}, the claim follows from Lemma \ref{socledie} and the pigeonhole principle.
\end{proof}

\begin{proposition}\label{socleabD}
Fix $1 \leq i \leq j \leq n$. For $d\in \JI{i}{j}$, the socle of $\co_d$ is isomorphic to some shift of $L_{w_{ij}}$. To be more precise,
\begin{enumerate}
    \item \label{socOA} for $1\leq k \leq \min\{i,n+1-j\}$, we have
    \[ \soc\co_{d_{\circ,k}} \cong L_{w_{ij}}\langle - \ell (w_{ij}) +2(n+1-j+i-k)\rangle ,\]
    \item \label{socOBX} for $i < k\leq n+1-j$, the socles  $\soc\co_{d_{+,k}}$, $\soc\co_{d_{-,k}}$ and $\soc\co_{d_{\times,k-i}}$ are three distinct simple summands of the isotypic component $ 2L_{w_{ij}}\langle -\ell (w_{ij}) +2(n+1-j+i-k) \rangle $ in $\Delta_e$,
    \item \label{socX} for $n+1-j-i< k \leq n+1-j$, we have
    \[ \soc\co_{d_{\times,k}} \cong L_{w_{ij}}\langle -\ell (w_{ij}) +2(n+1-j-k) \rangle .\]
\end{enumerate}
%(Recall that $\mathbf{a}(x)=n^2+n+1=\ell(w_0)-2n-1$.)
\end{proposition}

\begin{proof}
Using Lemma~\ref{lem:KL_sz_y} and induction one obtains the type D analogue of 
Lemma~\ref{lem:B_KL} which shows that the penultimate simple multiplicity of maximal
degree in  $\co_{d_{\times,1}}$ equals $1$. This implies \eqref{socOBX} for $d_{\times,1}$. 
This allows us to follow the steps in the proof of Proposition~\ref{Bsocab}.

We use Proposition~\ref{Dmost} and Lemma~\ref{socledie} throughout the proof without referring to it and we set  $w=w_{ij}$.

Consider the chain \eqref{ddgdgd} which we rename $x_1<\cdots <x_{2n-2j+2}$. Since the total multiplicity of penultimate composition factors (i.e., composition factors isomorphic, up to shift, to $L_w$) in $\Delta_e$ agrees with $2n-2j+2$ by Proposition~\ref{Dmost}, we have
\[[\co_{x_m}: L_w ] = m,\] 
for each $1\leq m\leq 2n-2j+2$ (see the proof of Proposition~\ref{newbgchain}).
By Proposition~\ref{gdsk}, for each even number $m$, the relation $x_{m-1}<x_{m}$ is socle-killing, and thus by Lemma \ref{socledie} the socle of $\co_{x_m}$, for each even number $m$, is simple, that is, each $\so_{d_{\times,k}}$ are simple.
As proved in the first paragraph, we have $\so_{d_{\times,1}}$ as stated in the formulation, and, in particular, it is in the desired degree.
 Since $d_{\times,k}<d_{\times,k+1}$ is socle-killing, Lemma~\ref{lem:bruh_soc}\eqref{sk m<m} 
 shows that, for all $k$, the socles $\so_{d_{\times,k}}$ are as in \eqref{socOBX} and in \eqref{socX}.
 
We prove, by induction on $l = n+1-j-k = 1,\cdots ,n-j$, that 
\begin{equation}\label{eqboD}
    \max\deg(\so_{d_{\pm,k+1}}) = \ell(w)-2i -2(l-1)
\end{equation}
Note that Lemma~\ref{skandmaxd} gives ``$\geq$'' in \eqref{eqboD}. 
Let $l=1$. Then ``$\leq$'' in \eqref{eqboD} follows from by combining the following facts:
\begin{itemize}
\item that there is at most one (graded) penultimate composition factor of $\Delta_e$ in each grade $> \ell(w)-2i-2(l-1)$,
\item that each such composition factor is the socle of $\co_{d_{\times,k'}}$, for some $k'\geq n+2-i-j$, by the previous paragraph,
\item that $d_{\pm,k+1}>d_{\times, k'}$, if true, is socle-killing (see Proposition~\ref{DY}),
for each $k'\geq n+2-i-j$.
\end{itemize}
Now let $l>1$. Then the socle-killing relation $d_{\pm,k+1}<d_{\pm,k+2}$ and 
induction provides the desired bound.

The socle-killing relation $d_{\pm,k-1}<d_{\pm,k}$ and Lemma~\ref{lem:soc_low_bound} upgrades \eqref{eqboD} to
\[ \deg(\so_{d_{\pm,k+1}}) = \ell(w)-2i -2(l-1).\]
This implies \eqref{socOA} immediately. For \eqref{socOBX}, it remains to observe that 
either of $\so_{d_{\pm,k}}$ having the maximal possible multiplicity, namely $2$,
with respect to the composition factor $L_w$
in the maximal degeree contradicts that $d_{+,k}$ and $d_{-,k}$ are Bruhat incomparable. This proves the claims for $d_{\pm,k}$ and $d_{\circ,k}$. 

The proposition is proved.
\end{proof}

\subsection{Type $E_6$}\label{ss:E6}

Before we discuss type $E$ in general, it is useful to first look at type $E_6$ in detail. 
Let $(W,S)$ of type $E_6$.
We denote the simple reflections by
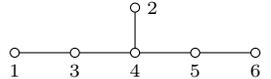
\[ \begin{tikzpicture}[scale=0.4,baseline=-3]
\protect\draw (8 cm,0) -- (6 cm,0);
\protect\draw (6 cm,0) -- (4 cm,0);
\protect\draw (4 cm,0) -- (2 cm,0);
\protect\draw (2 cm,0) -- (0 cm,0);
\protect\draw (4 cm,0) -- (4 cm,1.5 cm);

\protect\draw[fill=white] (8 cm, 0 cm) circle (.15cm) node[below=1pt]{\scriptsize $6$};
\protect\draw[fill=white] (6 cm, 0 cm) circle (.15cm) node[below=1pt]{\scriptsize $5$};
\protect\draw[fill=white] (4 cm, 0 cm) circle (.15cm) node[below=1pt]{\scriptsize $4$};
\protect\draw[fill=white] (2 cm, 0 cm) circle (.15cm) node[below=1pt]{\scriptsize $3$};
\protect\draw[fill=white] (4 cm, 1.5 cm) circle (.15cm) node[right=1pt]{\scriptsize $2$};
\protect\draw[fill=white] (0 cm, 0 cm) circle (.15cm) node[below=1pt]{\scriptsize $1$};
\end{tikzpicture}.\]
Since $\jc$ is strongly regular, i.e., all $\mathtt{H}$-cells inside $\mathcal{J}$ are singletons, we denote the unique element ${}^i\hc^j$ as $w_{ij}$. These elements can be given as follows.
Denote by $\widehat{\cdot}$ the unique non-trivial automorphism of the Dynkin diagram. For simple reflections $i,j$, denote by $s_{ij} \in W$ the product $i \cdots j$ of simple reflections along the unique shortest path starting in $i$ and ending in $j$ in the Dynkin diagram.
Then we have $w_{ij} = s_{i \, \widehat{\jmath}} \cdot w_0$.

The Kazhdan-Lusztig polynomials $p_{e,w}$, for $w \in \mathcal{J}$, are collected in Table \ref{tab:KLE6}. We put $p_{e,w}$ in the position $(i,j)$ of the table, where $i$ is the left, and $j$ the right ascent for $w$. The computation was performed in SageMath v.9.0.

\begin{table}[ht]
    \centering
\renewcommand{\arraystretch}{1.2}
\scalebox{1}{
\begin{tabular}{c}
\begin{tabular}{|l||l|l|l|l|l|l|} \hline
 & $1$ & $2$ & $3$ \\ \hline\hline
$1$ & $v^{31} + v^{25}$ & $v^{32} + v^{28}$ & $v^{32} + v^{30} + v^{26}$ \\ \hline
$2$ & $v^{32} + v^{28}$ & $v^{35} + v^{31} + v^{29} + v^{25}$ & $v^{33} + v^{31} + v^{29} + v^{27}$  \\ \hline
$3$ & $v^{32} + v^{30} + v^{26}$ & $v^{33} + v^{31} + v^{29} + v^{27}$ & $v^{33} + 2v^{31} + v^{29} + v^{27} + v^{25}$ \\ \hline
$4$ & $v^{33} + v^{31} + v^{29} + v^{27}$ & $v^{34} + v^{32} + 2v^{30} + v^{28} + v^{26}$ & $v^{34} + 2v^{32} + 2v^{30} + 2v^{28} + v^{26}$ \\ \hline
$5$ & $v^{34} + v^{30} + v^{28}$ & $v^{33} + v^{31} + v^{29} + v^{27}$ & $v^{35} + v^{33} + v^{31} + 2v^{29} + v^{27}$ \\ \hline
$6$ & $v^{35} + v^{29}$ & $v^{32} + v^{28}$ & $v^{34} + v^{30} + v^{28}$ \\ \hline
%\hline
\end{tabular}
\\ \\
\begin{tabular}{|l||l|l|l|l|l|l|} \hline
 & $4$ & $5$ & $6$ \\ \hline\hline
$1$ &  $v^{33} + v^{31} + v^{29} + v^{27}$ & $v^{34} + v^{30} + v^{28}$ & $v^{35} + v^{29}$ \\ \hline
$2$ &  $v^{34} + v^{32} + 2v^{30} + v^{28} + v^{26}$ & $v^{33} + v^{31} + v^{29} + v^{27}$ & $v^{32} + v^{28}$ \\ \hline
$3$ &  $v^{34} + 2v^{32} + 2v^{30} + 2v^{28} + v^{26}$ & $v^{35} + v^{33} + v^{31} + 2v^{29} + v^{27}$ & $v^{34} + v^{30} + v^{28}$ \\ \hline
$4$ &  $v^{35} + 2v^{33} + 3v^{31} + 3v^{29} + 2v^{27} + v^{25}$ & $v^{34} + 2v^{32} + 2v^{30} + 2v^{28} + v^{26}$ & $v^{33} + v^{31} + v^{29} + v^{27}$ \\ \hline
$5$ &  $v^{34} + 2v^{32} + 2v^{30} + 2v^{28} + v^{26}$ & $v^{33} + 2v^{31} + v^{29} + v^{27} + v^{25}$ & $v^{32} + v^{30} + v^{26}$ \\ \hline
$6$ & $v^{33} + v^{31} + v^{29} + v^{27}$ & $v^{32} + v^{30} + v^{26}$ & $v^{31} + v^{25}$ \\ \hline
\end{tabular}
\end{tabular}}
\vskip 10pt
    \caption{Kazhdan-Lusztig polynomials $p_{e,w}$, for $w \in \mathcal{J}$, in $E_6$.}
    \label{tab:KLE6}
\end{table}

%\subsection{Type $E_6$}\label{ss:JIE6}
%We assume $(W,S)$ is of type $E_6$. Recall the notation from Subsection~\ref{ssKLE6}. 
The posets $\JI{1}{j}$, for $j=1,2,\ldots,6$, are given in Figure \ref{fig:E6_1j}, with all the relations socle-killing. 

\begin{figure}
    \centering
\[ \scalebox{.8}{\xymatrix{ 1 \ar[d] \\ 13452431 } \qquad
\xymatrix{ 1342 \ar[d] \\ 13456245342 } \qquad
\xymatrix{ 13 \ar[d] \\ 1345243 \ar[d] \\ 134562453413 } \qquad
\xymatrix{ 134 \ar[d] \\ 134524 \ar[d] \\ 1345624534 \ar[d] \\ 13456245341324 }   \qquad
\xymatrix{ 1345 \ar[d] \\ 13456245 \ar[d] \\ 134562453413245 } \qquad 
\xymatrix{ 13456 \ar[d] \\ 1345624534132456 }  }\]
    \caption{The posets $\JI{1}{j}$, for $j=1,2,\ldots,6$, in type $E_6$.}
    \label{fig:E6_1j}
\end{figure}

In Figure \ref{fig:E6_23} and Figure \ref{fig:E6_44}, we present all the remaining $\JI{i}{j}$, up to symmetries, with the non socle-killing arrows dashed. Moreover, to the right of each element $x \in \JI{i}{j}$ in the diagrams, we assign a symbol of the form $cL\langle -d\rangle$, which means the following:
\begin{itemize}
\item the maximal degree component of $\Delta_e/\Delta_x$ is isomorphic to $L_{w_{ij}}\langle -d\rangle$ (calculated by computer),
\item $[\Delta_e \colon L_{w_{ij}}\langle-d\rangle] = c$ (see Table \ref{tab:KLE6}).
\end{itemize}

\begin{figure}
    \centering
\[ \scalebox{.7}{\xymatrix@C=.5em{ 2 \ar[d] & L\langle-25\rangle \\
245342 \ar@{-->}[d] & L\langle-29\rangle \\
{24563451342} \ar[d] & L\langle-31\rangle \\
245341324565432451342 & L\langle-35\rangle &    }  \qquad\qquad

\xymatrix@C=.5em{ 243 \ar[d] \ar[rdd] && L\langle-27\rangle  \\ 
2453413 \ar[rdd] && L\langle-29\rangle \\ 
&  2456345243 \ar[d] & L\langle-31\rangle \\ 
&  24563413245432413 & L\langle-33\rangle}    \qquad\qquad
\xymatrix@C=0em{ & 24 \ar[d] && L\langle-26\rangle \\
& 24534 \ar[dl]\ar[d]\ar[dr] && L\langle-28\rangle \\
245341324 \ar[rd] & 245634524 \ar[d] & 2456345134 \ar[ld] & 2L\langle-30\rangle\\
& 245634132454324 \ar[d] && L\langle-32\rangle \\
& 24534132456543245134 && L\langle-34\rangle}   }\]
\vskip .5cm
\[ \scalebox{.7}{ \xymatrix@C=.5em{ & 3 \ar[d] \ar[rdd] &&& L\langle-25\rangle \\
& 3413 \ar[ddl]\ar[dd]\ar[ddr] &&& L\langle-27\rangle\\
& &  345243 \ar[lld]\ar[ld]\ar[d] && L\langle-29\rangle\\
 34561345243 \ar[dr] & 34562453413 \ar[d] & 3413245432413 \ar@{-->}[dl] && 2L\langle-31\rangle\\
& 3413245643245432413 &&& L\langle-33\rangle} \qquad\qquad 
\xymatrix@C=.5em{ & 345 \ar[ld] \ar[d] \ar[rd] && L\langle-27\rangle \\
 3413245 \ar[d] \ar[rdd] & 3456245 \ar[ld] \ar[dd] & 34561345 \ar@{-->}[lld] \ar@{-->}[ldd] & 2L\langle-29\rangle\\
 3413245643245 \ar[rdd] &&& L\langle-31\rangle\\
&  34562453413245 \ar[d] && L\langle-33\rangle \\
& 3456245341324562453413245 && L\langle-35\rangle  }}\]
\vskip .5cm
\[ \scalebox{.7}{\xymatrix@R=4em@C=.5em{   & 34  \ar[dl]\ar[d]\ar[dr] && L\langle-26\rangle \\
345134 \ar@{-->}[ddr]\ar[dr]\ar[rrd]   &    34524 \ar[dl]\ar[d]\ar[rd] & 341324 \ar@{-->}[ld]\ar[d]\ar[ldd] & 2L\langle-28\rangle    \\
{345624534} \ar[d]\ar[rd]\ar[rrd]  & 3456134524 \ar@{-->}[ld]\ar[rd]    & 34132454324 \ar[lld]\ar[d] & 2L\langle-30\rangle \\
3413245643245134   \ar[rd] & {3456245341324} \ar[d] & 34132456432454324 \ar[ld] & 2L\langle-32\rangle  \\
& 34132456453413245341324 && L\langle-34\rangle } } \]
    \caption{The posets $\JI{2}{2}$,  $\JI{2}{3}$ and $\JI{2}{4}$ in the first row, $\JI{3}{3}$ and $\JI{3}{5}$  in the second row, and $\JI{3}{4}$ in the third row, in type $E_6$.}
    \label{fig:E6_23}
\end{figure}

\begin{figure}
    \centering
\[ \scalebox{.75}{\xymatrix@R=6em@C=0em{
&&    4 \ar[dl]\ar[d]\ar[rd] &&&& L\langle-25\rangle \\
&   4324 \ar[ld]\ar@{-->}[d]\ar[dr]\ar@{-->}[drr]\ar[drrrr]
& 4524 \ar@{-->}[lld]\ar[ld]\ar@{-->}[d]\ar[dr]\ar[drrr]
& 4534  \ar[llld]\ar[lld]\ar[ld]\ar[d]\ar[dr]\ar[drr] &&& 2L\langle-27\rangle\\
43245134 \ar[d]\ar[drr]\ar@{-->}[drrr]\ar[drrrr] & 
45624534 \ar[d]\ar[dr]\ar@{-->}[drr]\ar[drrrr]& 
45341324 \ar[ld]\ar@{-->}[d]\ar[dr]\ar[drr]& 
45634524 \ar[llld]\ar@{-->}[ld]\ar[d]\ar[drr]& 
456345134 \ar[lld]\ar[ld]& 
432454324 \ar[llld]\ar[ld]\ar[d]\ar[dll] & 3L\langle-29\rangle\\
432456134524 \ar[drr]\ar@{-->}[drrr] & 
456245341324 \ar@{-->}[dr]\ar[drr] & 
43245643245134 \ar[dl]\ar[d]\ar[dr] & 
45634132454324 \ar[dll]\ar[dl]\ar[d] & 
453413245341324 \ar[dl]\ar[dll]  & 
432456432454324 \ar[dll]\ar[dlll] & 3L\langle-31\rangle\\
& 4534132456543245134 \ar[rd] &  453413245634132454324 \ar[d] & 432456453413245341324 \ar[ld] &&& 2L\langle-33\rangle\\
&&  45624534132456453413245341324 &&&& L\langle-35\rangle
}} \]
    \caption{The poset $\JI{4}{4}$ in type $E_6$.}
    \label{fig:E6_44}
\end{figure}

\begin{proposition}\label{E44}
For each $x \in \JI{}{}\setminus \{24563451342\}
$, the socle of $\Delta_e/\Delta_x$ is simple, and hence equal to the maximal degree component of $\Delta_e/\Delta_x$.
For $x=24563451342$, we have 
\begin{equation}\label{eqE6remainingone}
\soc\Delta_e/\Delta_x \cong L_{w_{22}}\langle -31 \rangle \quad \text{ or } \quad \soc\Delta_e/\Delta_x \cong L_{w_{22}}\langle -31 \rangle \oplus L_{w_{22}}\langle -29 \rangle  .    
\end{equation}
\end{proposition}

\begin{proof}
Comparing the diagrams in Figure \ref{fig:E6_1j} with the corresponding Kazhdan-Lusztig polynomials in Table \ref{tab:KLE6}, one immediately gets the socles of $\Delta_e/\Delta_x$, for all $x$ in Figure \ref{fig:E6_1j}.

For $\JI{2}{2}$, we have the same argument for all elements except for $x = 24563451342$ which is hit by a non-socle-killing relation from $245342$. We thus have the two possibilities as stated.

In $\JI{2}{3}$, the only possibility for a non-simple socle is \[\Delta_e/\Delta_{2456345243} = L_{w_{23}}\langle -31 \rangle \oplus  L_{w_{23}}\langle -29 \rangle. \]
But in such a case, Lemma \ref{lem:bruh_soc} would imply that there is an arrow (dashed) between the two middle elements in $\JI{2}{3}$, a contradiction. 

The claim for $\JI{2}{4}$ follows from the Kazhdan-Lusztig polynomials. Note that, in this case,
the three non-comparable elements in the same degree have isomorphic socles, and that any two of the three socles generate the full isotypic component $2L_{w_{24}}\langle -30\rangle$ in $\Delta_e$.

In $\JI{3}{3}$, the element $345243$ has simple socle again because of Lemma \ref{lem:bruh_soc}. The biggest element there has simple socle because it is hit by two socle-killing arrows, and the socles corresponding to the sources of the two arrows generate the full isotypic component $2L_{w_{33}}\langle -31\rangle$ in $\Delta_e$. Similar arguments apply to $\JI{3}{5}$.

Let us consider now $\JI{3}{4}$. For $x=345624534$, the quotient $\Delta_e / \Delta_{x}$ contains the composition factor $L_{w_{34}}\langle -28\rangle$ with multiplicity one, as one can check by a Kazhdan-Lusztig computation. This composition factor must be the one coming from the socle of $x' =34524$. Since the arrow $x' \to x$ is socle-killing, this composition factor is not in the socle of $\Delta_e / \Delta_{x}$, and therefore $\Delta_e / \Delta_{x}$ has simple socle. An analogous argument applies to $3456245341324$.

Finally, consider now $\JI{4}{4}$. One can check that the socles are simple for all the elements in degrees $25$, $27$, and $29$ by the same arguments as above. In degree $31$, the first two elements from the left, as well as the first two elements from the right contain $L_{w_{44}}\langle-29\rangle$ with multiplicity $2$, and therefore it is enough to observe that each is hit by at least two solid arrows. 

Consider now $x=43245643245134$ in degree $31$.
There are four socle-killing relations $x'<x$ for $x'$ in degree $29$ and it is enough to show that the sum of $\co_{x_i}$, where we name these four elements $x_1,x_2,x_3,x_4$ from the left, is the isotypic component $3L\langle -29 \rangle$. Suppose not. Then the sum is a subquotient $X\cong 2L\langle -29 \rangle$ in $3L\langle -29 \rangle$. 
However, computation shows that the joins $y=x_1\vee x_2$ and $z = x_1\vee x_4$ exists and that $y\neq z$.
By Proposition~\ref{cor:sum}, the socles of $\co_y$ and $\co_z$ are contained in $X$.
At the same time, by $x_1<y$ and $x_1<z$, these socle strictly contain $\so_{x_1}$. 
It follows that $\so_y=X=\so_z$ which contradicts $y\neq z$.

Analogous arguments cover the rest of the elements in $\JI{4}{4}$.
\end{proof}

\begin{example}\label{e644bigex}
Let $b\in W$ be the join of the six elements in Figure~\ref{fig:E6_44} whose socle is contained in $3L\langle-29\rangle$ (computations confirm that this join exists). 
Computation also shows that $\JM(b)$ consists exactly of these six elements, while we also have $b=x\vee y$ where $x,y\in\JM(b)$ are the two length 9 elements among these six.
Since $\so_z\cong L\langle-29 \rangle$, for each $z\in\JM(b)$, by Proposition~\ref{E44}, we have
\[[\Delta_x\langle -\ell(x) \rangle \cap \Delta_y\langle -\ell(y) \rangle : L\langle-29\rangle] = 1,\]
while a Kazhdan-Lusztig computation shows 
\[[\Delta_b\langle -\ell(b) \rangle:L\langle-29\rangle] = 0.\]
It follows that $\Delta_x \langle -\ell(x) \rangle\cap \Delta_y \langle -\ell(y) \rangle \neq \Delta_b\langle -\ell(b) \rangle$, which provides an example where the answer to Question~\ref{con2} is negative.
%Since we have a chain $z_1 < z_1\vee z_2 < \vee\{z_1,z_2,z_3\} = b$ with some $z_i\in \JM(b)$ we have (see Proposition~\ref{chaincond} below)
%\[[\bigcup_{z\in\JM(b)} \Delta_z : L\langle-29\rangle] = 0\]
\end{example}

\subsection{Type $E$}\label{ss:JIE}

Let $(W,S)$ be of type $E$.
The polynomials $p_{e,w}$, for $w \in \mathcal{J}$, in types $E_7$ and $E_8$ are listed in Appendix
(see Subsection~\ref{ss:E6} for $E_6$).

Computer computations confirm the following fact:

\begin{proposition}\label{Ebgchain}
For any $i,j\in S$, there exists a chain of length $p_{e,w_{ij}}(1)$ in $\BG{i}{j}$.
\end{proposition}

As a consequence of Proposition~\ref{Ebgchain}, we have:

\begin{corollary}\label{halfsosumE}
Let $w\in W$. The socle of $\co_w$ is contained in the sum of $\so_{x}$ taken over $x\in \JM(w)$.
\end{corollary}

\begin{proof}
Proposition~\ref{Ebgchain} and Proposition~\ref{cor:sum} proves the claim.
\end{proof}

%We compute the posets $\JI{i}{j}$ and $\BG{i}{j}$ for each $i,j\in S$. We also compute the joins of the elements in $\JI{i}{j}$ (for a fixed pair $i,j$). 

%Let $x$ be the unique element in $\hc(a,b)$ and write
%\begin{equation}\label{eq:newpex2}
%    p_{e,x}=c_0v^{d_0}+ c_1v^{d_1}+\cdots c_rv^{d_r}
%\end{equation} with $c_i=c_i(a,b)\neq 0$ and $d_i-c_i(a,b)\in \mathbb N$.
%Thus, $r+1=r(a,b)+1$ is the number of homogenoeus terms in $p_{e,x}$.
%\subsubsection{Multiplicity-free cases}
%Let $a,b\in S$ be such that $c_i=1$ for all $i$.
%Computer computation verifies that, in this case, $p_{e,w_{a,b}}(1) = |JI(a,b)|$.
%By pigeon hole principle, the socles of $\co_w$ for $w\in JI(a,b)$ are simple.
%\subsubsection{Cases with multiplicity}
%Now consider the case where $p_{e,x}$ has multiplicity, i.e., some $c_i$ is greater than $1$.

\begin{conjecture}\label{combE}
For each $x\in \JI{}{}$, the socle of $\Delta_e/\Delta_x$ is simple. 
%Moreover, the image of the assignment % $qs:JI(s,t)\to CF$ given by  
%$x\mapsto \soc\Delta_e/\Delta_x$ contains exactly $\frac{(c+1)c}{2}$ distinct summands in each isotypic component $cL_{w_{i,j}}\langle -d\rangle$.
\end{conjecture}

\begin{remark}

Conjecture~\ref{combE} is verified in a majority of cases in types $E_7$ and $E_8$ by the bounds from Subsection~ \ref{s:Strategies} and various computer computations. 
Completely proving Conjecture~\ref{combE}, however, seems to require new ideas, as already seen in \eqref{eqE6remainingone}.
% In fact, there seems no reason in \eqref{eqE6remainingone} to expect that $\so_x$ is simple.
% We do still make Conjecture~\ref{combE} in order to have a complete conjectural description of the socles in all types.
\end{remark}

If Conjecture~\ref{combE} is true, all $\so_z$, for $z\in\JI{}{}$, are determined by Kazhdan-Lusztig computations and follow a pattern similar to the one described in Subsection~\ref{ss:E6} for type $E_6$. 
For example, consider the case of a `large' $\JI{i}{j}$, for example, $\JI{4}{4}$ in type $E_8$.
Recall our notation $(c_r)$ for the coefficients of KL polynomials, see \eqref{eq:newpex}, 
(e.g., we have $ (1,2,3,4,5,6,6,6,6,6,5,4,3,2,1)$ in type $E_8$ at $\JI{4}{4}$). 
Then, for each degree $r$, the isotypic component $c_r L\langle d_r\rangle$ contains $\frac{(c_r+1)c_r}{2}$ distinct simple subquotients which are the socles of $\frac{(c_r+1)c_r}{2}$ elements in $\JI{i}{j}$. 
It would be nice to explain why the number $\frac{(c+1)c}{2}$ appears.

\begin{comment}

\begin{obs}\label{join pattern} 
Let $cL\langle -d \rangle$ be an isotypic component in $\Delta_e$, (i.e., $c=c_k(i,j)$ and $d=d_k(i,j)$ for some $0\leq k\leq r(i,j)$ in \eqref{eq:newpex}).
Then $|\JI{i}{j}^d| = \frac{(c+1)c}{2}$, where  $\JI{i}{j}^d = \{y\in \JI{i}{j}\ |\ m(y)=d\}$. @@when JId is defined@@
Moreover, for each computable (write when?)\hk{We verified E7 and maybe some of E8 or all of? check and write}
  $y\in \JI{i}{j}$, we can find $y_k\in \JI{i}{j}^d$ such that the chain  $y=y_1<y_1\vee y_2 <\cdots< y_1\vee\cdots\vee y_{c}=\bigvee \JI{i}{j}^d$ is strictly increasing.
\end{obs}
%The  and Lemma~\ref{join pattern} suggests the following. The proof is not complete.

\begin{proof}
%By Lemma \ref{Ehomogsocle}, the socle of $\co_x$ is homogeneous. 
Let $c L\langle -d \rangle$ be the isotypic component of $\Delta_e$ that contains $\soc\co_x$. %and let $JI_d(a,b) = \{y\in JI(a,b)\ |\ m(y) = d\}$.
Then by Proposition \ref{join pattern} there exists a chain $x=x_1<x_1\vee x_2<\cdots x_1\vee\cdots\vee x_c$ for $x_k\in \JI{i}{j}^d$ in $W$. By Lemma \ref{lm:cosets} the elements $y_k = x_1\vee \ldots\vee x_k$ in fact belongs to $\BG{i}{j}$. 
Then by Proposition \ref{prop6}, the socles of $\co_{y_i}$ are contained in the isotypic component $cL\langle -d\rangle$, and  $\{\soc\co_{y_i}\}_i$ forms a strictly increasing chain of submodules in $cL\langle -d \rangle$. 
It follows that $\co_{y_1}=\co_x$ is simple.
\end{proof}
\end{comment}

\begin{remark}
It is interesting to point out that, in the simply laced types, the maximal number appearing as a 
coefficient of the KL polynomial between $e$ and a penultimate element coincides with the
maximal coefficient of a root. For non-simply laced types this fails.
\end{remark}

\subsection{Type $F_4$}\label{ssKLF}
Assume that $(W,S)$ is of type $F_4$.
We denote the simple reflections by
\begin{tikzpicture}[scale=0.4,baseline=-3]
\protect\draw (6 cm,0) -- (4 cm,0);
\protect\draw (4 cm,0.1cm) -- (2 cm,0.1 cm);
\protect\draw (4 cm,-0.1cm) -- (2 cm,-0.1 cm);
\protect\draw (2 cm,0) -- (0 cm,0);

\protect\draw[fill=white] (6 cm, 0 cm) circle (.15cm) node[above=1pt]{\scriptsize $4$};
\protect\draw[fill=white] (4 cm, 0 cm) circle (.15cm) node[above=1pt]{\scriptsize $3$};
\protect\draw[fill=white] (2 cm, 0 cm) circle (.15cm) node[above=1pt]{\scriptsize $2$};
\protect\draw[fill=white] (0 cm, 0 cm) circle (.15cm) node[above=1pt]{\scriptsize $1$};
\end{tikzpicture}.
%Note that in this type not all $\mathtt{H}$-cells in $\mathcal{J}$ are singletons. 
We let ${}^i\mathcal{H}^j = \{u_{ij},w_{ij}\}$, where either $u_{ij} = w_{ij}$ or $u_{ij}$ is strictly shorter than $w_{ij}$.

The Kazhdan-Lusztig polynomials $p_{e,w}$ for $w \in \mathcal{J}$ are collected in Table \ref{tab:KLF4}. We put $p_{e,w}$ in the position $(i,j)$ of the table, where $i$ is the left, and $j$ the right ascent for $w$. The computations was performed in SageMath v.9.0.

\begin{table}[ht]
    \centering
\renewcommand{\arraystretch}{1.2}
\scalebox{0.92}{
\begin{tabular}{|l||l|l|l|l|}
\hline
 & $1$  &  $2$  &  $3$  &  $4$ 
 \\
\hline\hline
$1$ & \begin{tabular}{@{}l@{}}$v^{19} + v^{13}$ \\ $v^{23} + v^{17}$\end{tabular} & \begin{tabular}{@{}l@{}}$v^{20} + v^{18} + v^{14}$ \\ $v^{22} + v^{18} + v^{16}$\end{tabular} & $v^{21} + v^{19} + v^{17} + v^{15}$ & $v^{20} + v^{16}$ \\ \hline
$2$ & \begin{tabular}{@{}l@{}}$v^{20} + v^{18} + v^{14}$ \\ $v^{22} + v^{18} + v^{16}$\end{tabular} & \begin{tabular}{@{}l@{}} $v^{21} + 2v^{19} + v^{17} + v^{15} + v^{13}$ \\ $v^{23} + v^{21} + v^{19} + 2v^{17} + v^{15}$\end{tabular} & $v^{22} + 2v^{20} + 2v^{18} + 2v^{16} + v^{14}$ & $v^{21} + v^{19} + v^{17} + v^{15}$ \\ \hline
$3$ & $v^{21} + v^{19} + v^{17} + v^{15}$ & $v^{22} + 2v^{20} + 2v^{18} + 2v^{16} + v^{14}$ & \begin{tabular}{@{}l@{}} $v^{21} + 2v^{19} + v^{17} + v^{15} + v^{13}$ \\ $v^{23} + v^{21} + v^{19} + 2v^{17} + v^{15}$\end{tabular} & \begin{tabular}{@{}l@{}}$v^{20} + v^{18} + v^{14}$ \\ $v^{22} + v^{18} + v^{16}$\end{tabular} \\ \hline
$4$ & $v^{20} + v^{16}$ & $v^{21} + v^{19} + v^{17} + v^{15}$ & \begin{tabular}{@{}l@{}}$v^{20} + v^{18} + v^{14}$ \\ $v^{22} + v^{18} + v^{16}$\end{tabular} & \begin{tabular}{@{}l@{}}$v^{19} + v^{13}$ \\ $v^{23} + v^{17}$\end{tabular} \\ \hline
\end{tabular}
} \vskip 10pt
    \caption{Kazhdan-Lusztig polynomials $p_{e,w}$, for $w \in \mathcal{J}$, in $F_4$.}
    \label{tab:KLF4}
\end{table}

%We can completely determine the Bruhat structure of the subposet $\JI{i}{j}$ for each $i,j\in S$.
%These determine the socles of $\co_y$ in most cases.
%but for $\JI{1}{1}$, $\JI{2}{2}$ and $\JI{2}{3}$ we only reduce the answer to 2 or 4 possibilities.  
%\tothink{Fix the descents (a,b). Suppose $\Delta_e$ has penultimate $L(w1) \oplus L(w2)$  in degree $d$ (i.e., both KL poly at $\hc(a,b)={w1,w2}$, has $v^d$ with coefficient one). And suppose we proved that $x<y$ in $JI(a,b)$ has \[\so_x = Lw1<-d>   \] \[\so_y = Lw1<-d>\oplus Lw2<-d> .\] How can we tell which element in $\hc(a,b)$ is w1? - By calculating KL poly's.}
%
Figure \ref{fig:F_JI} lists all $\JI{i}{j}$, up to symmetries, where the dashed arrows denote non-socle killing relations. %\todo{there might be a better presetantion of the graphs}.
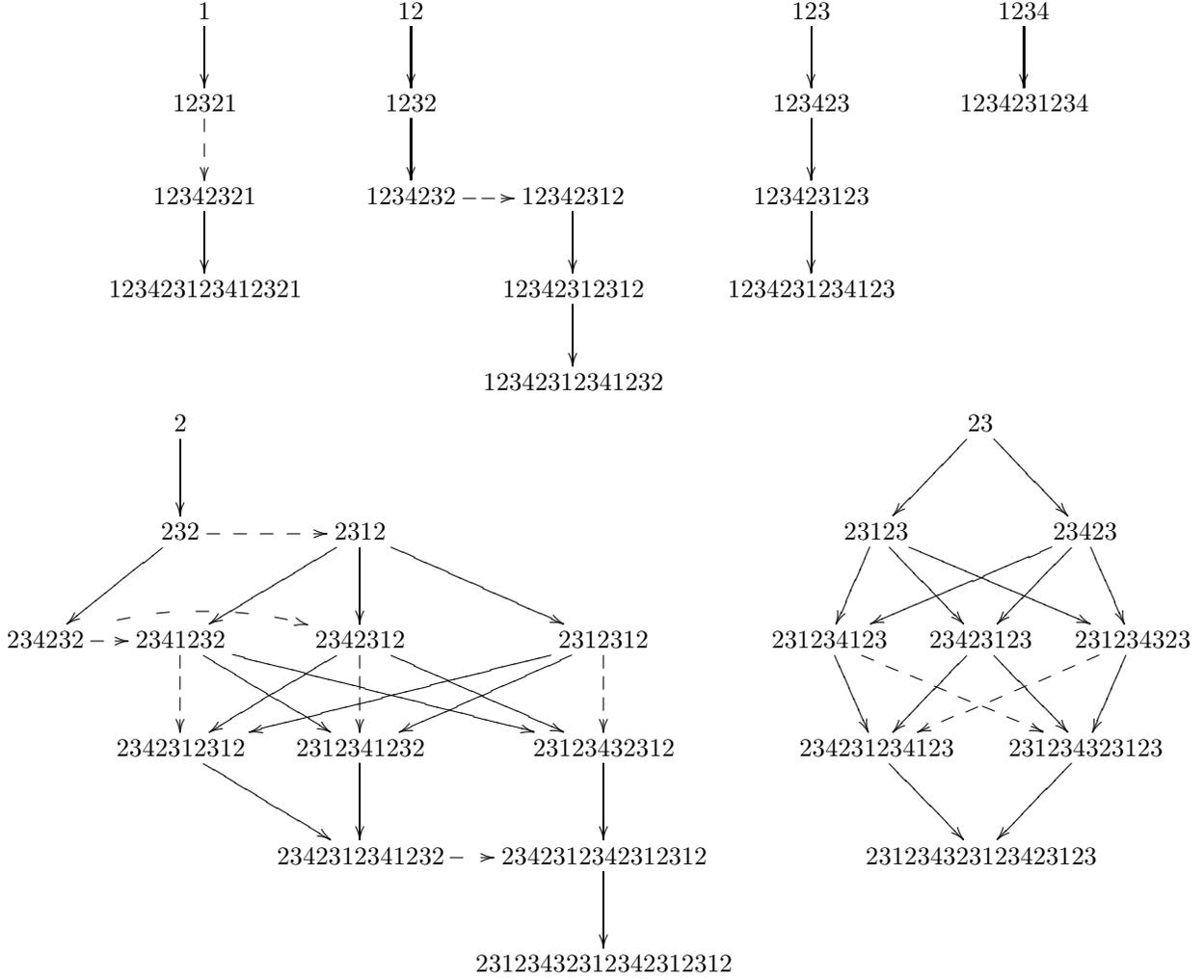
\begin{figure}
    \centering
\begin{gather*}
\xymatrix{ 1 \ar[d] \\
12321 \ar@{-->}[d] \\  12342321 \ar[d] \\
123423123412321} \qquad \xymatrix@C=0.5em{ 12 \ar[d] \\
1232 \ar[d] \\
1234232 \ar@{-->}[r] & 12342312 \ar[d] \\
& 12342312312 \ar[d] \\
& 12342312341232  } \qquad \xymatrix{  123 \ar[d] \\
123423 \ar[d] \\
123423123 \ar[d] \\
1234231234123 } \qquad  \xymatrix{ 1234 \ar[d] \\
1234231234 }     \\
\xymatrix@C=0.7em@R=3em{ &  2 \ar[d] \\
    &  232 \ar@{-->}[r] \ar[ld] &  2312 \ar[dl] \ar[d] \ar[dr] \\      
    234232 \ar@{-->}[r] \ar@{-->}@/^1pc/[rr] & 2341232 \ar@{-->}[d] \ar[dr] \ar[rrd] & 2342312 \ar[ld] \ar@{-->}[d] \ar[rd] &  2312312 \ar@{-->}[d] \ar[lld] \ar[ld]  \\
    & 2342312312 \ar[rd]    &  2312341232 \ar[d]  & 23123432312 \ar[d] \\
    && 2342312341232 \ar@{-->}[r]   & 2342312342312312 \ar[d]  \\
    &&& 23123432312342312312}  \quad  
\xymatrix@C=-4em@R=3em{  && 23 \ar[dr] \ar[dl] && \\
& 23123 \ar[dl] \ar[dr] \ar[drrr]  && 23423 \ar[dl] \ar[dr] \ar[dlll] & \\
 231234123 \ar[rd] \ar@{-->}[rrrd] &&  23423123 \ar[rd] \ar[ld] && 231234323 \ar[ld] \ar@{-->}[llld] \\
& 234231234123 \ar[rd] &&  231234323123 \ar[ld] & \\
&& 231234323123423123    }   \end{gather*}
    \caption{The Bruhat graph of $\JI{i}{j}$ in type $F_{4}$, for $(i,j)=(1,1)$, $(1,2)$, $(1,3)$, $(1,4)$, $(2,2)$ and $(2,3)$, with the non socle-killing arrows dashed.}
    \label{fig:F_JI}
\end{figure}

A Kazhdan-Lusztig computation determines the maximal degree components of $\Delta_e / \Delta_x$. %They can be read, up to symmetry, from Table \ref{tab:F4_lowest_deg}. 
Using this information, we calculate the socles of $\co_x$.

\begin{proposition}
\label{prop:F4_soc}
Table \ref{tab:F4_lowest_deg} gives the socles of $\Delta_e/\Delta_x$. In Table \ref{tab:F4_lowest_deg} ,we denote a simple subquotient by its parameter $y$, instead of $L_y$. If a socle component is not a (graded) isotypic component in $\Delta_e$, but a simple submodule in an isotypic component of multiplicity 2, we write $y$ in {\color{gray}gray}.)%, except for $x=12342321$, in which case we have
%\begin{equation}
%\label{eq:F4_ex}    \soc \Delta_e/\Delta_{12342321} = L_{u_{11}} \langle -19 \rangle \oplus L_{w_{11}} \langle -17 \rangle.
%\end{equation}
\end{proposition}

\begin{table}[ht]
    \centering
    \begin{tabular}{ccc}
    \begin{tabular}{|c|c|}
\hline
 $x \in \JI{}{}$ &  $\soc\Delta_e / \Delta_x$ \\
 \hline\hline
 $1$ & $u_{11}\langle -13 \rangle$ \\
 $12321$ & $w_{11}\langle -17 \rangle$ \\
 $12342321$ & $u_{11}\langle -19 \rangle  \oplus w_{11}\langle -17 \rangle $ \\
 $123423123412321$ & $w_{11}\langle -23 \rangle$ \\ \hline
 $12$ & $u_{12}\langle -14 \rangle$ \\
 $1232$ & $w_{12}\langle -16 \rangle$ \\
 $1234232$ & $u_{12}\langle -18 \rangle$ \\
 $12342312$ & $(u_{12} \oplus  w_{12})\langle -18 \rangle$ \\
 $12342312312$ & $u_{12}\langle -20 \rangle$ \\
 $12342312341232$ & $w_{12}\langle -22 \rangle$ \\ \hline
 $123$ & $u_{13}\langle -15 \rangle$ \\
 $123423$ & $u_{13}\langle -17 \rangle$ \\
 $123423123$ & $u_{13}\langle -19 \rangle$ \\
 $1234231234123$ & $u_{13}\langle -21 \rangle$ \\ \hline
  $1234$ & $u_{14}\langle -16 \rangle$ \\
 $1234231234$ & $u_{14}\langle -20 \rangle$ \\ \hline
\end{tabular}      &     & \begin{tabular}{|c|c|}
\hline
 $x \in \JI{}{}$ & $\soc\Delta_e / \Delta_x$ \\
 \hline\hline
  $2$ & $u_{22}\langle -13 \rangle$ \\
 $232$ & $w_{22}\langle -15 \rangle$ \\
 $2312$ & $(u_{22} \oplus  w_{22})\langle -15 \rangle$ \\
 $234232$ & $u_{22}\langle -17 \rangle$ \\
 $2312312$ & ${\color{gray}w_{22}\langle -17 \rangle}$ \\
 $2342312$ & $(u_{22} \oplus  {\color{gray}w_{22}})\langle -17 \rangle$ \\
 $2341232$ & $(u_{22} \oplus  {\color{gray}w_{22}})\langle -17 \rangle$ \\
 $2312341232$ & ${\color{gray}u_{22}\langle -19 \rangle}$ \\
 $2342312312$ & ${\color{gray}u_{22}\langle -19 \rangle}$ \\
 $23123432312$ & $({\color{gray}u_{22}} \oplus  w_{22})\langle -19 \rangle$ \\
 $2342312341232$ & $w_{22}\langle -21 \rangle$ \\
 $2342312342312312$ & $(u_{22} \oplus  w_{22})\langle -21 \rangle$ \\
 $23123432312342312312$ & $w_{22}\langle -23 \rangle$ \\ \hline
  $23$ & $u_{23}\langle -14 \rangle$ \\
 $23423$ & ${\color{gray}u_{23}\langle -16 \rangle}$ \\
 $23123$ & ${\color{gray}u_{23}\langle -16 \rangle}$ \\
 $23423123$ & ${\color{gray}u_{23}\langle -18 \rangle}$ \\
 $231234123$ & ${\color{gray}u_{23}\langle -18 \rangle}$ \\
 $231234323$ & ${\color{gray}u_{23}\langle -18 \rangle}$ \\
 $234231234123$ & ${\color{gray}u_{23}\langle -20 \rangle}$ \\
 $231234323123$ & ${\color{gray}u_{23}\langle -20 \rangle}$ \\
 $231234323123423123$ & $u_{23}\langle -22 \rangle$ \\ \hline
\end{tabular}
    \end{tabular}
    \bigskip
    \caption{The socle of $\Delta_e / \Delta_x$ for $x$ from Figure \ref{fig:F_JI}.}
    \label{tab:F4_lowest_deg}
\end{table}

\begin{proof}
Note that $\soc \Delta_e/\Delta_x$ always contains the maximal degree component of $\Delta_e/\Delta_x$, but, other than that, all other summands, if any, must be of strictly higher degree. Using the diagrams in Figure \ref{fig:F_JI}, %the data in Table \ref{tab:F4_lowest_deg}
the Kazhdan-Lusztig computation of the maximal degree components, and the property of socle-killing relations (Lemma \ref{socledie}), we can verify Table \ref{tab:F4_lowest_deg} except maybe in the following three cases, each of which has two possibilities:
\begin{gather}
\label{gath:F4_ex1}    \soc \Delta_e/\Delta_{12342321} = L_{u_{11}} \langle -19 \rangle \ \text{ or } \  L_{u_{11}} \langle -19 \rangle \oplus L_{w_{11}} \langle -17 \rangle, \\
\label{gath:F4_ex2}     \soc \Delta_e/\Delta_{234232} = L_{u_{22}} \langle -17 \rangle \ \text{ or } \  L_{u_{22}} \langle -17 \rangle \oplus L_{u_{22}} \langle -15 \rangle, \\
\label{gath:F4_ex3}     \soc \Delta_e/\Delta_{2342312341232} = L_{u_{22}} \langle -21 \rangle \ \text{ or } \  L_{u_{22}} \langle -21 \rangle \oplus L_{w_{22}} \langle -19 \rangle .
\end{gather}

For (\ref{gath:F4_ex2}) and (\ref{gath:F4_ex3}), one can check by calculating the Kazhdan-Lusztig polynomials that the hypothetical summand above the maximal degree does not appear in the corresponding $\Delta_e/\Delta_x$, and therefore cannot appear in the socle.

Now we consider (\ref{gath:F4_ex1}). For $x = 12342321$, we claim that $L_{w_{11}} \langle -17 \rangle$ does appear in $\Delta_e/\Delta_{x}$. Assume that it does not appear in the socle. Then $L_{w_{11}} \langle -17 \rangle$ must extend to $L_{u_{11}} \langle -19 \rangle$ inside $\Delta_e$, which implies that the projective cover $P_{w_{11}}$ of $L_{w_{11}}$ must contain $L_{u_{11}}\langle-21\rangle$ as a composition factor. However, this is not the case, as one can check by Kazhdan-Lusztig computations. The claim follows.
\end{proof}

\begin{corollary}\label{halfsosumF}
Let $w\in W$. The socle of $\co_w$ is contained in the sum of $\so_{x}$ taken over $x\in \JM(w)$.
\end{corollary}

\begin{proof}
The socles of $\co_z$ for $z\in\JI{}{}$ determined in Proposition~\ref{prop:F4_soc} satisfy the assumption of Proposition~\ref{ccoor:sum}, thus the claim follows from Proposition~\ref{ccoor:sum}.
\end{proof}

\begin{example}\label{F4example}
Let
%x=c y=b z= a, w=vabc
\[x=32341232, y=234231234, z=2312312, w= 23423123432.\]
We have $x,y,z\in \JI{}{}$ and $w\in W$ are such that 
$\JM(w)=\{x,y,z\}$, $\JM''(w)=\{x,y\}$ and $w=x\vee y \vee z = x\vee y$.
We claim 
\begin{equation}\label{F4ineq}
\so_w = \so_x\oplus \so_y \subsetneq \so_x +\so_y +\so_z.
\end{equation}
In particular, the socle-sum property does not hold for $w$.

To prove the claim, we need to consider the elements $x',y'\in\JI{2}{2}$
given by
\[x' = 2341232, \quad y'=2342312\]
which satisfy $x'<x$ and $y'<y$, necessarily socle-killing, and $x'\vee y' > z$.
The sum of socles of $\co_{x'}$ and $\co_{y'}$ thus contains $\so_z$ while it is not contained in $\so_w$.
On the other hand, by Corollary~\ref{halfsosumF}, the socle of $\co_w$ is contained in
$\so_x\oplus\so_y\oplus\so_z$ (the sum is direct since $x,y,z\in\JI{}{}$ with distinct descents).
We conclude that $\so_w \subseteq \so_x\oplus\so_y$, but a strict inclusion would imply $w= x$ or $w = y$ which
is not true. This proves Formula~\eqref{F4ineq}.
\end{example}

\subsection{Type $G$}\label{ssKLG}

Assume that $(W,S)$ is of type $G_2$. We denote the simple reflections by
\begin{tikzpicture}[scale=0.4,baseline=-3]
\protect\draw (4 cm,0.2cm) -- (2 cm,0.2 cm);
\protect\draw (4 cm,0.0cm) -- (2 cm,0.0 cm);
\protect\draw (4 cm,-0.2cm) -- (2 cm,-0.2 cm);

\protect\draw[fill=white] (4 cm, 0 cm) circle (.15cm) node[above=1pt]{\scriptsize $2$};
\protect\draw[fill=white] (2 cm, 0 cm) circle (.15cm) node[above=1pt]{\scriptsize $1$};
\end{tikzpicture}.

The KL polynomials are trivial (see Table \ref{tab:KLG2}) and the submodule structure of the dominant Verma module is well-known. For $x \neq e,w_0$, we have $\soc\Delta_e/\Delta_x=L_{\sigma(x)}\langle -\ell(x)\rangle$ where $\sigma$ is the diagram automorphism swapping the two elements in $S$, and the socle-sum property and Conjecture~\ref{con2JM} hold, for all $w\in W$. 

\begin{table}[ht]
    \centering
\renewcommand{\arraystretch}{1.2}
\scalebox{0.92}{
\begin{tabular}{|l||p{.6cm}|p{.6cm}|}
\hline
 & $1$  &  $2$  \\
\hline\hline
$1$ & \begin{tabular}{@{}l@{}}$v$ \\ $v^{3}$ \\ $v^{5}$ \end{tabular} & \begin{tabular}{@{}l@{}}$v^{2}$ \\ $v^{4}$ \end{tabular}  \\ \hline
$2$ & \begin{tabular}{@{}l@{}}$v^{2}$ \\ $v^{4}$ \end{tabular} & \begin{tabular}{@{}l@{}}$v$ \\ $v^{3}$ \\ $v^{5}$ \end{tabular} \\ \hline
\end{tabular}
} \vskip 10pt
    \caption{Kazhdan-Lusztig polynomials $p_{e,w}$ for $w \in \mathcal{J}$ in $G_2$.}
    \label{tab:KLG2}
\end{table}

\subsection*{Acknowledgments}
The second author was partially supported by
the Swedish Research Council, 
G{\"o}ran Gustafsson Stiftelse and Vergstiftelsen. The third author was supported by G{\"o}ran Gustafsson Stiftelse, Vergstiftelsen, and the QuantiXLie Center of Excellence grant no. KK.01.1.1.01.0004 funded by the European Regional Development Fund.
%This research was partially supported by
%the Swedish Research Council, 
%G{\"o}ran Gustafsson Stiftelse and Vergstiftelsen.
%The third author was also partially supported by the QuantiXLie Center of Excellence grant no. KK.01.1.1.01.0004 funded by the European Regional Development Fund.
For computer assisted calculations we used SageMath v9.0.
We thank Meinolf Geck for sharing the computer code for calculating join-irreducible elements in Bruhat orders. The code runs in the GAP-part of CHEVIE platform, \cite{chevie}.

\vspace{2mm}

\noindent
H.~K.: Department of Mathematics, Uppsala University, Box. 480,
SE-75106, Uppsala,\\ SWEDEN, email: {\tt hankyung.ko\symbol{64}math.uu.se}

\noindent
V.~M.: Department of Mathematics, Uppsala University, Box. 480,
SE-75106, Uppsala,\\ SWEDEN, email: {\tt mazor\symbol{64}math.uu.se}

\noindent
R.~M.: Department of Mathematics, Uppsala University, Box. 480,
SE-75106, Uppsala,\\ SWEDEN, email: {\tt rafaelmrdjen\symbol{64}gmail.com}

%\end{document}

\vspace{2cm}

\section{Appendix: Kazhdan-Lusztig polynomials for $E_7$ and $E_8$}

In the tables below we list the Kazhdan-Lusztig polynomials $p_{e,w}$, for $w \in \mathcal{J}$, for the exceptional types $E_7$ and $E_8$. We put $p_{e,w}$ in the position $(i,j)$ of the table, where $i$ is the left, and $j$ the right ascent for $w$. The computations for $E_7$ were performed in SageMath v.9.0. For $E_8$, some additional tricks were necessary.

\subsection{Type $E_7$}

We denote the simple reflections by
\begin{tikzpicture}[scale=0.4,baseline=-3]
\protect\draw (10 cm,0) -- (8 cm,0);
\protect\draw (8 cm,0) -- (6 cm,0);
\protect\draw (6 cm,0) -- (4 cm,0);
\protect\draw (4 cm,0) -- (2 cm,0);
\protect\draw (2 cm,0) -- (0 cm,0);
\protect\draw (4 cm,0) -- (4 cm,1.5 cm);

\protect\draw[fill=white] (10 cm, 0 cm) circle (.15cm) node[below=1pt]{\scriptsize $7$};
\protect\draw[fill=white] (8 cm, 0 cm) circle (.15cm) node[below=1pt]{\scriptsize $6$};
\protect\draw[fill=white] (6 cm, 0 cm) circle (.15cm) node[below=1pt]{\scriptsize $5$};
\protect\draw[fill=white] (4 cm, 0 cm) circle (.15cm) node[below=1pt]{\scriptsize $4$};
\protect\draw[fill=white] (2 cm, 0 cm) circle (.15cm) node[below=1pt]{\scriptsize $3$};
\protect\draw[fill=white] (4 cm, 1.5 cm) circle (.15cm) node[right=1pt]{\scriptsize $2$};
\protect\draw[fill=white] (0 cm, 0 cm) circle (.15cm) node[below=1pt]{\scriptsize $1$};
\end{tikzpicture}. See Table \ref{tab:KLE7} for the results.

\begin{table}[ht]
    \centering
\renewcommand{\arraystretch}{1.4}
\scalebox{.8}{
\begin{tabular}{c}
\begin{tabular}{|l||p{4.5cm}|p{4.5cm}|p{4.5cm}|p{4.5cm}|} \hline
 & $1$ & $2$ & $3$ & $4$ \\ \hline\hline
$1$ & $v^{62} + v^{56} + v^{52} + v^{46}$ & $v^{59} + v^{55} + v^{53} + v^{49}$ & $v^{61} + v^{57} + v^{55} + v^{53} + v^{51} + v^{47}$ & $v^{60} + v^{58} + v^{56} + 2v^{54} + v^{52} + v^{50} + v^{48}$  \\ \hline
$2$ & $v^{59} + v^{55} + v^{53} + v^{49}$ & $v^{62} + v^{58} + v^{56} + v^{54} + v^{52} + v^{50} + v^{46}$ & $v^{60} + v^{58} + v^{56} + 2v^{54} + v^{52} + v^{50} + v^{48}$ & $v^{61} + v^{59} + 2v^{57} + 2v^{55} + 2v^{53} + 2v^{51} + v^{49} + v^{47}$ \\ \hline
$3$ & $v^{61} + v^{57} + v^{55} + v^{53} + v^{51} + v^{47}$ & $v^{60} + v^{58} + v^{56} + 2v^{54} + v^{52} + v^{50} + v^{48}$ & $v^{62} + v^{60} + v^{58} + 2v^{56} + 2v^{54} + 2v^{52} + v^{50} + v^{48} + v^{46}$ & $v^{61} + 2v^{59} + 2v^{57} + 3v^{55} + 3v^{53} + 2v^{51} + 2v^{49} + v^{47}$ \\ \hline
$4$ & $v^{60} + v^{58} + v^{56} + 2v^{54} + v^{52} + v^{50} + v^{48}$ & $v^{61} + v^{59} + 2v^{57} + 2v^{55} + 2v^{53} + 2v^{51} + v^{49} + v^{47}$ & $v^{61} + 2v^{59} + 2v^{57} + 3v^{55} + 3v^{53} + 2v^{51} + 2v^{49} + v^{47}$ & $v^{62} + 2v^{60} + 3v^{58} + 4v^{56} + 4v^{54} + 4v^{52} + 3v^{50} + 2v^{48} + v^{46}$ \\ \hline
$5$ & $v^{59} + v^{57} + v^{55} + v^{53} + v^{51} + v^{49}$ & $v^{60} + v^{58} + 2v^{56} + v^{54} + 2v^{52} + v^{50} + v^{48}$ & $v^{60} + 2v^{58} + 2v^{56} + 2v^{54} + 2v^{52} + 2v^{50} + v^{48}$ & $v^{61} + 2v^{59} + 3v^{57} + 3v^{55} + 3v^{53} + 3v^{51} + 2v^{49} + v^{47}$ \\ \hline
$6$ & $v^{58} + v^{56} + v^{52} + v^{50}$ & $v^{59} + v^{57} + v^{55} + v^{53} + v^{51} + v^{49}$ & $v^{59} + 2v^{57} + v^{55} + v^{53} + 2v^{51} + v^{49}$ & $v^{60} + 2v^{58} + 2v^{56} + 2v^{54} + 2v^{52} + 2v^{50} + v^{48}$  \\ \hline
$7$ & $v^{57} + v^{51}$ & $v^{58} + v^{54} + v^{50}$ & $v^{58} + v^{56} + v^{52} + v^{50}$ & $v^{59} + v^{57} + v^{55} + v^{53} + v^{51} + v^{49}$ \\ \hline
\end{tabular}
\\ \\
\begin{tabular}{|l||p{4.5cm}|p{4.5cm}|p{4.5cm}|} \hline
 & $5$ & $6$ & $7$ \\ \hline\hline
$1$ & $v^{59} + v^{57} + v^{55} + v^{53} + v^{51} + v^{49}$ & $v^{58} + v^{56} + v^{52} + v^{50}$ & $v^{57} + v^{51}$ \\ \hline
$2$ & $v^{60} + v^{58} + 2v^{56} + v^{54} + 2v^{52} + v^{50} + v^{48}$ & $v^{59} + v^{57} + v^{55} + v^{53} + v^{51} + v^{49}$ & $v^{58} + v^{54} + v^{50}$ \\ \hline
$3$ & $v^{60} + 2v^{58} + 2v^{56} + 2v^{54} + 2v^{52} + 2v^{50} + v^{48}$ & $v^{59} + 2v^{57} + v^{55} + v^{53} + 2v^{51} + v^{49}$ & $v^{58} + v^{56} + v^{52} + v^{50}$ \\ \hline
$4$ & $v^{61} + 2v^{59} + 3v^{57} + 3v^{55} + 3v^{53} + 3v^{51} + 2v^{49} + v^{47}$ & $v^{60} + 2v^{58} + 2v^{56} + 2v^{54} + 2v^{52} + 2v^{50} + v^{48}$ & $v^{59} + v^{57} + v^{55} + v^{53} + v^{51} + v^{49}$ \\ \hline
$5$ & $v^{62} + v^{60} + 2v^{58} + 2v^{56} + 3v^{54} + 2v^{52} + 2v^{50} + v^{48} + v^{46}$ & $v^{61} + v^{59} + v^{57} + 2v^{55} + 2v^{53} + v^{51} + v^{49} + v^{47}$ & $v^{60} + v^{56} + v^{54} + v^{52} + v^{48}$ \\ \hline
$6$ & $v^{61} + v^{59} + v^{57} + 2v^{55} + 2v^{53} + v^{51} + v^{49} + v^{47}$ & $v^{62} + v^{60} + v^{56} + 2v^{54} + v^{52} + v^{48} + v^{46}$ & $v^{61} + v^{55} + v^{53} + v^{47}$ \\ \hline
$7$ & $v^{60} + v^{56} + v^{54} + v^{52} + v^{48}$ & $v^{61} + v^{55} + v^{53} + v^{47}$ & $v^{62} + v^{54} + v^{46}$ \\ \hline
\end{tabular}
\end{tabular}
}
\vskip 10pt
    \caption{Kazhdan-Lusztig polynomials $p_{e,w}$ for $w \in \mathcal{J}$ in $E_7$.}
    \label{tab:KLE7}
\end{table}

\subsection{Type $E_8$}
We did not manage to get the polynomials directly from the computer, due to the complexity of the Weyl group, and the length of the elements from $\jc$, so we use a different approach here. Denote the simple reflections by
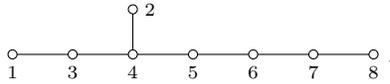
\[ \begin{tikzpicture}[scale=0.4,baseline=-3]
\protect\draw (12 cm,0) -- (10 cm,0);
\protect\draw (10 cm,0) -- (8 cm,0);
\protect\draw (8 cm,0) -- (6 cm,0);
\protect\draw (6 cm,0) -- (4 cm,0);
\protect\draw (4 cm,0) -- (2 cm,0);
\protect\draw (2 cm,0) -- (0 cm,0);
\protect\draw (4 cm,0) -- (4 cm,1.5 cm);

\protect\draw[fill=white] (12 cm, 0 cm) circle (.15cm) node[below=1pt]{\scriptsize $8$};
\protect\draw[fill=white] (10 cm, 0 cm) circle (.15cm) node[below=1pt]{\scriptsize $7$};
\protect\draw[fill=white] (8 cm, 0 cm) circle (.15cm) node[below=1pt]{\scriptsize $6$};
\protect\draw[fill=white] (6 cm, 0 cm) circle (.15cm) node[below=1pt]{\scriptsize $5$};
\protect\draw[fill=white] (4 cm, 0 cm) circle (.15cm) node[below=1pt]{\scriptsize $4$};
\protect\draw[fill=white] (2 cm, 0 cm) circle (.15cm) node[below=1pt]{\scriptsize $3$};
\protect\draw[fill=white] (4 cm, 1.5 cm) circle (.15cm) node[right=1pt]{\scriptsize $2$};
\protect\draw[fill=white] (0 cm, 0 cm) circle (.15cm) node[below=1pt]{\scriptsize $1$};
\end{tikzpicture}. \]

The element $w_0$ is central, and we have $\ell(w_0)=120$. According to \cite[Table C.6]{GG}, we have $\mathbf{a}(\mathcal{J})=91$. Since $\mathcal{J}$ is strongly regular, i.e., all $\mathtt{H}$-cells inside $\mathcal{J}$ are singletons, we denote the unique element $\!^i\hc^j$ as $w_{ij}$. These elements can be given explicitly as $w_{ij} = s_{ij} \cdot w_0$, where $s_{ij} \in W$ denotes the product $i \cdots j$ of simple reflections along the unique shortest path starting in $i$ and ending in $j$ in the Dynkin diagram. The Bruhat graph of $\mathcal{J}$ is given on Figure \ref{fig:Bruhat_J_E_8}. The rows are right cells, and the columns are left cells.  We will simplify the notation and use $p_{ij} := p_{e,w_{ij}}$ for $i,j \in S$.

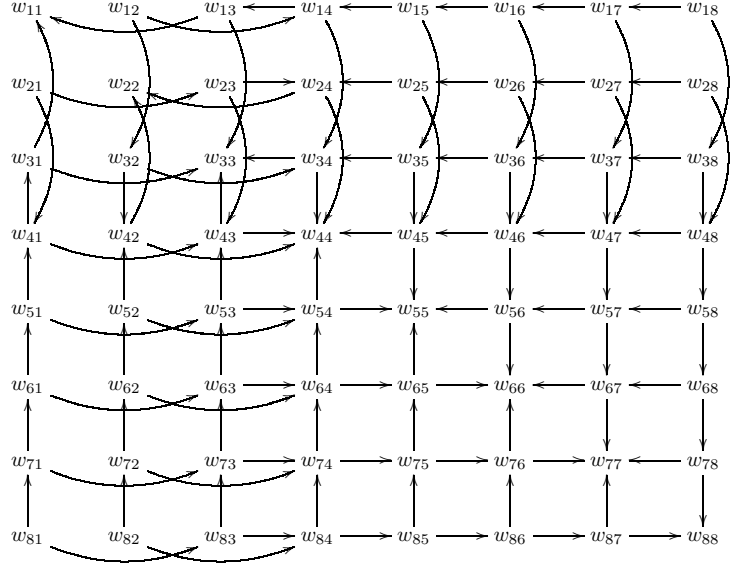
\begin{figure}
    \[ \scalebox{.8}{\xymatrix{
w_{11}  & w_{12} \ar@/_1pc/[rr] \ar@/^1pc/[dd] & w_{13} \ar@/^1pc/[ll] \ar@/^1pc/[dd] & w_{14} \ar[l] \ar@/^1pc/[dd] & w_{15} \ar[l] \ar@/^1pc/[dd] & w_{16} \ar[l] \ar@/^1pc/[dd] & w_{17} \ar[l] \ar@/^1pc/[dd] & w_{18} \ar[l] \ar@/^1pc/[dd] \\
w_{21} \ar@/_1pc/[rr] \ar@/^1pc/[dd] & w_{22}  & w_{23} \ar[r] \ar@/^1pc/[dd] & w_{24} \ar@/^1pc/[ll] \ar@/^1pc/[dd] & w_{25} \ar[l] \ar@/^1pc/[dd] & w_{26} \ar[l] \ar@/^1pc/[dd] & w_{27} \ar[l] \ar@/^1pc/[dd] & w_{28} \ar[l] \ar@/^1pc/[dd] \\
w_{31} \ar@/_1pc/[uu] \ar@/_1pc/[rr] & w_{32} \ar@/_1pc/[rr] \ar[d] & w_{33}  & w_{34} \ar[l] \ar[d] & w_{35} \ar[l] \ar[d] & w_{36} \ar[l] \ar[d] & w_{37} \ar[l] \ar[d] & w_{38} \ar[l] \ar[d] \\
w_{41} \ar[u] \ar@/_1pc/[rr] & w_{42} \ar@/_1pc/[uu] \ar@/_1pc/[rr] & w_{43} \ar[u] \ar[r] & w_{44}  & w_{45} \ar[l] \ar[d] & w_{46} \ar[l] \ar[d] & w_{47} \ar[l] \ar[d] & w_{48} \ar[l] \ar[d] \\
w_{51} \ar[u] \ar@/_1pc/[rr] & w_{52} \ar[u] \ar@/_1pc/[rr] & w_{53} \ar[u] \ar[r] & w_{54} \ar[u] \ar[r] & w_{55}  & w_{56} \ar[l] \ar[d] & w_{57} \ar[l] \ar[d] & w_{58} \ar[l] \ar[d] \\
w_{61} \ar[u] \ar@/_1pc/[rr] & w_{62} \ar[u] \ar@/_1pc/[rr] & w_{63} \ar[u] \ar[r] & w_{64} \ar[u] \ar[r] & w_{65} \ar[u] \ar[r] & w_{66}  & w_{67} \ar[l] \ar[d] & w_{68} \ar[l] \ar[d] \\
w_{71} \ar[u] \ar@/_1pc/[rr] & w_{72} \ar[u] \ar@/_1pc/[rr] & w_{73} \ar[u] \ar[r] & w_{74} \ar[u] \ar[r] & w_{75} \ar[u] \ar[r] & w_{76} \ar[u] \ar[r] & w_{77}  & w_{78} \ar[l] \ar[d] \\
w_{81} \ar[u] \ar@/_1pc/[rr] & w_{82} \ar[u] \ar@/_1pc/[rr] & w_{83} \ar[u] \ar[r] & w_{84} \ar[u] \ar[r] & w_{85} \ar[u] \ar[r] & w_{86} \ar[u] \ar[r] & w_{87} \ar[u] \ar[r] & w_{88} 
}} \]
    \caption{Bruhat graph of the penultimate two-sided cell in type $E_8$.}
    \label{fig:Bruhat_J_E_8}
\end{figure}

\begin{proposition}
The polynomials $p_{ij}$ are given in Table \ref{tab:KLE8}.
\begin{table}[ht]
    \centering
\renewcommand{\arraystretch}{1.4}
\scalebox{.744}{
\begin{tabular}{|l||p{5cm}|p{5cm}|p{5cm}|p{5cm}|} \hline
 & $1$ & $2$ & $3$ & $4$ \\ \hline\hline
$1$ & $v^{119} + v^{113} + v^{109} + v^{107} + v^{103} + v^{101} + v^{97} + v^{91}$ & $v^{116} + v^{112} + v^{110} + v^{108} + v^{106} + v^{104} + v^{102} + v^{100} + v^{98} + v^{94}$ & $v^{118} + v^{114} + v^{112} + v^{110} + 2v^{108} + v^{106} + v^{104} + 2v^{102} + v^{100} + v^{98} + v^{96} + v^{92}$ & $v^{117} + v^{115} + v^{113} + 2v^{111} + 2v^{109} + 2v^{107} + 2v^{105} + 2v^{103} + 2v^{101} + 2v^{99} + v^{97} + v^{95} + v^{93}$ \\ \hline
$2$ & $v^{116} + v^{112} + v^{110} + v^{108} + v^{106} + v^{104} + v^{102} + v^{100} + v^{98} + v^{94}$ & $v^{119} + v^{115} + v^{113} + v^{111} + 2v^{109} + v^{107} + 2v^{105} + v^{103} + 2v^{101} + v^{99} + v^{97} + v^{95} + v^{91}$ & $v^{117} + v^{115} + v^{113} + 2v^{111} + 2v^{109} + 2v^{107} + 2v^{105} + 2v^{103} + 2v^{101} + 2v^{99} + v^{97} + v^{95} + v^{93}$ & $v^{118} + v^{116} + 2v^{114} + 2v^{112} + 3v^{110} + 3v^{108} + 3v^{106} + 3v^{104} + 3v^{102} + 3v^{100} + 2v^{98} + 2v^{96} + v^{94} + v^{92}$ \\ \hline
$3$ & $v^{118} + v^{114} + v^{112} + v^{110} + 2v^{108} + v^{106} + v^{104} + 2v^{102} + v^{100} + v^{98} + v^{96} + v^{92}$ & $v^{117} + v^{115} + v^{113} + 2v^{111} + 2v^{109} + 2v^{107} + 2v^{105} + 2v^{103} + 2v^{101} + 2v^{99} + v^{97} + v^{95} + v^{93}$ & $v^{119} + v^{117} + v^{115} + 2v^{113} + 2v^{111} + 3v^{109} + 3v^{107} + 2v^{105} + 3v^{103} + 3v^{101} + 2v^{99} + 2v^{97} + v^{95} + v^{93} + v^{91}$ & $v^{118} + 2v^{116} + 2v^{114} + 3v^{112} + 4v^{110} + 4v^{108} + 4v^{106} + 4v^{104} + 4v^{102} + 4v^{100} + 3v^{98} + 2v^{96} + 2v^{94} + v^{92}$ \\ \hline
$4$ & $v^{117} + v^{115} + v^{113} + 2v^{111} + 2v^{109} + 2v^{107} + 2v^{105} + 2v^{103} + 2v^{101} + 2v^{99} + v^{97} + v^{95} + v^{93}$ & $v^{118} + v^{116} + 2v^{114} + 2v^{112} + 3v^{110} + 3v^{108} + 3v^{106} + 3v^{104} + 3v^{102} + 3v^{100} + 2v^{98} + 2v^{96} + v^{94} + v^{92}$ & $v^{118} + 2v^{116} + 2v^{114} + 3v^{112} + 4v^{110} + 4v^{108} + 4v^{106} + 4v^{104} + 4v^{102} + 4v^{100} + 3v^{98} + 2v^{96} + 2v^{94} + v^{92}$ & $v^{119} + 2v^{117} + 3v^{115} + 4v^{113} + 5v^{111} + 6v^{109} + 6v^{107} + 6v^{105} + 6v^{103} + 6v^{101} + 5v^{99} + 4v^{97} + 3v^{95} + 2v^{93} + v^{91}$ \\ \hline
$5$ & $v^{116} + v^{114} + v^{112} + 2v^{110} + v^{108} + 2v^{106} + 2v^{104} + v^{102} + 2v^{100} + v^{98} + v^{96} + v^{94}$ & $v^{117} + v^{115} + 2v^{113} + 2v^{111} + 2v^{109} + 3v^{107} + 2v^{105} + 3v^{103} + 2v^{101} + 2v^{99} + 2v^{97} + v^{95} + v^{93}$ & $v^{117} + 2v^{115} + 2v^{113} + 3v^{111} + 3v^{109} + 3v^{107} + 4v^{105} + 3v^{103} + 3v^{101} + 3v^{99} + 2v^{97} + 2v^{95} + v^{93}$ & $v^{118} + 2v^{116} + 3v^{114} + 4v^{112} + 4v^{110} + 5v^{108} + 5v^{106} + 5v^{104} + 5v^{102} + 4v^{100} + 4v^{98} + 3v^{96} + 2v^{94} + v^{92}$ \\ \hline
$6$ & $v^{115} + v^{113} + v^{111} + v^{109} + v^{107} + 2v^{105} + v^{103} + v^{101} + v^{99} + v^{97} + v^{95}$ & $v^{116} + v^{114} + 2v^{112} + v^{110} + 2v^{108} + 2v^{106} + 2v^{104} + 2v^{102} + v^{100} + 2v^{98} + v^{96} + v^{94}$ & $v^{116} + 2v^{114} + 2v^{112} + 2v^{110} + 2v^{108} + 3v^{106} + 3v^{104} + 2v^{102} + 2v^{100} + 2v^{98} + 2v^{96} + v^{94}$ & $v^{117} + 2v^{115} + 3v^{113} + 3v^{111} + 3v^{109} + 4v^{107} + 4v^{105} + 4v^{103} + 3v^{101} + 3v^{99} + 3v^{97} + 2v^{95} + v^{93}$ \\ \hline
$7$ & $v^{114} + v^{112} + v^{108} + v^{106} + v^{104} + v^{102} + v^{98} + v^{96}$ & $v^{115} + v^{113} + v^{111} + v^{109} + v^{107} + 2v^{105} + v^{103} + v^{101} + v^{99} + v^{97} + v^{95}$ & $v^{115} + 2v^{113} + v^{111} + v^{109} + 2v^{107} + 2v^{105} + 2v^{103} + v^{101} + v^{99} + 2v^{97} + v^{95}$ & $v^{116} + 2v^{114} + 2v^{112} + 2v^{110} + 2v^{108} + 3v^{106} + 3v^{104} + 2v^{102} + 2v^{100} + 2v^{98} + 2v^{96} + v^{94}$ \\ \hline
$8$ & $v^{113} + v^{107} + v^{103} + v^{97}$ & $v^{114} + v^{110} + v^{106} + v^{104} + v^{100} + v^{96}$ & $v^{114} + v^{112} + v^{108} + v^{106} + v^{104} + v^{102} + v^{98} + v^{96}$ & $v^{115} + v^{113} + v^{111} + v^{109} + v^{107} + 2v^{105} + v^{103} + v^{101} + v^{99} + v^{97} + v^{95}$ \\ \hline \hline
 & $5$ & $6$ & $7$ & $8$ \\ \hline\hline
$1$ & $v^{116} + v^{114} + v^{112} + 2v^{110} + v^{108} + 2v^{106} + 2v^{104} + v^{102} + 2v^{100} + v^{98} + v^{96} + v^{94}$ & $v^{115} + v^{113} + v^{111} + v^{109} + v^{107} + 2v^{105} + v^{103} + v^{101} + v^{99} + v^{97} + v^{95}$ & $v^{114} + v^{112} + v^{108} + v^{106} + v^{104} + v^{102} + v^{98} + v^{96}$ & $v^{113} + v^{107} + v^{103} + v^{97}$ \\ \hline
$2$ & $v^{117} + v^{115} + 2v^{113} + 2v^{111} + 2v^{109} + 3v^{107} + 2v^{105} + 3v^{103} + 2v^{101} + 2v^{99} + 2v^{97} + v^{95} + v^{93}$ & $v^{116} + v^{114} + 2v^{112} + v^{110} + 2v^{108} + 2v^{106} + 2v^{104} + 2v^{102} + v^{100} + 2v^{98} + v^{96} + v^{94}$ & $v^{115} + v^{113} + v^{111} + v^{109} + v^{107} + 2v^{105} + v^{103} + v^{101} + v^{99} + v^{97} + v^{95}$ & $v^{114} + v^{110} + v^{106} + v^{104} + v^{100} + v^{96}$ \\ \hline
$3$ & $v^{117} + 2v^{115} + 2v^{113} + 3v^{111} + 3v^{109} + 3v^{107} + 4v^{105} + 3v^{103} + 3v^{101} + 3v^{99} + 2v^{97} + 2v^{95} + v^{93}$ & $v^{116} + 2v^{114} + 2v^{112} + 2v^{110} + 2v^{108} + 3v^{106} + 3v^{104} + 2v^{102} + 2v^{100} + 2v^{98} + 2v^{96} + v^{94}$ & $v^{115} + 2v^{113} + v^{111} + v^{109} + 2v^{107} + 2v^{105} + 2v^{103} + v^{101} + v^{99} + 2v^{97} + v^{95}$ & $v^{114} + v^{112} + v^{108} + v^{106} + v^{104} + v^{102} + v^{98} + v^{96}$ \\ \hline
$4$ & $v^{118} + 2v^{116} + 3v^{114} + 4v^{112} + 4v^{110} + 5v^{108} + 5v^{106} + 5v^{104} + 5v^{102} + 4v^{100} + 4v^{98} + 3v^{96} + 2v^{94} + v^{92}$ & $v^{117} + 2v^{115} + 3v^{113} + 3v^{111} + 3v^{109} + 4v^{107} + 4v^{105} + 4v^{103} + 3v^{101} + 3v^{99} + 3v^{97} + 2v^{95} + v^{93}$ & $v^{116} + 2v^{114} + 2v^{112} + 2v^{110} + 2v^{108} + 3v^{106} + 3v^{104} + 2v^{102} + 2v^{100} + 2v^{98} + 2v^{96} + v^{94}$ & $v^{115} + v^{113} + v^{111} + v^{109} + v^{107} + 2v^{105} + v^{103} + v^{101} + v^{99} + v^{97} + v^{95}$ \\ \hline
$5$ & $v^{119} + v^{117} + 2v^{115} + 3v^{113} + 3v^{111} + 4v^{109} + 4v^{107} + 4v^{105} + 4v^{103} + 4v^{101} + 3v^{99} + 3v^{97} + 2v^{95} + v^{93} + v^{91}$ & $v^{118} + v^{116} + 2v^{114} + 2v^{112} + 3v^{110} + 3v^{108} + 3v^{106} + 3v^{104} + 3v^{102} + 3v^{100} + 2v^{98} + 2v^{96} + v^{94} + v^{92}$ & $v^{117} + v^{115} + v^{113} + 2v^{111} + 2v^{109} + 2v^{107} + 2v^{105} + 2v^{103} + 2v^{101} + 2v^{99} + v^{97} + v^{95} + v^{93}$ & $v^{116} + v^{112} + v^{110} + v^{108} + v^{106} + v^{104} + v^{102} + v^{100} + v^{98} + v^{94}$ \\ \hline
$6$ & $v^{118} + v^{116} + 2v^{114} + 2v^{112} + 3v^{110} + 3v^{108} + 3v^{106} + 3v^{104} + 3v^{102} + 3v^{100} + 2v^{98} + 2v^{96} + v^{94} + v^{92}$ & $v^{119} + v^{117} + v^{115} + v^{113} + 2v^{111} + 3v^{109} + 2v^{107} + 2v^{105} + 2v^{103} + 3v^{101} + 2v^{99} + v^{97} + v^{95} + v^{93} + v^{91}$ & $v^{118} + v^{116} + v^{112} + 2v^{110} + 2v^{108} + v^{106} + v^{104} + 2v^{102} + 2v^{100} + v^{98} + v^{94} + v^{92}$ & $v^{117} + v^{111} + v^{109} + v^{107} + v^{103} + v^{101} + v^{99} + v^{93}$ \\ \hline
$7$ & $v^{117} + v^{115} + v^{113} + 2v^{111} + 2v^{109} + 2v^{107} + 2v^{105} + 2v^{103} + 2v^{101} + 2v^{99} + v^{97} + v^{95} + v^{93}$ & $v^{118} + v^{116} + v^{112} + 2v^{110} + 2v^{108} + v^{106} + v^{104} + 2v^{102} + 2v^{100} + v^{98} + v^{94} + v^{92}$ & $v^{119} + v^{117} + v^{111} + 2v^{109} + v^{107} + v^{103} + 2v^{101} + v^{99} + v^{93} + v^{91}$ & $v^{118} + v^{110} + v^{108} + v^{102} + v^{100} + v^{92}$ \\ \hline
$8$ & $v^{116} + v^{112} + v^{110} + v^{108} + v^{106} + v^{104} + v^{102} + v^{100} + v^{98} + v^{94}$ & $v^{117} + v^{111} + v^{109} + v^{107} + v^{103} + v^{101} + v^{99} + v^{93}$ & $v^{118} + v^{110} + v^{108} + v^{102} + v^{100} + v^{92}$ & $v^{119} + v^{109} + v^{101} + v^{91}$ \\ \hline
\end{tabular}}
\vskip 10pt
    \caption{Kazhdan-Lusztig polynomials $p_{e,w}$ for $w \in \mathcal{J}$ in $E_8$.}
    \label{tab:KLE8}
\end{table}
\end{proposition}
\begin{proof}
By applying Lemma \ref{prop2p} to non-diagonal elements of $\jc$, one can see that all $p_{ij}$'s can be reconstructed from $p_{18}$ (similarly to the proofs from Subsection \ref{ssKLD}), see Figure \ref{fig:Bruhat_J_E_8}. In fact, it is convenient to use the computer to  calculate each $p_{ij}$ from $p_{18}$. So it is enough to prove that $p_{18} = v^{113} + v^{107} + v^{103} + v^{97}$.

Using Lemma \ref{prop2p} along the vertical arrows in the last column in Figure \ref{fig:Bruhat_J_E_8}, we can get
\begin{align}
\label{align:p78}       & p_{78} = \frac{v^6 + v^4 -1 + v^{-4} + v^{-6}}{v + v^{-1}} \cdot p_{18} , \\
\label{align:p88}        &  p_{88}=  (v^6 -1 + v^{-6})  \cdot p_{18} .
\end{align}
Applying again Lemma \ref{prop2p} to $w_{88} \to w_0$, we get
\begin{align}
\label{align:p888}    & v\cdot p_{88} + p_{8,w_{88}} = v^{120} + p_{78}.
\end{align}

From (\ref{eq:p_Delta}) and the fact that $\Delta_8\langle -1 \rangle\subset \Delta_e$, it follows that  $p_{88} - v\cdot p_{8,w_{88}}$ has non-negative coefficients. From (\ref{align:p78}), (\ref{align:p88}) and (\ref{align:p888}) we get
\begin{equation}
\label{eq:non_neg_coeff}
    p_{88} - v\cdot p_{8,w_{88}} = \frac{v^{16} + v^{14} - v^{10} - v^8 -v^6 + v^2 +1}{v^6(v^2+1)} \cdot p_{18} - v^{121}.
\end{equation}
Recall that $w_{88}$ is a Duflo element, so the lowest degree monomial appearing in $p_{88}$ is $v^{\mathbf{a}(\jc)}=v^{91}$, and moreover, it appears with coefficient $1$. Therefore, from (\ref{align:p88}) we see that we can write
\begin{equation*}
    p_{18} = v^{113} + a_1 v^{111} + a_2 v^{109} + a_3 v^{107} + a_4 v^{105} + a_5 v^{103} + a_6 v^{101} + a_7 v^{99} + v^{97}, \qquad a_i \geq 0.
\end{equation*}

From (\ref{eq:non_neg_coeff}) we see that $p_{18}$ must be divisible by $v^2+1$, and therefore
\begin{equation*}
%\label{eq:alt_sum}
    a_1 - a_2 + a_3 - a_4 + a_5 - a_6 +a_7 = 2 .
\end{equation*}
Write also
\begin{equation}
\label{eq:b}
    \frac{p_{18}}{v^6(v^2+1)} = v^{105} + b_1 v^{103} + b_2 v^{101} + b_3 v^{99} + b_4 v^{97} + b_5 v^{95} + b_6 v^{93} + v^{91} , 
\end{equation}
where
\begin{equation}
\label{eq:ab}
    \begin{split}
        &a_1 = 1 + b_1, \\
        &a_i = b_{i-1} + b_i, \quad i=2,3,\ldots, 6, \\
        &a_7 = b_6 +1.
    \end{split}
\end{equation}

By multiplying (\ref{eq:b}) with the numerator in (\ref{eq:non_neg_coeff}), and from the non-negativity of the coefficients in (\ref{eq:non_neg_coeff}), we get the following conditions on $b_i$'s:
\begin{align*}
      &b_2 + b_3 -1 \geq 0 &     &b_1 - b_3 - b_4 - b_5 + 2 \geq 0 \\
   -  &b_1 + b_3 + b_4 - 1 \geq 0 &     &b_1 + b_2 - b_4 - b_5 - b_6 \geq 0 \\
    - &b_1 - b_2 + b_4 + b_5 - 1  \geq 0 &      &b_2 + b_3 - b_5 - b_6 - 1 \geq 0 \\
    - &b_1 - b_2 - b_3 + b_5 + b_6 \geq 0 &      &b_3 + b_4 - b_6 - 1 \geq 0 \\
    - &b_2 - b_3 - b_4 + b_6 + 2 \geq 0 &      &b_4 + b_5 - 1 \geq 0 .
\end{align*}

Under (\ref{eq:ab}), these conditions translate to the following conditions on $a_i$'s:
\begin{align*}
      &a_3 \geq 1  &     &a_2 + 2 \geq a_3 + a_5 \\
      &a_4 \geq a_1 &     &a_2 + 1 \geq a_5 + a_7 \\
      &a_5  \geq a_2+1 &      &a_3 \geq a_6 + 1 \\
      &a_6+1 \geq a_1 + a_3 &      &a_4 \geq a_7 \\
      &a_1 + a_7 \geq a_2 + a_4 &      &a_5 \geq 1 .
\end{align*}
The fact the latter set of conditions implies $a_3 = a_5 = 1$ and $a_1=a_2=a_4=a_6=a_7=0$ is left as an exercise. This finishes the proof.
\end{proof}

%\todo{it might be better to present type E result as a plot with colors instead of tables, by assigning multiplicity a color.}

%\section{Appendix: $\BG{2}{2}$ in type $F_4$}\label{F4B22}
%\[\includegraphics[height = 16cm]{B_22.png}\]

\end{document}